\documentclass[10pt]{amsart}

\usepackage{graphicx}
\usepackage{hyperref}
\hypersetup{
    unicode=false,          
    pdftoolbar=true,        
    pdfmenubar=true,        
    pdffitwindow=false,     
    pdfstartview={FitH},    
    pdftitle={Large-scale sublinearly Lipschitz geometry of hyperbolic spaces},    
    pdfauthor={Gabriel Pallier},     
    pdfsubject={Metric Geometry},   
    pdfcreator={Gabriel Pallier},   
    pdfproducer={Universit{\'e} Paris-Sud, Departement de math{\'e}matiques}, 
    pdfkeywords={quasiconformal, quasiMoebius, sublinearly Lipschitz, hyperbolic spaces, hyperbolic groups}, 
    pdfnewwindow=true,      
    colorlinks=true,       
    linkcolor=blue,          
    citecolor=purple,        
    filecolor=magenta,      
    urlcolor=cyan           
}

\newtheorem{theorem}{Theorem}[section]
\newtheorem{lemma}[theorem]{Lemma}
\newtheorem{proposition}[theorem]{Proposition}

\newtheorem{theoremintro}{Theorem}

\newtheorem*{theorem*}{Theorem}
\newtheorem*{corollary*}{Corollary}
\newtheorem*{proposition*}{Proposition}

\theoremstyle{definition}
\newtheorem{definition}[theorem]{Definition}

\newtheorem*{definition*}{Definition}

\theoremstyle{remark}
\newtheorem{remark}[theorem]{Remark}

\numberwithin{equation}{section}

\usepackage{appendix}

\usepackage{amsmath,amssymb}
\usepackage{pgf,tikz}
\usetikzlibrary{arrows}

\newcommand{\R}[0]{\mathbf{R}}
\newcommand{\Z}[0]{\mathbf{Z}}
\newcommand{\gromprod}[2]{\left( #1 \mid #2 \right)} 

\usepackage{wasysym}
\usepackage{cleveref}
\usepackage{mathrsfs}

\begin{document}

\title{Large-scale Sublinearly Lipschitz geometry of hyperbolic spaces}

\author{Gabriel Pallier}

\date{\today}

\address{ 
Laboratoire de Math{\'e}matiques d'Orsay, Univ.\ Paris-Sud, CNRS, Universit{\'e} Paris-Saclay, 91405 Orsay, France.}
\email{gabriel.pallier@math.u-psud.fr}

\thanks{The author thanks the support of project ANR-15-CE40-0018.}




\begin{abstract}
Sublinearly Lipschitz maps have been introduced by Yves Cornulier in order to precisely state his theorems about asymptotic cones of Lie groups.
In particular, Sublinearly biLipschitz Equivalences (SBE) are a weak variant of quasiisometries, with the only requirement of still inducing biLipschitz maps at the level of asymptotic cones.
We focus here on hyperbolic metric spaces and study properties of boundary extensions of SBEs, reminiscent of quasiM{\"o}bius (or quasisymmetric) mappings. We give a dimensional invariant of the boundary that allows to distinguish hyperbolic symmetric spaces up to SBE, answering a question of Dru\c{t}u.
\end{abstract}

\maketitle

\renewcommand{\thesubsection}{\arabic{section}.\Alph{subsection}}

\section*{Introduction}

\makeatletter

\subsection{Main definitions and results}
Sublinearly Lipschitz maps between metric spaces have been gradually made into an object of study by Y. Cornulier in a series of papers starting in 2008 \cite{Cornulier_dimCone,Cornulier_Coneq,Cornulier_hypnil}.
Here is a short definition of a sublinearly biLipschitz equivalence (compare to Definition \ref{definition:sublinearly biLipschitz equivalence-maps-precise}):

\begin{definition}
\label{definition:sublinearly biLipschitz equivalence-intro}
Let $X$ and $Y$ be pointed metric spaces.
In $X$ and $Y$, denote the distances by $\vert \cdot - \cdot \vert$ and distances to the base-point by $\vert \cdot \vert$.
A map $f : X \to Y$ is called a sublinearly biLipschitz equivalence (SBE) if there exists a nondecreasing, doubling function $u : \R_{\geqslant 0} \to \R_{\geqslant 1}$ with $u(r) \ll r$ as $r \to + \infty$, and $(\underline{\lambda},\overline{\lambda}) \in \R_{>0}^2$ such that for any $x,x'$ in $X$ and $y$ in $Y$,
\begin{align}
\underline{\lambda} \vert x - x' \vert - u(\vert x \vert \vee \vert x' \vert) & \leqslant \vert f(x) - f(x') \vert \leqslant \overline{\lambda} \vert x - x' \vert + u(\vert x \vert \vee \vert x' \vert), \label{eq:SB_embedding} \\
\inf \left\{ \vert y - y' \vert : y' \in f(X) \right\} & \leqslant u(\vert y \vert), \label{eq:LS_surjectivity}
\end{align}
where $\vert x \vert \vee \vert x' \vert$ denotes $\sup \left\{ \vert x \vert, \vert x' \vert \right\}$.
\end{definition}

Note that while the function $u$ in the definition may depend up to an additive (or multiplicative, as $u$ takes values higher than $1$) error on base-points, the large-scale Lipschitz and reverse Lipschitz data $(\underline{\lambda},\overline{\lambda})$ do not.
The technical conditions on $u$ are required so that there is a well-behaved notion of $(\lambda, O(u))$-sublinearly biLipschitz equivalence (resp. $(\lambda, o(u))$-sublinearly biLipschitz equivalence) between nonpointed metric spaces; it is useful to retain only the class $O(u)$ or $o(u)$ for composition purposes, see Cornulier \cite[Proposition 2.2]{Cornulier_hypnil} and section \ref{section:back-sublinearly biLipschitz equivalence} below.
When $u=1$, $O(u)$-sublinearly biLipschitz equivalences are the more traditional quasiisometric maps.

Sublinearly Lipschitz maps were devised in the first place so that for any nonprincipal ultrafilter $\omega$ over $\Z_{\geqslant 0}$ or $\R_{\geqslant 0}$ and scaling sequence $(\lambda_j)$, $\mathrm{Con}_\omega(\cdot, \lambda_j)$ (with fixed basepoint) defines a functor from the large-scale sublinearly Lipschitz category to the Lipschitz category \cite[Proposition 2.9]{Cornulier_Coneq}.
The asymptotic cone characterization of hyperbolicity (Gromov \cite[2.A]{Gromov_AI}, Dru\c{t}u \cite[3.A.1.(iii)]{Drutu}) ensures that within the class of quasihomogeneous, geodesic metric spaces (such as finitely generated groups), hyperbolicity is preserved by sublinearly biLipschitz equivalences (see Cornulier \cite[Proposition 4.2]{Cornulier_hypnil}).
However, while asymptotic cones up to biLipschitz homeomorphisms are fine SBE invariants in order to distinguish, e.g., nilpotent groups, this is not the case in the hyperbolic setting, since all complete nonpositively curved Riemannian manifolds and nonelementary Gromov-hyperbolic group share the same asymptotic cones, namely the universal $2^{\aleph_0}$-branched $\R$-tree, even defined up to isometry (see for instance Erschler and Polterovich \cite[Theorem 1.1.3]{DyubinaPolterovich}).
This suggests to study the effects of SBEs on other asymptotic invariants instead.
In this direction, Cornulier proved that sublinearly biLipschitz equivalences induce biH{\"o}lder homeomorphisms between geodesic boundaries of proper geodesic hyperbolic metric spaces equipped with visual distances \cite[Theorem 1.7 and Theorem 4.3]{Cornulier_hypnil}.
Restated within the spaces, this says that for pairs of triples of far apart points sent to each other by a sublinearly biLipschitz equivalence, Gromov products in the source and target are within linear control of each other, a feature which may be derived from the large scale biLipschitz behavior.
Similarly to Gromov products, cross-differences, or positive logarithms of cross-ratios, have an incarnation as large distances within the space, so that one can hope that the same control remains between them, with a sublinear error term.
This is our main

\begin{theoremintro}[Restatement of Theorem \ref{th:boundary-maps-are-sqm}]
\label{mainthmintro}
Let $f : X \to Y$ be a $(\underline{\lambda}, \overline{\lambda}, O(u))$-sublinearly biLipschitz equivalence between hyperbolic proper geodesic spaces. Then $f$ induces a map $\varphi$ between the geodesic boundaries with the property that for all distinct $(\xi_1, \ldots \xi_4)$ on the geodesic boundary of $X$, all of them close enough,
\begin{equation}
\label{eq:sublin-qm-intro}
\underline{\lambda} \log^+ [\xi_i] - v \left(\overline{\boxtimes} \lbrace \xi_i \rbrace \right)
\leqslant \log^+[\varphi(\xi_i)] \leqslant \overline{{\lambda}} \log^+ [\xi_i] + v \left( \overline{\boxtimes}\lbrace \xi_i \rbrace \right),
\end{equation}
where $v = O(u)$ is a sublinear function, $\log^+(s) = \max(0, \log s)$ for all $s \in \R_{>0}$,
$\overline{\boxtimes} \lbrace \xi_i \rbrace$ denotes the supremum of all Gromov products over pairs in the four $\xi_i$'s, and the brackets $[\xi_i]$ denote the cross-ratios $[\xi_1, \ldots \xi_4]$ (see \ref{subsec:met-invariants} for definitions).
\end{theoremintro}

The homeomorphisms as in \eqref{eq:sublin-qm-intro} are given the name of sublinearly quasiM{\"o}bius (Definition \ref{def:subqM}).
A distinctive feature of sublinearly quasiM{\"o}bius homeomorphisms is that their distortion of the moduli of small annuli (or ``eccentricity" of small ellipsoids) is bounded at small, non-infinitesimal scale:

\begin{definition}
Let $\Xi$ be a metric space. An annulus $A$ of $\Xi$ is a difference of concentric balls $B(\xi, s) \setminus B(\xi, r)$ for some $\xi \in \Xi$ and $r,s \in \R_{>0}$ ; $\mathfrak{M} =  \log (s/r)$ is its modulus.
\end{definition}

\begin{proposition}
Let $\Xi$ and $\Psi$ be compact, uniformly perfect metric spaces and $\varphi : \Xi \to \Psi$ a $(\lambda^{-1}, \lambda, O(u))$-sublinearly quasiM{\"o}bius homeomorphism.
Let $A$ be an annulus of inner radius $r$, outer radius $R$ and modulus $\mathfrak{M}$.
There exists $w = O(u)$ such that if $R$ is sufficiently small, $\varphi(A)$ is contained in an annulus of modulus \[ \mathfrak{M}'= 2 \lambda \mathfrak{M} + w(-\log r). \]
\end{proposition}

When $u = 1$ this is a characterization of power-quasisymmetric mappings, compare Mackay and Tyson, \cite[Lemma 1.2.18]{MackayTyson}.
With their scale-sensitive moduli distortion, sublinearly quasiM{\"o}bius homeomorphisms may lack the analytic properties of quasisymmetric mappings, even between Euclidean spaces. Nevertheless we prove that they preserve the Hausdorff dimension of visual metrics in a favorable setting:

\begin{proposition}[Consequence of Proposition \ref{prop:almost-Carnot}]
Let $\Xi^\ast$ and $\Psi^\ast$ be punctured boundaries of purely real, normalized Heintze groups of Carnot type with homogeneous dimensions $p$ and $p'$ (see \ref{subsec:mm-boundary} for definitions). Assume there exists a homeomorphism $\varphi : \Xi^\ast \to \Psi^\ast$ which is sublinearly quasiM{\"o}bius over any compact subset (with respect to the visual metrics). Then $p = p'$.
\end{proposition}

The Heintze groups of Carnot type form an intermediate class between hyperbolic symmetric spaces and simply connected negatively curved homogeneous spaces.
The invariance of the topological dimension of the geodesic boundary is more generally granted by Cornulier's theorem on biH{\"o}lder continuity.
Once combined, those two asymptotic invariants allow to distinguish all hyperbolic symmetric spaces, answering a question of Dru\c{t}u \cite[Question 1.16 (2)]{Cornulier_hypnil}:

\begin{theoremintro}
\label{th:SBE-symm-sapce}
Let $X$ and $Y$ be rank one Riemannian symmetric spaces of noncompact type. If there exists a sublinearly biLipschitz equivalence between $X$ and $Y$, then $X$ and $Y$ are homothetic.
\end{theoremintro}

In view of the Cornulier-Tessera characterization of hyperbolic connected Lie groups \cite[Corollary 3]{CornulierTessera_hypLie}, this can be rephrased as: if there exists a large-scale sublinearly biLipschitz equivalence between two hyperbolic, non-amenable connected real Lie groups $G$ and $G'$, then $G$ and $G'$ are commable.

\subsection{Acknowledgements}
Yves Cornulier suggested to look for the main theorem, my advisor Pierre Pansu provided constant help and section \ref{sec:Riemannian} owes a lot to his ideas.
I benefited from a conversation with Peter Ha{\"i}ssinsky about Lipschitz-M{\"o}bius constants, explanations by Prof.\ Viktor Schroeder on cross-ratios, translation by Aliaksandr Minets on parts of \cite{EfremovichTihomirova}, and a key remark of S{\'e}bastien Gou{\"e}zel on the proof of the quantitative Morse lemma, see \ref{subsec:effective-morse-stab} herein. 

\tableofcontents

\section{Background}
\label{section:back-sublinearly biLipschitz equivalence}

\subsection{Large-scale Sublinearly Lipschitz maps}
\label{subsec:LS-SBE}
Here is a summary of Cornulier's definitions included for the reader's convenience. Call admissible any function $u : \R_{\geqslant 0} \to \R_{\geqslant 1}$ with the following properties:
\begin{enumerate}
\item{
\label{item:nondecreasing}
$u$ is nondecreasing}
\item{
\label{item:doubling}
$u$ is doubling: $\limsup_{r \to + \infty} u(2r)/u(r) < + \infty$}
\item{$u$ is sublinear: $u(r) \ll r$ as $r\to +\infty$.
\label{item:sublinear}}
\end{enumerate}
It is not really restrictive, and in fact useful in statements, to allow such a function to be only eventually defined and conditions \eqref{item:nondecreasing}, \eqref{item:doubling} to hold only on a neighborhood of $+\infty$ in $\R_{\geqslant 0}$.
However we will frequently work with a precise admissible function $u$ while keeping track on explicit bounds, and where they become valid. To facilitate this we introduce the following set of notations:
\begin{itemize}
\item{For all $\varepsilon > 0$, $r_{\varepsilon}(u)$ is $\sup \lbrace r \in \R_{\geqslant 0} : u(r) > \varepsilon r \rbrace$, or $0$ if this set is empty. This is finite by \eqref{item:sublinear}.}
\item{Properties \eqref{item:nondecreasing}, \eqref{item:doubling} and the fact that $\inf_r u(r) > 0$  ensure that for any $\tau >1$, $\sup_r u(\tau r)/ u(r)$ is finite.
We shall denote this number $u \uparrow \tau $.}
\end{itemize}

The following lemma is for our use only; it describes the way in which the constants $r_\varepsilon(u)$ and $u \uparrow \tau$ evolve when advancing function $u$.

\begin{lemma}
\label{lem:advancin-fun}
Let $u$ be an admissible function.
For any $p \in \R_{>0}$, define $u_p : \R_{\geqslant 0} \to \R_{\geqslant 1}$ as $u_p(t) = u(p+t)$. Then
\begin{enumerate}
\item for all $\tau \in \R_{> 1}$, 
$u_p \uparrow \tau \leqslant u \uparrow \tau$.
\label{item:vpL-smaller-than-vL}
\item{\label{item:r_epsilon(vp)}
For all $\varepsilon \in \R_{>0}$, if $p \geqslant r_{\varepsilon/2}(u)$ then
\[ r_{\varepsilon}(u_p) \leqslant  \frac{u \uparrow 2}{\varepsilon} u(p).  \]}
\end{enumerate}
\end{lemma}
\begin{proof}
Start with \eqref{item:vpL-smaller-than-vL}. By definition, $u$ is nondecreasing, hence
\[ u_p \uparrow \tau = \sup_r \frac{u(\tau r+p)}{u(r+p)} \leqslant \sup_r \frac{u(\tau t+\tau p)}{u(t+p)} = u \uparrow \tau. \]
As for \eqref{item:r_epsilon(vp)}, the hypothesis made on $p$ means that for all $p'$ greater than $p$, $u_p(p') \leqslant \frac{\varepsilon}{2}(p+p') \leqslant \varepsilon p'$, so $r_{\varepsilon}(u_p) \leqslant p$, which implies $\varepsilon r_\varepsilon(u_p) = u(p + r_{\varepsilon}(u_p)) \leqslant (u \uparrow 2) u(p)$, and then $\varepsilon r_{\varepsilon}(u_p) \leqslant (u \uparrow 2)u(p)$ so that $r_{\varepsilon}(u_p) \leqslant \varepsilon^{-1} (u \uparrow 2) u(p)$.
\end{proof}

In the following, let $u$ be an admissible function, and let $X$ and $Y$ be two pointed metric spaces. Recall that whenever $r$ and $s$ are real numbers, $r \vee s$ denotes $\sup \lbrace r, s \rbrace$ and $r \wedge s$ denotes $\inf \lbrace r,s \rbrace$.

\begin{definition}
A map $f : X \to Y$ is called $(\overline{\lambda}, O(u))$-Lipschitz if there exists $\overline{\lambda} \in \R_{>0}$ (called a large-scale Lipschitz constant) and a nondecreasing function $v = O(u)$ such that for all $(x_1,x_2) \in X^2$,
\[ \vert f(x_1) - f(x_2) \vert \leqslant \overline{\lambda} \vert x_1 - x_2 \vert + v(\vert x_1 \vert  \vee \vert x_2 \vert). \]
We may write that $f$ is $(\overline{\lambda}, v)$-Lipschitz to put emphasis on $v$, or on the contrary a $O(u)$-Lipschitz map if the actual Lipschitz constant and function $v$ are not relevant.
\end{definition}

\begin{definition}
$f,g : X \to Y$ are $O(u)$-close if $\vert f(x) - g(x) \vert = O(u(\vert x \vert)$.
\end{definition}

One checks that $O(u)$-Lipschitz maps can be composed (with a multiplicative effect on large-scale Lipschitz and expansion constants), in a way compatible with $O(u)$-closeness \cite[Proposition 2.2]{Cornulier_hypnil}, hence there is a well-defined category $\mathcal{L}_{O(u)}$ with metric spaces as objects\footnote{More precisely, at first, objects are  pointed metric spaces. Nevertheless the notion does not really depend on a given base-point.} and large-scale $O(u)$-Lipschitz maps modulo $O(u)$-closeness as morphisms.

\begin{definition}[compare Definition  \ref{definition:sublinearly biLipschitz equivalence-intro}]
\label{definition:sublinearly biLipschitz equivalence-maps-precise}
$f : X \to Y$ is a $O(u)$-Sublinearly Bilipschitz Equivalence (SBE) if the $O(u)$-closeness class of $f$ is an isomorphism in $\mathcal{L}_{O(u)}$. This can be metric-geometrically rephrased as follows \cite[Proposition 2.4]{Cornulier_hypnil}:
\begin{enumerate}
\item{$f$ is $O(u)$-Lipschitz;
\label{item:large-scale-Ou-Lipschitz}
}
\item{\label{item:large-scale-Ou-expansive}$f$ is $O(u)$-expansive : there exists a nondecreasing $v = O(u)$ and $\underline{\lambda} \in \R_{>0}$ such that
\[\forall(x_1, x_2) \in X^2, \, \vert f(x_1) - f(x_2) \vert \geqslant \underline{\lambda} \vert x_1 - x_2 \vert - v(\vert x \vert  \vee \vert x' \vert); \]
}
\item{$f$ is $O(u)$-surjective : for $y \in Y$, \label{item:Ou-surjectivity}
\[ d(y,f(X)) = O(u(\vert y \vert)). \]
}
\end{enumerate}
Conditions \eqref{item:large-scale-Ou-Lipschitz} and \eqref{item:large-scale-Ou-expansive} alone define the notion of a $O(u)$-Lipschitz embedding ; precisely a $(\underline{\lambda}, \overline{\lambda}, v)$-embedding is a map such that
\[\forall(x_1, x_2) \in X^2,\, \underline{\lambda} \vert x_1 - x_2 \vert - v(\vert x \vert  \vee \vert x' \vert) \leqslant \vert f(x_1) - f(x_2) \vert \leqslant \overline{\lambda} \vert x_1 - x_2 \vert + v(\vert x_1 \vert  \vee \vert x_2 \vert). \]
We will give an equivalent definition in subsection \ref{subsec:preliminaries-of-tracking}.
If there exists an admissible $u$ such that $f$ is a $O(u)$-sublinearly biLipschitz equivalence (resp.\ embedding), then $f$ is called a sublinearly biLipschitz equivalence (resp.\ embedding). In some occasion, we will abbreviate $(\underline{\lambda}, \overline{\lambda})$ into a single biLipschitz constant $\lambda = \sup \lbrace \overline{\lambda}, 1/\underline{\lambda})$ and call $f$ a $(\lambda, O(u))$-sublinearly biLipschitz equivalence.
\end{definition}

\subsection{Gromov products and Cornulier's estimates}
Let $X$ be a metric space. Recall that for $x_0, x_1, x_2 \in X$, the Gromov product of $x_2$ and $x_3$ seen from $x_0$ is by definition
$(x_2 \mid x_3)_{x_0} := \frac{1}{2} \left( \vert x_1 - x_0 \vert + \vert x_2 - x_0 \vert - \vert x_1 - x_2 \vert \right)$,
and that for all $\delta \in \R_{\geqslant 0}$, $X$ is $\delta$-hyperbolic (as defined by Gromov \cite[1.1.C]{Gromov_HG}) if
\begin{equation}
\label{eq:def-hyperbolicity}
\forall (x_0, x_1, x_2, x_3) \in X^4,\,
(x_1 \mid x_3)_{x_0} \geqslant \inf \left\{ (x_1 \mid x_2)_{x_0}, (x_2 \mid x_3)_{x_0} \right\} - \delta.
\end{equation}

If $X$ is $\delta$-hyperbolic and geodesic, then in addition, the Rips inequality is available: triangles in $X$ are $4 \delta$-slim, \cite[2.21]{GhysHarpe}.
A Cauchy-Gromov sequence in $X$ is a sequence $( x_n )_{n \in \Z_{\geqslant 0}}$ such that $(x_n \mid x_m) \to + \infty$ as $n,m \to + \infty$. Two Cauchy-Gromov sequences $\lbrace x_n \rbrace$, $\lbrace y_n \rbrace$ are equivalent, denoted $(x_n) \sim (y_n)$, if $(x_n \mid y_n) \to + \infty$ as $n \to + \infty$. This is an equivalence relation if $X$ is hyperbolic thanks to \eqref{eq:def-hyperbolicity}, and the Gromov boundary of $X$ is $\partial_\mathrm{G} X = \left\lbrace \text{Cauchy-Gromov} \; \text{sequences} \right\rbrace / \sim$. If $X$ is in addition proper and geodesic, this is also the visual boundary, or geodesic boundary that we will denote $\partial_\infty X$.
Though not stated by Cornulier in this form, the following  is given by the proof of his theorem \cite[4.3]{Cornulier_hypnil}.

\begin{proposition}
\label{prop:Gromov-boundary-map}
Let $u$ be an admissible function. Assume $X$ and $Y$ are hyperbolic, that $X$ is geodesic, and let $f : X \to Y$ be a $O(u)$-Lipschitz embedding. Then $f$ induces a (set-theoretic) boundary map $\partial_G f : \partial_G X \to \partial_G Y$. If $g$ is $O(u)$-close to $f$, then $\partial_G f = \partial_G g$. If $f$ is $O(u)$-surjective, then $\partial_G f$ is a bijection.
\end{proposition}

This can be expressed quantitatively; we restate below certain estimates ``before infinity" from Cornulier's proof. Whenever $\delta$ is a hyperbolicity constant, set a parameter
\begin{equation}
\mu = \begin{cases}
2^{1/\delta} & \delta > 0 \\
e & \delta = 0,
\label{eq:def-mu}
\end{cases}
\end{equation}
Fix a base-point $o \in X$ and define a kernel $\rho_\mu : X \times X \to \R_{\geqslant 0}$, $\rho_\mu(x,y) := \mu^{-(x \mid y)_o}$. The $\delta$-hyperbolicity inequality \eqref{eq:def-hyperbolicity} translates into a quasi-ultrametric inequality for $\rho_\mu$ : $\rho_\mu (x_0,x_2) \leqslant \mu^\delta \rho(x_0, x_1) \vee \rho_\mu(x_1, x_2)$ for all $(x_0, x_1, x_2) \in X^3$. This $\rho_\mu$ can be made subadditive by the chain construction:
\begin{equation}
\label{eq:chain-construction}
\check{\rho}_\mu(x,x'):= \inf \left\{ \sum_{i=1}^n \mu^{-(x_{i-1}\mid x_i)_o} : n \in \Z_{\geqslant 1}, x = x_0, \ldots x_n = x'\right\}.
\end{equation}

\begin{lemma}[Frink 1937, {\cite[Lemma 2]{Frink}}\footnote{see also Bourbaki \cite[IX.6, Proposition 2]{BourbakiTG}.}]
Let $\mathcal{X}$ be a set and $\varrho : \mathcal{X} \times \mathcal{X} \to \R_{\geqslant 0}$ be a $\R$-valued kernel on $\mathcal{X}$. Assume there is $K \in \R_{\geqslant 1}$ such that for all $(x_0, x_1, x_2) \in \mathcal{X}^3$,  $\rho(x_0,x_2) \leqslant K \rho(x_0,x_1) \vee \rho(x_1,x_2)$. Let $\check{\varrho}$ be associated to $\varrho$ by the chain construction \eqref{eq:chain-construction}. If $K \leqslant 2$, then $\check{\varrho} \leqslant \varrho \leqslant 4 \check{\varrho}$.
\end{lemma}

This allows the construction of the true distance $d_\mu = \check{\rho}_\mu$ from $\rho_\mu$ on the visual boundary. Already in the space, though, subadditivity of the kernel obtained by chain construction is useful and plays a key role in

\begin{theorem}[Cornulier]
\label{thm:Cornulier-Holder}
Let $v$ be an admissible function. Let $(\underline{\lambda}, \overline{\lambda})$ be large-scale expansion and Lispchitz constants.
Let $f : (X,o) \to (Y,o)$ be a large-scale $(\underline{\lambda}, \overline{\lambda}, v)$-sublinearly biLipschitz embedding. Assume there exists $\delta \in \R_{\geqslant 0}$ such that $X$ and $Y$ are $\delta$-hyperbolic and that $X$ is geodesic. For all $\alpha \in (0, \underline{\lambda})$ there exists a constant $M = M(\alpha, \delta) \in {\R}_{>0}$ and $R= R(\alpha, \lambda,v,\delta) \in {\R}_{>0}$ such that
for all $x,x' \in X$,
\begin{equation}
(x \mid x')_o \geqslant R(\lambda, v, \delta) \implies (f(x) \mid f(x'))_o \geqslant \alpha (x \mid x')_o - M(\alpha, \delta).
\label{ineq:general-Gromov-product-estimates-Cornulier}
\end{equation}
Especially, if $X$ and $Y$ are proper geodesic, then $\partial_G f = \partial_\infty f$ is $\alpha$-H{\"o}lder continuous for metrics $d_\mu$ on the boundaries, where $\mu$ is set as in \eqref{eq:def-mu}.
\end{theorem}

\begin{remark}
There is a dependence on $\mu$ in Cornulier's version which disappears in  \eqref{ineq:general-Gromov-product-estimates-Cornulier} because $\mu$ depends on $\delta$ according to convention \eqref{eq:def-mu}.
\end{remark}

A particular instance of theorem \ref{thm:Cornulier-Holder} occurs when the source space is $\R_{\geqslant 0}$ or $\Z_{\geqslant 0}$. For the latter, constants $R$ and $M$ can be explicitly extracted from the beginning of Cornulier's proof:
\begin{equation}
\forall s, t \in \Z_{\geqslant t_\alpha}, \left( \widetilde{\gamma}(s) \mid \widetilde{\gamma}(t) \right)_o \geqslant \alpha \inf \lbrace s,t \rbrace - \log_\mu \left( \frac{2}{1- \mu^{(\alpha + \underline{\lambda})/2}} \right),
\label{eq:cauchy-estimates-Ou-geodesic}
\end{equation}
where $\widetilde{\gamma}$ replaces $f$ of Lemma \ref{thm:Cornulier-Holder}, and 
\[ t_\alpha = \sup \left\{ s \in \Z_{\geqslant 0} : \vert \widetilde{\gamma}(0) \vert + v(s) \geqslant 4 (\underline{\lambda} - \alpha) s \right\}. \] 
This form will be of special interest in subsection \ref{subsec:tracking-rays}.

\subsection{Metric invariants of $4$ points at infinity}
\label{subsec:met-invariants}

Let $(Y,o)$ be a pointed, proper geodesic hyperbolic space, and let  $\partial^4_\infty Y$ denote the space of distinct $4$-tuples on $\partial_\infty Y$. For $(\eta_1, \eta_2, \eta_3, \eta_4) \in \partial_\infty^4 Y$, define
\begin{align*}
\overline{\boxtimes} \left\{ {\eta_1},{\eta_2},{\eta_3},{\eta_4} \right\} & := \sup \left\{ (\eta_i \mid \eta_j)_o : i \neq j \right\}, \; \text{and} \\
\underline{\boxtimes} \left\{ {\eta_1},{\eta_2},{\eta_3},{\eta_4} \right\} & := \inf \left\{ (\eta_i \mid \eta_j)_o : i \neq j \right\}.
\end{align*}
More generally, let $(\Xi, \varrho)$ be a metric space (to be thought of as a geodesic boundary with a visual distance) and let $(\xi_1, \ldots \xi_4)$ be distinct points in $\Xi$. Define their metric cross-ratio as
\[ [\xi_1, \xi_2, \xi_3, \xi_4]^\varrho = [\xi_i]^\varrho := \frac{{\varrho}(\xi_1, \xi_3){\varrho}(\xi_2, \xi_4)}{{\varrho}(\xi_1, \xi_4){\varrho}(\xi_2, \xi_3)}. \]
The superscript $\varrho$ might be omitted if sufficiently clear. Observe that if $\varrho$ has been obtained by the chain construction \eqref{eq:chain-construction} from a quasi-distance $\widehat{\varrho}$ on $\Xi$ such that
\[ \exists K \in [1,2), \, \forall (\xi_1, \xi_2, \xi_3) \in \Xi^3, \varrho(\xi_1, \xi_3) \leqslant K \sup \left\{ \varrho(\xi_1, \xi_2), \varrho(\xi_2, \xi_3) \right\}, \]
then by Frink's theorem $\varrho \leqslant \widehat{\varrho} \leqslant 4 {\varrho}$, and
\begin{equation}
\forall \nu \in \R_{>1}, \left\vert \log_\nu [\xi_i] - \log_\nu \frac{{\varrho}(\xi_1, \xi_3){\varrho}(\xi_2, \xi_4)}{{\varrho}(\xi_1, \xi_4){\varrho}(\xi_2, \xi_3)} \right\vert \leqslant \log_\nu 16.
\label{eq:approx-X-ratio-quasimetric}
\end{equation}
Especially, if $(\Xi, \varrho) = (\partial_\infty X, \check{\rho}_\nu)$ for a $\delta$-hyperbolic, proper geodesic, pointed space $(X,o)$ and a parameter $\nu \in (1, \mu(\delta)]$, then by \eqref{eq:approx-X-ratio-quasimetric}, $\log_\nu [\xi_i]^{d_\nu}$ depends on $\nu$ only up to an additive error: precisely for all $\nu, \nu' \in (1, \mu]$,
\begin{align*}
\left\vert \log_\nu[\xi_i]^{d_\nu} - \log_{\nu'}[\xi_i]^{d_{\nu'}} \right\vert & \leqslant  \left\vert \log_\nu[\xi_i]^{d_\nu} - (\xi_1, \xi_4)_o - (\xi_2, \xi_3)_o + (\xi_1, \xi_3)_o + (\xi_2, \xi_4)_o \right\vert \\
& + \left\vert \log_{\nu'}[\xi_i]^{d_{\nu'}} - (\xi_1, \xi_4)_o - (\xi_2, \xi_3)_o + (\xi_1, \xi_3)_o + (\xi_2, \xi_4)_o \right\vert \\
& \leqslant \log_{\nu} 16 + \log_{\nu'} 16.
\end{align*}
In the sequel we refer to $\log_\mu [\xi_i]^{d_\nu}$ as $\log [\xi_i]$, where $\mu$ follows convention \eqref{eq:def-mu}.
If nonnegative, this logarithm has a geometric interpretation:

\begin{proposition}
\label{prop:geometric-interpretation-of-X-ratio}
Let $(X,o)$ be a proper geodesic, $\delta$-hyperbolic space. There exists a constant $C = C(\delta)$ in $\R_{\geqslant 0}$ such that for all $(\xi_1, \ldots \xi_4) \in \partial^4 X$,
\[ d_X(\chi_{14}, \chi_{23}) - C \leqslant \log^+[\xi_i] \leqslant d_X(\chi_{14}, \chi_{23}) + C. \]
where $\xi_{ij}$ are geodesic lines between $\xi_i$ and $\xi_j$ (whose existence is provided by the visibility property of $X$, see Ghys-Harpe \cite[7.6]{GhysHarpe}).
\end{proposition}

Proposition \ref{prop:geometric-interpretation-of-X-ratio} seems well-known, yet we could not locate a proof in the literature, so we include one in subsection \ref{subsec:geom-x-ratio}.
It is better understood as a statement about cross-differences, see Buyalo and Schroeder \cite[4.1]{BuyaloSchroeder}.

\section{Preliminaries from hyperbolic metric geometry}

\subsection{A lemma on right-angled quadrilaterals}
Let $\delta \in \R_{\geqslant 0}$ be a constant, and
let $X$ be a geodesic $\delta$-hyperbolic metric space. 
We shall work under the following convention. In the course of a statement or a proof, as soon as a point or subspace of $X$ has a name, it is fixed until the end of the statement or proof even if its definition only allows to locate it in $X$ up to a few $\delta$s, and forthcoming objects will be attached to it. For instance, if a geodesic segment between two points has been previously defined, then the midpoint of these points will be understood as the midpoint of this geodesic segment.
Especially, if\footnote{We will abusively write $\gamma$ when referring to $\mathrm{im}(\gamma)$ when $\gamma$ is a (quasi)geodesic.} $\gamma \subset X$ is a geodesically convex subspace and $b \in X$ is a point, $p_\gamma(b)$ is an orthogonal projection (closest point) of $b$ on $\gamma$. This is well defined up to $16 \delta$, and $p_\gamma$ has a contracting behavior on distances expressed by the following

\begin{lemma}[Shchur, {\cite[Lemma 1]{Shchur2013}}\footnote{There is a $4\delta$ additive error term instead of our $16 \delta$ in Shchur's version, because Shchur defines a $\delta$-hyperbolic space via Rips inequality there.}]
\label{lem:contraction}
Let $\gamma$ be a geodesic, $b$ a point in $X$. Then for all $c \in \gamma$,
$\vert c - p_\gamma(b) \vert \leqslant \vert b - c \vert - \vert b - p_\gamma(b) \vert + 16 \delta$.
In particular, for all $b,b' \in X$,
\begin{equation}
\label{eq:contraction-shchur}
\vert p_\gamma (b) - p_\gamma(b') \vert \leqslant \vert b - b' \vert + 16 \delta.
\end{equation}
\end{lemma}

\begin{definition}
Let $\alpha \in \R_{\geqslant 0}$. Say that a metric space $P$ is $\alpha$-connected if for any $\alpha' \in \R_{>\alpha}$, the equivalence relation generated by
$[d(x,y) \leqslant \alpha']$ over $x,y \in P$ has a unique class.
\end{definition}

\begin{lemma}
\label{lem:metric-connectedness}
Let $\alpha> 0$ and $S \subset X$ a ${\alpha}$-connected subspace (for instance a quasigeodesic). Let $\gamma$ be a geodesic of $X$. Then any $p_\gamma(S)$ is $(\alpha + 16 \delta)$-connected. In particular if $S$ is a geodesic then $p_\gamma(S)$ is $16 \delta$-connected.
\end{lemma}

\begin{proof}
Let $S' = p_\gamma(S)$ and let $\alpha' \in \R_{>0}$ be such that $\alpha' > \alpha+ 16 \delta$. If there is $s'_1 = p_\gamma(s_1)$ such that $d(s_1', S' \setminus \lbrace s_1' \rbrace) \geqslant \alpha'$, then for all $s'_2 = p_\gamma(s_2)$, $\vert s'_1 - s_2' \vert > \alpha'$ implies with \eqref{eq:contraction-shchur}, that $s_1 - s_2 > \alpha'$. Thus $S$ is not $\alpha$-connected.
\end{proof}

\begin{definition}
Let $\eta \in \R_{\geqslant 0}$ be a spatial constant and let $X$ be a geodesic space. Say that an ordered list $x_1, \ldots x_r$ of points in $X$ with $r \geqslant 3$ is $\eta$-almost lined up if there exists a geodesic segment $\sigma$ such that for all $i$, $x_i$ lies in the $\eta$-neighborhood $\mathcal{N}_\eta(\sigma)$ of $\mathrm{im}(\sigma)$ and the $p_\sigma(x_i)$ are lined up in this order on $\sigma$.
\end{definition}

\begin{lemma}[Gromov product of almost lined up points]
\label{lem:grom-prod-of-almost-lined-up}
Let $\eta \in \R_{\geqslant 0}$ and assume  $x_1, x_2, x_3$ in a geodesic metric space $X$ are $\eta$-almost lined up ; then
\begin{equation}
\left\vert (x_2 \mid x_3)_{x_1} - \vert x_1 - x_2 \vert \right\vert \leqslant 5 \eta.
\end{equation}
\end{lemma}

\begin{proof}
Let $\sigma$ be a geodesic segment achieving the almost-lined upness assumption. For $i \in \lbrace 1, 2, 3 \rbrace$, let $y_i = p_\sigma(x_i)$. By hypothesis $\vert x_i - y_i \vert \leqslant \eta$, so by the triangle inequality $\vert \vert y_i - y_j \vert - \vert x_i - x_j \vert \vert \leqslant 2 \eta$; then by definition of the Gromov product $\left\vert \gromprod{x_2}{x_3}_{x_1}  - \gromprod{y_2}{y_3}_{y_1} \right\vert \leqslant 3 \eta$. Finally, $y_1$, $y_2$ and $y_3$ are lined up, hence $\gromprod{y_2}{y_3}_{y_1} = \vert y_1 - y_2 \vert$. Conclusion follows from the triangle inequality in $\R$.
\end{proof}

\begin{lemma}[Right-angled triangles degenerate]
\label{lem:triangle-rectangle-hyperbolique}
Let $\sigma$ be a geodesic of a geodesic hyperbolic space $X$, $b \in X$ and $a=p_\sigma(b)$ on $\sigma$. Let $c$ be a point of $\sigma$.
Then there exists $t \in [bc]$ such that
\begin{enumerate}
\item{$\vert a - t \vert \leqslant 28 \delta$
\label{item:t-close-a}
}
\item{$d(t, \sigma) \leqslant 4 \delta$ and $d(t,[ba]) \leqslant 4 \delta$. \label{item:t-close-ba}}
\item{for any $u$ in the subsegment $[tc]$ of $[bc]$, $d(u, \sigma) \leqslant 4\delta$.
\label{item:u-close-sigma}}
\end{enumerate}

\end{lemma}

In particular, if $\vert b - a \vert$, $\vert c - a \vert$ are large enough, then $b,a,c$ are $28 \delta$-almost lined up in this order.

\begin{proof}
Let $\triangle$ be the geodesic triangle $abc$ with sides $[ba]$, $[bc]$ and the subsegment $[ac]$ of $\sigma$. Set $\ell = \vert a - c\vert$ and assume $\alpha : [0, \ell] \to X$ parametrizes $[bc]$ so that $\alpha(0) = c$, $\alpha(\ell) = b$.
If $\sup \lbrace (\alpha (s), \sigma) :s \in [0, \ell] \rbrace \leqslant 4 \delta$, set $t=b$; then \eqref{item:u-close-sigma} and \eqref{item:t-close-ba} are automatically true, while $\vert a-t \vert = d(t, \sigma) \leqslant 4 \delta \leqslant 28 \delta$ so that also \eqref{item:t-close-a} is true. Otherwise, define
\[ t = \alpha(s),\, s = \inf \left\{ u \in [0,\ell], d(\alpha(u), \sigma) > 4 \delta \right\}.\]
As $\triangle$ is $4\delta$-slim, $d(t,[ba]) \leqslant 4 \delta$ while $d(t,\sigma) \leqslant 4 \delta$ also. Let $t_b$, resp.\ $t_c$ be an orthogonal projection of $t$ on $\sigma$, resp.\ on $[ba]$. By the triangle inequality, $\vert t_c - t_b \vert \leqslant 4 \delta + 4 \delta = 8 \delta$. Then $\vert t_b - b \vert \leqslant \vert t_c - b \vert + 8 \delta \leqslant \vert b- a \vert + 8 \delta$. By the contraction Lemma \ref{lem:contraction}, $\vert t_b - a \vert \leqslant 8 \delta + 16 \delta  = 24 \delta$. By the triangle inequality, $\vert t - a \vert \leqslant 24 \delta + 4 \delta = 28 \delta$.
\end{proof}

\begin{lemma}[Quadrilaterals with two consecutive right-angles degenerate]
\label{lem:quasialign}
Let $a_0, a_1, b_0, b_1$ be four points in $X$. For $i \in \lbrace 1, 2 \rbrace$, let $\gamma_i$ be a geodesic segment between $a_i$ and $b_i$. Assume that $138 \delta \leqslant \vert a_0 - a_1 \vert$, and that one of the following holds:
\begin{enumerate}
\item Either, $a_i = p_\sigma(b_i)$ for all $i \in \lbrace 1, 2 \rbrace$, or \label{case1quad}
\item $a_i = p_{\gamma_i} a_{1-i}$ for all $i \in \lbrace 1, 2 \rbrace$. \label{case2quad}
\end{enumerate}
Then for all $i \in \lbrace 0, 1 \rbrace$, $d(a_i,[b_0 b_1]) \leqslant 56 \delta$.
\end{lemma}

\begin{proof}
Let $\sigma$ be a geodesic segment between $a_1$ and $a_2$, and let $m$ be the midpoint of $\sigma$.
By Lemma \ref{lem:triangle-rectangle-hyperbolique}, there exists $t_0$ and $t_1$ on $[b_0 m]$ and $[b_1 m]$ respectively such that
\begin{equation}
\forall i \in \lbrace 1, 2 \rbrace,\; \vert a_i - t_i \vert  \leqslant 28 \delta
\label{eq:colle-plus-loin-que-t}
\end{equation}
Moreover, by \eqref{item:t-close-ba} and the triangle inequality,
\begin{equation}
\label{eq:proj-ti-close-to-ai}
\vert p_{\sigma}(t_i) - a_i \vert \leqslant \vert p_\sigma(t_i) - t_i \vert + \vert t_i - a_i \vert \leqslant 4 \delta + 28 \delta = 32 \delta.
\end{equation}
Thus $a_i$, $p_\sigma(t_i)$, $m$ and $a_{1-i}$ are lined up on $\sigma$ as below:
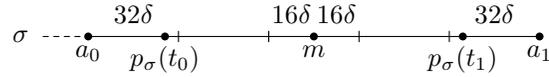
\begin{center}
\begin{tikzpicture}[line cap=round,line join=round,>=triangle 45,x=0.6cm,y=0.6cm]
\clip(-2,-1) rectangle (12,1);
\draw (0,0)-- (10,0);
\draw [dash pattern=on 2pt off 2pt] (-1,0)-- (0,0);
\draw[shift={(2,0)},color=black] (0pt,2pt) -- (0pt,-2pt);
\draw[shift={(8,0)},color=black] (0pt,2pt) -- (0pt,-2pt);
\draw (1,0.5) node {$32 \delta$}; \draw (9,0.5) node {$32 \delta$};
\draw (-1.5,0) node  {$\sigma$};
\fill [color=black] (0,0) circle (1.5pt) node[below] {$a_0$};
\fill [color=black] (5,0) circle (1.5pt) node[below] {$m$};
\fill [color=black] (10,0) circle (1.5pt) node[below] {$a_1$};
\fill [color=black] (1.7,0) circle (1.5pt) node[below] {$p_\sigma(t_0)$};
\fill [color=black] (8.3,0) circle (1.5pt) node[below] {$p_\sigma(t_1)$};
\draw[shift={(4,0)},color=black] (0pt,2pt) -- (0pt,-2pt);
\draw[shift={(6,0)},color=black] (0pt,2pt) -- (0pt,-2pt);
\draw (4.5,0.5) node {$16 \delta$}; \draw (5.5,0.5) node {$16 \delta$};
\end{tikzpicture}
\end{center}

Next, we proceed to prove that $t_i$ is far from $[mb_{1-i}]$. Note that since the triangles $ma_ib_i$ are slim, one need only show that $t_i$ is far from $[a_{1-i}b_{1-i}]$ and $[ma_{1-i}]$.

\begin{itemize}
\item In case \eqref{case1quad}, for all $a'_i \in \gamma_i$, since $p_{\sigma}(a'_i)=a_i$ and by \eqref{eq:proj-ti-close-to-ai} and Lemma \ref{lem:contraction},
\begin{align*}
\vert t_i - a'_{1-i} \vert \geqslant \vert p_{\sigma}(t_i) - a_{1-i} \vert - 16 \delta & \geqslant \vert a_i - a_{1-i} \vert - \vert p_{\sigma}(t_i) - a_i \vert - 16 \delta \\
& \geqslant 138 \delta - 48 \delta = 90 \delta,
\end{align*}
hence $d(t_i, \gamma_{1-i}) \geqslant 90 \delta$.
\item In case \eqref{case2quad}, as $a_{1-i} = p_{\gamma_{1-i}} a_i$,
$d(t_i, \gamma_{1-i}) \geqslant d(a_i, \gamma_{1-i}) - 28 \delta \geqslant 110 \delta$.
\item In both cases, $d(t_i, [ma_{1-i}]) \geqslant 79 \delta - 32 \delta = 45 \delta$.
\end{itemize}
Using the previous inequality together with the fact that the triangle $a_{1-i}mb_{1-i}$ is $4 \delta$-slim,
\begin{equation*}
d(t_i, [mb_{1-i}]) \geqslant 45 \delta - 4 \delta = 41 \delta > 4 \delta.
\end{equation*}
Finally, $b_1 m b_2$ is $4 \delta$-slim, hence $d(t_i, [b_1 b_2]) \leqslant 4 \delta$, and by the triangle inequality,
\[ d(a_i, [b_0 b_1]) \leqslant \vert a_i - t_i \vert + d(t_i, [b_0 b_1]) \leqslant 28 \delta + 4 \delta \leqslant 56 \delta. \qedhere \]
\end{proof}

\begin{figure}
\begin{center}
\begin{tikzpicture}[line cap=round,line join=round,>=triangle 45,x=0.5cm,y=0.5cm]
\clip(-2,-1) rectangle (12,6);
\draw (0,0) -- (10,0); \draw (0,0) -- (0,5);
\draw (10,0) -- (10,5);
\draw (0,0.3) -- (0.3,0.3); \draw (0.3,0) -- (0.3,0.3);
\draw (9.7,0) -- (9.7,0.3); \draw (10,0.3) -- (9.7,0.3); 
\draw plot[domain=0:5,variable=\t]({\t},{6/(\t+1)-1});
\draw plot[domain=0:5,variable=\t]({10-\t},{6/(\t+1)-1});
\draw plot[domain=0:5,variable=\t]({\t},{6/(\t+1)-1 + \t/5 + \t*(10-\t)/50});
\draw plot[domain=0:5,variable=\t]({10-\t},{6/(\t+1)-1 + (\t)/5 + \t*(10-\t)/50});
\fill (0,0) node[anchor=north east]{$a_0$} circle(1.5pt);
\fill (10,0) node[anchor=north west]{$a_1$} circle(1.5pt);
\fill (5,0) node[anchor=north]{$m$} circle(1.5pt);
\fill (0,5) node[anchor=south east]{$b_0$} circle(1.5pt);
\fill (10,5) node[anchor=south west]{$b_1$} circle(1.5pt);
\fill (1,2) node[anchor=east]{$t_0$} circle(1.5pt);
\fill (9,2) node[anchor=west]{$t_1$} circle(1.5pt);

\end{tikzpicture}
\caption{Main points occurring in the proof of Lemma \ref{lem:quasialign}.}
\end{center}
\end{figure}
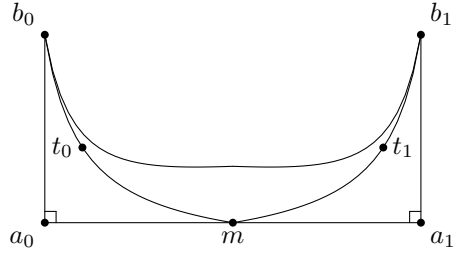

\subsection{An estimate on geodesic projections}
Let $X$ be as before a geodesic $\delta$-hyperbolic metric space, and fix a base-point $o \in X$.

\begin{lemma}
\label{lem:controle-projetes-geodesiques}
Let $\gamma, \gamma' : \R \to X$ be two geodesics; define $\xi_- = [\gamma]_{-\infty}$, $\xi_+ = [\gamma]_{+ \infty}$, $\xi'_- = [\gamma']_{-\infty}$ and $\xi'_+ = [\gamma']_{+ \infty}$ on the boundary at infinity $\partial_\infty X$ of $X$. Assume that the $\xi_\pm, \xi'_\pm$ are all distinct. Then
\begin{equation}\label{eq:controle-par-gromprod}
\sup \left\{ \vert p_\gamma(b) \vert : b \in \gamma' \right\} \leqslant \overline{\boxtimes} \lbrace \xi_-, \xi_+, \xi'_-, \xi'_+ \rbrace + 284 \delta,
\end{equation}
where we recall that $\overline{\boxtimes} \lbrace \xi_-, \xi_+, \xi'_-, \xi'_+ \rbrace$ is an abbreviation for $ \sup \gromprod{\xi_1}{\xi_2}_o$over distinct pairs $\lbrace \xi_1, \xi_2 \rbrace$ in $\lbrace \xi_\pm, \xi'_\pm \rbrace$.
\end{lemma}

\begin{proof}
Change if necessary the parametrizations of $\gamma$ and $\gamma'$ in such a way that $\gamma(0) = p_\gamma(o)$, $\gamma'(0) = p_{\gamma'}(o)$. Let $b \in \gamma'$.
\begin{itemize}
\item{Either $\vert p_\gamma(b) - \gamma(0) \vert < 138 \delta$; then by the triangle inequality, $\vert p_\gamma(b) \vert < \vert \gamma(0) \vert + 138 \delta$.
Let $s \in \R$. Since $X$ is $\delta$-hyperbolic,
\begin{equation}
(\gamma(s) \mid \gamma(-s))_o \geqslant \min \left\{ (\gamma(-s) \mid \gamma(0))_o, (\gamma(0) \mid \gamma(s))_o \right\} - \delta.
\label{eq:grom-prod-hyp-est-proj-geod}
\end{equation}
By Lemma \ref{lem:triangle-rectangle-hyperbolique}, when $s$ is large enough $o$, $\gamma(0)$ and $\gamma(s)$ (resp.\ $o$, $\gamma(0)$ and $\gamma(s)$) are $28 \delta$-almost lined up in this order, so by Lemma \ref{lem:grom-prod-of-almost-lined-up}, \eqref{eq:grom-prod-hyp-est-proj-geod} becomes
\[ (\gamma(s) \mid \gamma(-s))_o \geqslant \vert \gamma(0) \vert - 5 \cdot 28 \delta - \delta = \vert \gamma(0) \vert - 141 \delta. \]
Finally, $\vert p_\gamma(b) \vert < \vert \gamma(0) \vert + 138 \delta \leqslant (\gamma(s) \mid \gamma(-s))_o + 138 \delta + 141 \delta$. Letting $s \to + \infty$,
\begin{align*}
\vert p_\gamma(b) \vert \leqslant \liminf_{s \to + \infty} ( \gamma(s) \mid \gamma(-s))_o & \leqslant (\xi_- \mid \xi_+ )_o + 279 \delta \\ & \leqslant (\xi_- \mid \xi_+ )_o + 284 \delta.
\end{align*}
}
\item{Or $\vert p_\gamma(b) - \gamma(0) \vert \geqslant 138 \delta$ in which case Lemma \ref{lem:quasialign} applies so that $o$, $\gamma(0)$ and $p_\gamma(b), b$ are $56\delta$-almost lined up in this order.
Let $s, s' \in \R$ be such that $\inf \lbrace \vert s \vert, \vert s' \vert \rbrace \geqslant \sup \lbrace \vert p_\gamma(b) - \gamma(0) \vert, b - \gamma'(0) \vert \rbrace$.
Then
\begin{align*}
(\gamma(s) \mid \gamma'(s'))_o  \geqslant \min & \left\{
(\gamma(s) \mid \gamma(0))_o,  (\gamma(0), p_\gamma(b))_o, \right. \\
& \left. (p_\gamma(b) \mid b)_o, (b \mid \gamma'(0))_o, (\gamma'(0) \mid \gamma'(s'))_o
\right\} - 4 \delta.
\end{align*}
Applying repeatedly Lemma \ref{lem:grom-prod-of-almost-lined-up},
\begin{align*}
(\gamma(s) \mid \gamma'(s'))_o  \geqslant \min & \left\{
\vert \gamma(0) \vert - 140 \delta,  \vert \gamma(0) \vert - 140 \delta, \right. \\
& \left. \vert p_\gamma(b) \vert - 5 \cdot 56 \delta, \vert \gamma'(0) \vert - 140 \delta, \vert \gamma'(0) \vert - 140 \delta
\right\} - 4 \delta.
\end{align*}
Now letting $s, s' \to \pm \infty$,
\begin{align*}
\vert p_\gamma(b) \vert & \leqslant \overline{\boxtimes} \lbrace \xi_-, \xi'_-, \xi_+, \xi'_+ \rbrace + 5 \cdot 56 \delta + 4 \delta \\ & = \overline{\boxtimes} \lbrace \xi_-, \xi'_-, \xi_+, \xi'_+ \rbrace  + 284 \delta. \qedhere
\end{align*}
}
\end{itemize}
\end{proof}

\begin{figure}
\begin{center}
\begin{tikzpicture}[line cap=round,line join=round,>=triangle 45,x=0.7cm,y=0.7cm]
\clip(-4,-0.7) rectangle (10,7.5);
\draw[fill=black,fill opacity=0.1] (4.33,3.31) -- (4.25,3.49) -- (4.07,3.41) -- (4.15,3.23) -- cycle;
\draw[fill=black,fill opacity=0.1] (-0.52,1.95) -- (-0.49,2.15) -- (-0.69,2.18) -- (-0.72,1.98) -- cycle;
\draw[fill=black,fill opacity=0.1] (0.34,1.5) -- (0.45,1.66) -- (0.29,1.78) -- (0.18,1.62) -- cycle;
\draw [dash pattern=on 3pt off 3pt,domain=-4.22:10.63] plot(\x,{(-0-0*\x)/1});
\draw [shift={(-1,0)}] plot[domain=0:pi,variable=\t]({1*2*cos(\t r)+0*2*sin(\t r)},{0*2*cos(\t r)+1*2*sin(\t r)});
\draw [shift={(5.5,0)}] plot[domain=0:pi,variable=\t]({1*3.5*cos(\t r)+0*3.5*sin(\t r)},{0*3.5*cos(\t r)+1*3.5*sin(\t r)});
\draw [shift={(2.4,0)}] plot[domain=1.28:2.51,variable=\t]({1*2.75*cos(\t r)+0*2.75*sin(\t r)},{0*2.75*cos(\t r)+1*2.75*sin(\t r)});
\draw [shift={(-3.56,0)}] plot[domain=0.4:0.99,variable=\t]({1*8.35*cos(\t r)+0*8.35*sin(\t r)},{0*8.35*cos(\t r)+1*8.35*sin(\t r)});
\draw [shift={(13.25,0)}] plot[domain=2.62:3,variable=\t]({1*14.11*cos(\t r)+0*14.11*sin(\t r)},{0*14.11*cos(\t r)+1*14.11*sin(\t r)});
\draw (3.2,2.8) node[anchor=north west] {$b$};
\draw (4,3.3) node[anchor=north west] {$\gamma'(0)$};
\draw (-2,2.7) node[anchor=north west] {$\gamma(0)$};
\draw (-1,1.7) node[anchor=north west] {$p_\gamma(b)$};
\fill [color=black] (-3,0) circle (1.5pt) node[below] {$\xi_-$};
\fill [color=black] (1,0) circle (1.5pt) node[below] {$\xi_+$};
\fill [color=black] (2,0) circle (1.5pt) node[below] {$\xi'_-$};
\fill [color=black] (9,0) circle (1.5pt) node[below] {$\xi'_+$};
\fill [color=black] (3.2,2.64) circle (1.5pt);
\fill [color=black] (0.18,1.62) circle (1.5pt);
\fill [color=black] (1,7) circle (1.5pt) node[left] {$o$};
\fill [color=black] (4.15,3.23) circle (1.5pt);
\fill [color=black] (-0.72,1.98) circle (1.5pt);
\end{tikzpicture}
\caption{Configuration of Lemma \ref{lem:controle-projetes-geodesiques} in the half-plane model of $\mathbb{H}^2$.}
\end{center}
\end{figure}

\subsection{Quantitative Morse stability}
\label{subsec:effective-morse-stab}

\begin{lemma}[Morse stability for quasigeodesics]
\label{lem:quantitative-Morse}
Let $c, \delta \in \R_{\geqslant 0}$, $(\underline{\lambda}, \overline{\lambda}) \in \R_{>0}^2$ be constants. Let $X$ be a geodesic, $\delta$-hyperbolic metric space. Let $J = [a,b]$ be a closed bounded interval of $\R$ and let $\widetilde{\gamma} : J \to X$ be $(\underline{\lambda}, \overline{\lambda}, c)$ quasigeodesic, i.e.\
\[ \forall(s, t) \in J^2, \underline{\lambda}\vert s - t \vert - c \leqslant \left\vert \widetilde{\gamma}(s) - \widetilde{\gamma}(t) \right\vert \leqslant \overline{\lambda} \vert s - t \vert +c. \]
Recall that $\lambda = \sup \lbrace \overline{\lambda}, 1/\underline{\lambda})$, and assume that
$c \geqslant 6 \lambda^2 \delta$.
There exist functions $h, \widetilde{h} : \R_{> 0} \to \R_{> 0}$ such that if $\gamma : [0, \vert \widetilde{\gamma}(b) - \widetilde{\gamma}(a) \vert] \to X$ is any geodesic segment with same endpoints as $\widetilde{\gamma}$, then
\begin{align}
\forall t \in J, \, & d (\widetilde{\gamma}(t), \mathrm{im}(\gamma)) \leqslant h(\lambda)(\delta + c)
\label{eq:Morse-quasigeodesic} \\
\forall s \in [0, \vert \widetilde{\gamma}(b) - \widetilde{\gamma}(a) \vert], \, & d(\gamma(s), \mathrm{im}(\widetilde{\gamma})) \leqslant \widetilde{h}(\lambda)(\delta + c).
\label{eq:antiMorse-quasigeodesic}
\end{align}
Precisely, $h$ and $\widetilde{h}$ can be taken as $h(\lambda) = 12 (1 + 8\lambda^2)$ and $\widetilde{h}(\lambda) = 16(5 + 6 \lambda^2)$.
\end{lemma}

\begin{remark}
Our expression for $\widetilde{h}(\lambda)$ is certainly not optimal: Shchur \cite[Theorem 2]{Shchur2013} claims that $\widetilde{h}(\lambda) = O(\log \lambda)$. For us in the following, only the linear dependence over the sum of additive errors $\delta + c$ in \eqref{eq:Morse-quasigeodesic} and \eqref{eq:antiMorse-quasigeodesic} matters.
\end{remark}

\begin{proof}
A sketch of proof for the part of lemma expressed by \eqref{eq:Morse-quasigeodesic} can be found in Thurston's exposition of the Mostow rigidity theorem, \cite[5.9.2]{Thurston} with non-explicit right-hand side bound; see also an early (and more explicit) proof by Efremovich and Tihomirova \cite[p.\ 1142--1143]{EfremovichTihomirova}, also taking place in $\mathbb{H}_{\R}^n$. When projecting onto a geodesic line in hyperbolic space, the lengths of curves situated at a distance $\eta$ are contracted with a factor depending exponentially\footnote{It is useful to write the hyperbolic metric in cylindrical coordinates around $\gamma$ to appreciate that the contraction factor is a hyperbolic cosine of $\eta$.} on $\eta$, so that the length of portions of quasigeodesic leaving a tube of thickness $\eta$ around a geodesic can be bounded. This can be carried into a general argument in $\delta$-hyperbolic space, replacing length by a rough analogue; for this we build on Shchur's work \cite{Shchur2013}.
For $\alpha \in \R_{>0}$, $I \subset \R$ a bounded interval and $\sigma : I \to X$ a curve such that $\sigma(I)$ is $\alpha/2$-connected, define the length of $\sigma$ at scale $\alpha$ as
\[ \ell_\alpha (\sigma) = \sup_{(t_i) \in T(\sigma)} \sum_i \vert \sigma(t_i+1) - \sigma(t_i) \vert,\]
where $(t_i) \in T_\alpha(\sigma)$ if there is $r \in \Z_{\geqslant 0}$ such that $\inf I = t_0 < \cdots < \sup I = t_r$ and if $\lbrace \sigma(t_i) \rbrace$ is a $\alpha$-separated net in $\mathrm{im}(\sigma)$. If $\sigma$ is a $(\underline{ \lambda}, \overline{\lambda}, c)$-quasi-geodesic segment (e.g.\ a portion of $\widetilde{\gamma}$) and $\alpha$ is such that $\alpha  \geqslant 2c$, then
\begin{equation}
\ell_\alpha(\sigma) \leqslant 2 \overline{\lambda} \vert I \vert,
\label{eq:control-length-distance}
\end{equation}
see Shchur \cite[Lemma 7]{Shchur2013}. Now let $\eta$ be a positive real number (to be fixed later). Define $\mathcal{N}_{\eta} \gamma$ as the $\eta$-neighborhood of $\mathrm{im}(\gamma)$ in $X$, and
\[ U_{\eta} = \left\{ t \in J, \widetilde{\gamma}(t) \notin \mathcal{N}_{\eta} \gamma \right\}. \]

\begin{figure}[t]
\begin{center}
\begin{tikzpicture}[line cap=round,line join=round,>=angle 60,x=0.6cm,y=0.6cm]
\clip(-5,-3) rectangle (15,3);
\draw (0,0)-- (10,0);
\draw [dash pattern=on 2pt off 2pt] (-2,0)-- (0,0);
\draw [dash pattern=on 2pt off 2pt] (10,0)-- (12,0);
\draw (-2.5,0) node  {$\gamma$};
\draw [color=black] plot[domain=0:6.28, variable = \t]({0.5*cos (\t r)},{1.5*sin(\t r)});
\draw [color=black, shift={(10,0)}] plot[domain=-1.57:1.57, variable = \t]({0.5*cos (\t r)},{1.5*sin(\t r)});
\draw [color=black, shift={(10,0)}, dash pattern = on 0pt off 2pt] plot[domain=1.57:4.71, variable = \t]({0.5*cos (\t r)},{1.5*sin(\t r)});
\draw (0,1.5)-- (10,1.5);
\draw (0,-1.5)-- (10,-1.5);
\fill [color=black] (7,.9) circle (1.5pt) node[right]{$\widetilde{\gamma}(t')$};
\fill [color=black] (3,1.2) circle (1.5pt) node[below]{$\widetilde{\gamma}(t)$};
\draw [line width = 1pt, dash pattern = on 3pt off 2pt] plot[domain=3:7, variable = \t]({\t},{1.2+1.2*sin(45*(\t-3))+0.3*sin(337.5*(\t-3))});
\draw [color=black] plot[domain=0:90, variable = \t]({3+1.2*sin(\t)},{1.2*cos(\t)});
\draw [color=black] plot[domain=90:180, variable = \u]({7+cos(\u)},{0.9*sin(\u)});
\fill [color=black] (4.2,0) circle (1.5pt) node[anchor=north east]
{$p_{\gamma}(\widetilde{\gamma}(t))$};
\fill (6.1,0) circle (1.5pt) node[anchor = north west]{$p_{\gamma}(\widetilde{\gamma}(t'))$};
\draw [color=black] plot[domain=0:6.28, variable = \t]({0.9*cos (\t r)},{2.7*sin(\t r)});
\draw [color=black, shift={(10,0)}] plot[domain=-1.57:1.57, variable = \t]({0.9*cos (\t r)},{2.7*sin(\t r)});
\draw [color=black, shift={(10,0)}, dash pattern = on 0pt off 2pt] plot[domain=1.57:4.71, variable = \t]({0.9*cos (\t r)},{2.7*sin(\t r)});
\draw (0,2.7)-- (10,2.7);
\draw (0,-2.7)-- (10,-2.7);
\draw [<->] (12.5,0) -- (12.5,1.5);
\draw (12.5,0.75) node[anchor=west]{$\eta$};
\draw [<->] (13.5,0)-- (13.5,2.7);
\draw (13.5,1.35) node[anchor=west]{$\eta_I$};
\end{tikzpicture}
\caption{Proof of the Morse stability lemma \ref{lem:quantitative-Morse}.}
\end{center}
\end{figure}
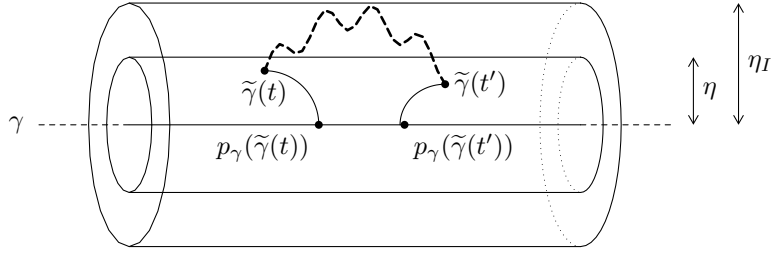
Let $I \in \pi_0 \left( U_{\eta} \right)$, $t = \inf I$ and $t' = \sup I$. $t$ and $t'$ are both finite, since $J$ is bounded and $\widetilde{\gamma}$ and $\gamma$ have the same endpoints. Then $\widetilde{\gamma}_{\mid [t,t']}$ is outside $\mathcal{N}_{\eta}(\gamma)$; by Shchur's exponential contraction estimate\footnote{Shchur's lemma is actually stated in a slightly different form, namely our $\left\vert p_{\gamma} \widetilde{\gamma}(t) - p_{\gamma} \widetilde{\gamma}(t') \right\vert$ is replaced by $\operatorname{diam} p_{\gamma} \widetilde{\gamma}(I)$, and  follows a different convention on the $\delta$ hyperbolicity constant.} \cite[Lemma 10]{Shchur2013}, there exists a constant $S \in \R_{>0}$ such that, as soon as $\eta \geqslant 2c + 12 \delta$,
\begin{align}
\left\vert p_{\gamma} \widetilde{\gamma}(t) - p_{\gamma} \widetilde{\gamma}(t') \right\vert
& \leqslant \sup \left\{ \frac{6 \delta}{c} e^{- S \eta} \ell_{2c} \widetilde{\gamma}_{\mid [t,t']}, 24 \delta \right\} \notag \\
& \leqslant 24 \delta + \frac{6 \delta}{c} e^{- S \eta} \ell_{2c} \widetilde{\gamma}_{\mid [t,t']}.
\label{eq:sum-form-of-exponential-contraction}
\end{align}
On the other hand,
\begin{align}
\ell_{2c} \widetilde{\gamma}_{\mid [t,t']}
\underset{\eqref{eq:control-length-distance}}{\leqslant} 2 \overline{\lambda} \vert t' - t \vert &
\leqslant 2 ({\overline{\lambda}}/{\underline{\lambda}}) \left[ \vert \widetilde{\gamma}(t') - \widetilde{\gamma}(t) \vert + c  \right] \notag \\
& \leqslant 2 \lambda^2  \left[ 2 \eta + \left\vert p_{\gamma} \widetilde{\gamma}(t) - p_{\gamma} \widetilde{\gamma}(t') \right\vert + c  \right], \label{eq:maj-l2c}
\end{align}
where we used the triangle inequality together with the fact that $\widetilde{\gamma}(t), \widetilde{\gamma}(t') \in \partial \mathcal{N}_{\eta}\gamma$ for the last inequality. Combining \eqref{eq:sum-form-of-exponential-contraction} and \eqref{eq:maj-l2c},
\begin{align*}
 \frac{1}{2 \lambda^2} \ell_{2c} \widetilde{\gamma}_{\mid [t,t']} & \leqslant 2 \eta + \left\vert p_{\gamma} \widetilde{\gamma}(t) - p_{\gamma} \widetilde{\gamma}(t') \right\vert + c \notag \\
 & \leqslant
 2 \eta + 24 \delta + \frac{6 \delta}{c} e^{- S \eta} \ell_{2c} \widetilde{\gamma}_{\mid [t,t']} + c,
\end{align*}
hence
\begin{equation}
\left( \frac{1}{2 \lambda^2} - \frac{6 \delta}{c} e^{-S \eta} \right) \ell_{2c} \widetilde{\gamma}_{\mid [t,t']} \leqslant 2 \eta + 24 \delta + c.
\label{eq:maj-l2c-bis}
\end{equation}
Define $\eta_I = \sup_{u \in I} d (\widetilde{\gamma}(u), \gamma)$. Then, as $c > 3 \delta \lambda^2$ by hypothesis,
\begin{equation}
\eta_I \leqslant \eta + \frac{1}{2} \ell_{2c} \widetilde{\gamma}_{\mid [t,t']} \vee 2c
\underset{\eqref{eq:maj-l2c-bis}}{\leqslant} \eta + 2c  + \frac{2 \eta + 24 \delta + c}{1/\lambda^{2} - (3 \delta /c)\cdot  e^{-S \eta}  }.
\end{equation}
It remains to set $\eta$ in order to explicit the bound on $\eta_I$ given by the last inequality. Actually, as $c \geqslant 6 \delta \lambda^2$, if $\eta = 2c + 12 \delta$ (remember that $\widetilde{\gamma}_{\mid I}$ must be at least this far for the exponential contraction to operate),
\begin{align*}
\eta_I & \leqslant \eta + 2c + \frac{2 \eta + 24 \delta +c}{1 /(2 \lambda^2)} \\
& \leqslant12 \delta + 4 c + \lambda^2 \left( 4 \eta + 48 \delta + c \right) \\
& = 12 \delta + 4c + \lambda^2 (96 \delta + 5c) \leqslant 12 ( 1+ 8 \lambda^2)(\delta + c).
\end{align*}
Finally,
\begin{align*}
\sup \left\{  d(\widetilde{\gamma}(t), \gamma) : t \in J \right\} = \eta \vee \sup_{I \in \pi_0 U_{\eta}} \eta_I & \leqslant 12(\delta + c)\vee 12 ( 1+ 8 \lambda^2)(\delta + c) \\
& =  12 ( 1+ 8 \lambda^2)(\delta + c).
\end{align*}
This is \eqref{eq:Morse-quasigeodesic}. Now, let $s \in [0, \vert \widetilde{\gamma}(b) - \widetilde{\gamma}(a) \vert]$. Because $\widetilde{\gamma}$ is $c$-connected, by Lemma \ref{lem:metric-connectedness} $p_{\gamma} \widetilde{\gamma}$ is $c + 16 \delta$-connected, so there is $s' \in [0, \vert \widetilde{\gamma}(b) - \widetilde{\gamma}(a) \vert]$ such that $\vert s' - s \vert \leqslant c + 16 \delta$ and  $s' = p_\gamma(\widetilde{\gamma}(t))$ for a $\hat{t} \in J$. The triangle inequality in $X$ yields
\begin{align*}
d(\gamma(s), \mathrm{im}(\widetilde{\gamma}) )
\leqslant \vert \gamma(s) - \widetilde{\gamma}(\hat{t}) \vert
& \leqslant \vert s - s' \vert +  \vert \gamma(s') - \widetilde{\gamma}(\hat{t}) \vert \\
& \underset{\eqref{eq:Morse-quasigeodesic}}{\leqslant} 12(1 + 8 \lambda^2) (\delta + c) + 16 \delta + c \\
& \leqslant 16(5+\lambda^2)(\delta + c).
\end{align*}
This is \eqref{eq:antiMorse-quasigeodesic}.
\end{proof}

\begin{remark}
Shchur \cite[Theorem 1]{Shchur2013} claims a stronger result (with no restriction on $c$). However S{\'e}bastien Gou{\"e}zel has informed us of a gap in the proof, so we prefer not to use this until it is fixed. 
\end{remark}

\subsection{Proof for Proposition \ref{prop:geometric-interpretation-of-X-ratio}}
\label{subsec:geom-x-ratio}

\newcounter{ecart}
\setcounter{ecart}{138}

Let $\xi_1, \ldots \xi_4$ be as in the statement of Proposition \ref{prop:geometric-interpretation-of-X-ratio} and assume that the geodesic lines $\chi_{14}$ and $\chi_{23}$ are parametrized in such a way that a common perpendicular geodesic segment $\sigma$ falls on $ \chi_{14}(0)$ and $\chi_{23}(0)$, accordingly to figure \ref{fig:geom-x-ratio}.
Let $\mathsf{H}$ be the metric subspace of $X$ defined as $\chi_{14} \cup \chi_{23} \cup \sigma$ and denote by $\vert \cdot \vert_\mathsf{H}$ the path distance in $\mathsf{H}$. By Lemma \ref{lem:quasialign} \eqref{case1quad}, if $d(\chi_{14}, \chi_{23}) \geqslant \theecart \delta$ then for all $t \in \R$, whenever $(\chi, \chi') \in \lbrace \chi_{14}, \chi_{23} \rbrace^2$,
\begin{equation}
\left\vert \vert \chi(t) - \chi'(\epsilon t) \vert - \vert \chi(t) - \chi'(\epsilon t) \vert_{\mathsf{H}} \right\vert \leqslant 4 \cdot 56 \delta = 212 \delta.
\label{eq:schroeder}
\end{equation}
For all $t \in \R$ (compare Buyalo and Schroeder \cite[p. 37]{BuyaloSchroeder}),
\begin{align}
2 \left\{ \begin{array}{c}
\left( \chi_{14}(-t) \mid \chi_{14}(t) \right)_o \\  + \left( \chi_{23}(t) \mid \chi_{23}(-t) \right)_o \\
  - \left( \chi_{14}(-t) \mid \chi_{23}(t) \right)_o \\ - \left( \chi_{14}(t) \mid \chi_{23}(-t) \right)_o
\end{array} \right\}
& = \left\{ \begin{array}{c} \vert \chi_{14}(-t) \vert +  \vert \chi_{14}(t) \vert - 2t  \\   +  \vert \chi_{23}(-t) \vert  +   \vert \chi_{23}(t) \vert - 2t \notag  \\  -\vert \chi_{14}(-t) \vert +\vert \chi_{23}(t) \vert  +  \vert \chi_{14}(-t) - \chi_{23}(t) \vert  \\   -  \vert \chi_{14}(t) \vert   - \vert \chi_{23}(-t) \vert +  \vert \chi_{14}(t) - \chi_{23}(-t) \vert \end{array} \right\} \\
& = - 4t +\vert \chi_{14}(-t) - \chi_{23}(t) \vert  + \vert \chi_{14}(t) - \chi_{23}(-t) \vert.
\label{eq:cross-difference}
\end{align}
By \eqref{eq:schroeder}, there is $\Delta$ with $\vert \Delta \vert \leqslant 2 \cdot 212 \delta = 424 \delta$ such that
\begin{align}
& - 4t +\vert \chi_{14}(-t) - \chi_{23}(t) \vert  + \vert \chi_{14}(t) - \chi_{23}(-t) \vert \notag \\
& = - 4t +\vert \chi_{14}(-t) - \chi_{23}(t) \vert_{\mathsf{H}}  + \vert \chi_{14}(t) - \chi_{23}(-t) \vert_{\mathsf{H}} + \Delta \notag \\
& = 2 d (\chi_{14}, \chi_{23}) + \Delta. \label{eq:reexpression-in-H}
\end{align}
On the other hand, by \eqref{eq:approx-X-ratio-quasimetric},
\begin{align*}
\left\vert \log_{\mu} [\xi_i]- \lim_{t \to + \infty} \left\{ \begin{array}{c}
\left( \chi_{14}(-t) \mid \chi_{14}(t) \right)_o   + \left( \chi_{23}(t) \mid \chi_{23}(-t) \right)_o \\
  - \left( \chi_{14}(-t) \mid \chi_{23}(t) \right)_o  - \left( \chi_{14}(t) \mid \chi_{23}(-t) \right)_o
\end{array} \right\} \right\vert \leqslant 8 \delta + \log_\mu 16.
\end{align*}
If $d(\chi_{14}, \chi_{23})$ is large enough, letting $ t \to + \infty$ in \eqref{eq:cross-difference} combined with the estimate \eqref{eq:reexpression-in-H}, we reach the desired inequality of Proposition \ref{prop:geometric-interpretation-of-X-ratio}. This is valid for small values as well since $\log^+$ then takes small values.

\begin{remark}
The right-hand side inequality of Proposition \ref{prop:geometric-interpretation-of-X-ratio} can be deduced from the elementary case of a metric tree via tree approximation \cite[Theorem 2.12]{GhysHarpe}. See Bourdon's remark  \cite[2.3]{BourdonChapterMostow}.
\end{remark}

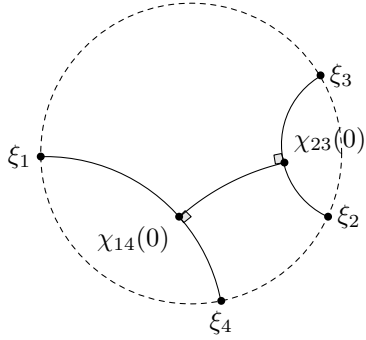
\begin{figure}
\begin{center}
\begin{tikzpicture}[line cap=round,line join=round,>=triangle 45,x=2.0cm,y=2.0cm]
\clip(-2.1,-1.2) rectangle (2.91,1.1);
\draw[line width=0.4pt,fill=black,fill opacity=0.1] (0.6,0) -- (0.55,-0.01) -- (0.56,-0.07) -- (0.62,-0.06) -- cycle;
\draw[line width=0.4pt,fill=black,fill opacity=0.1] (-0.05,-0.46) -- (0,-0.42) -- (-0.04,-0.38) -- (-0.08,-0.42) -- cycle;
\draw(0,0) [dash pattern= on 2pt off 2pt] circle (2cm);
\draw [shift={(1.13,0.06)}] plot[domain=2.11:4.28,variable=\t]({1*0.53*cos(\t r)+0*0.53*sin(\t r)},{0*0.53*cos(\t r)+1*0.53*sin(\t r)});
\draw [shift={(-0.97,-1.21)}] plot[domain=0.2:1.6,variable=\t]({1*1.19*cos(\t r)+0*1.19*sin(\t r)},{0*1.19*cos(\t r)+1*1.19*sin(\t r)});
\draw [shift={(0.98,-1.61)}] plot[domain=1.8:2.3,variable=\t]({1*1.59*cos(\t r)+0*1.59*sin(\t r)},{0*1.59*cos(\t r)+1*1.59*sin(\t r)});
\fill [color=black] (-1,-0.02) node[anchor=east]{$\xi_1$} circle (1.5pt);
\fill [color=black] (0.2,-0.98) node[anchor=north]{$\xi_4$} circle (1.5pt);
\fill [color=black] (0.91,-0.42) node[anchor=west]{$\xi_2$} circle (1.5pt);
\fill [color=black] (0.86,0.52) node[anchor=west]{$\xi_3$} circle (1.5pt);
\fill [color=black] (0.62,-0.06) node[anchor=south west]{$\chi_{23}(0)$} circle (1.5pt);
\fill [color=black] (-0.08,-0.42) node[anchor=north east]{$\chi_{14}(0)$}circle (1.5pt);
\end{tikzpicture}
\caption{Geometric interpretation of the nonnegative part of the logarithm of cross-ratio: up to an additive error, this is the distance $\vert \xi_{14}(0) - \chi_{23}(0) \vert$.}
\label{fig:geom-x-ratio}
\end{center}
\end{figure}

\section{Sublinear tracking}
\label{sec:tracking-lemmata}

Sublinearly biLipschitz embeddings of the real half-line, resp.\ of the real line admit trackings by geodesic rays, resp.\ lines; we prove this in \ref{subsec:tracking-rays}, resp.\ \ref{subsec:tracking-geodesics}. In the spirit of \eqref{eq:Morse-quasigeodesic} and \eqref{eq:antiMorse-quasigeodesic}, the bound on the tracking distance can be expressed as a constant (denoted $H, \widetilde{H}$...) times the additive error function $v$, however at the cost of being valid only farther than a given tracking radius. The tracking constants and the tracking radii depend on $v$, more precisely through its large-scale features $v \uparrow \tau$, $r_{\varepsilon}(v)$ and $\sup \lbrace r : v(r) \leqslant \text{cst}(\lambda, \delta, \ldots ) \rbrace$ described in \ref{subsec:LS-SBE}. While the use of tracking radii allow tracking estimates to take a particularly simple form when applied in \ref{subsec:distance-btw-geodesics}, their dependence upon $v$ must not be kept entirely implicit, especially it must be taken into account for later use in section \ref{sec:on-the-sphere-at-infinity} when $v$ becomes a parameter, a task undertaken in \ref{subsec:tracking-radii}.

\subsection{Preliminaries}
\label{subsec:preliminaries-of-tracking}
Unless otherwise stated, geodesic rays into a pointed metric space are assumed to have their origin at base-point.
This convention will not apply to the rougher $O(v)$-rays that we define hereafter.

\begin{definition}
Let $u$ be an admissible function and $X$ a metric space. A $O(u)$-geodesic, resp. a $O(u)$-ray in $X$ is a $O(u)$-sublinearly biLipschitz embedding $\R \to X$, resp. $\R_{\geqslant 0} \to X$.
\end{definition}

When $u=1$, this is the classical notion of a quasigeodesic, resp.\ of a quasigeodesic ray.
By definition, $O(u)$-geodesics, resp.\ $O(u)$-rays, are sent to $O(u)$-geodesics resp.\ $O(u)$-rays when one applies a $O(u)$-sublinearly biLipschitz embedding to the space.
$O(u)$-geodesics behave like quasi-geodesic inside every ball, with an additive error parameter controlled by the radius; however the containing ball sits in the target space, so that the dependence of the additive error on radius only becomes apparent on the large scale. We turn this observation into a lemma, which may be considered as an alternative definition for large-scale Lipschitz embeddings, easier to handle through certain technical steps.

\begin{lemma}
\label{lem:techincal-def-embedding}
Let $u$ be an admissible function. Let $(\underline{\lambda}, \overline{\lambda})$ be Lipschitz constants, let $v = O(u)$ be nondecreasing, and let $f : (X,o) \to (Y,o)$ be a large-scale $(\underline{\lambda}, \overline{\lambda}, v)$-biLipschitz embedding. Then there exist $\widehat{v} = O(u)$, $t_{\ocircle} \in \R_{\geqslant 0}$ and $R_{\ocircle} \in \R_{\geqslant 0}$ (depending on $f$ and $v$) such that for all $x, x_1, x_2 \in X$
\renewcommand{\theenumi}{\Roman{enumi}}
\begin{enumerate}
\item{\label{item:first-point_of-lemma-32}If $x \notin B(o, t_\ocircle)$ or $f(x) \notin B(o, R_{\ocircle})$ then
\begin{equation*}
\frac{1}{3\lambda} \vert x \vert \leqslant \vert x \vert \wedge \vert f(x) \vert \leqslant 3 \lambda \vert x \vert.
\end{equation*}
}
\item{\label{item:second-point_of-lemma-32}If
$x_1, x_2 \in X \setminus B(o, t_{\ocircle}) \; \; \text{or} \; \; f(x_1), f(x_2) \in Y \setminus B(o, R_{\ocircle})$,
then
\begin{equation*}
\begin{cases}
\vert f(x_1) - f(x_2) \vert \leqslant \lambda \vert x_1 - x_2 \vert + \widehat{v} \left( (\vert x_1 \vert \vee \vert x_2 \vert ) \wedge (\vert f(x_1) \vert \vee \vert f(x_2) \vert) \right), \\
\vert f(x_1) - f(x_2) \vert \geqslant \frac{1}{\lambda} \vert x_1 - x_2 \vert -\widehat{v} \left( (\vert x_1 \vert \vee \vert x_2 \vert ) \wedge (\vert f(x_1) \vert \vee \vert f(x_2) \vert) \right).
\end{cases}
\end{equation*}
}
\end{enumerate}
Moreover $t_\ocircle$, $R_{\ocircle}$ and $\widehat{v}$ may be taken as:
\begin{align}
t_{\ocircle}(\vert f(o) \vert, v) & =  \sup \left\{ r : v(r) \geqslant \frac{r}{3 \lambda} \right\} \vee 3 \lambda \vert f(o) \vert = r_{1/(3 \lambda)}(v) \vee 3 \lambda \vert f(o) \vert ,
\label{eq:explicit-t-circle} \\
R_{\ocircle}(\vert f(o) \vert,v) & = 4 \vert f(o) \vert \vee 2 (2 {\lambda} + 1) t_\ocircle(\vert f(o)  \vert,v),\, \text{and} 
\label{eq:explicit-R-circle}
\\
\widehat{v} & = (v \uparrow 3 \lambda) v.
\end{align}
\end{lemma}

\begin{proof}
By definition of $t_{\ocircle}(\vert f(o) \vert,v)$, for all $x \in X \setminus B(o, t_\ocircle)$, $\vert f(o) \vert \leqslant 1/(3 \lambda) \vert x \vert$ and $v(\vert x \vert) \leqslant \frac{1}{3 \lambda} \vert x \vert$, so
$ \frac{1}{3\lambda} \vert x \vert \leqslant \vert f(x) \vert \leqslant \left( \lambda + \frac{2}{3\lambda} \right) \vert x \vert \leqslant 3 \lambda \vert x \vert $; this is the first case in \eqref{item:first-point_of-lemma-32}.
Now assume that $R_{\ocircle}$ is defined as in \eqref{eq:explicit-R-circle}. Note that $R_{\ocircle} \geqslant 2r _{1/(3 \lambda)}(v) \geqslant r_{1/2}(v)$ so that if $f(x) \in Y \setminus B(o,R_{\ocircle}(\vert f(o) \vert,v))$, then 
\begin{align*}
\vert x \vert & \geqslant \overline{\lambda}^{-1} \left({\vert f(x) \vert - \vert f(o) \vert - v(\vert x \vert)} \right) \\
& \geqslant
\begin{cases}
{{\lambda}}^{-1}({\vert f(x) \vert - \vert f(o) \vert - v(\vert f(x) \vert}) & \text{if} \; \vert x \vert \leqslant \vert f(x) \vert, \, \text{or} \\
{\lambda}^{-1}({\vert f(x) \vert - \vert f(o) \vert - \vert x \vert /2}) & \text{if} \; \vert x \vert \geqslant \vert f(x) \vert.
\end{cases}
\end{align*}
In both cases, 
\[ \vert x \vert \geqslant \frac{1}{\lambda + 1/2} \left( \frac{1}{2} \vert f(x) \vert - \vert f(o) \vert \right), \]
and then $\vert x \vert \geqslant t_{\ocircle}(\vert f(o) \vert, v)$ since by definition $R_{\ocircle} \geqslant 2 \vert f(o) \vert + ( 2 \overline{\lambda} +1) t_{\ocircle}$.
Hence the hypotheses in \eqref{item:first-point_of-lemma-32} actually reduce to the single first one.
\eqref{item:second-point_of-lemma-32} follows from \eqref{item:first-point_of-lemma-32}, the fact that $f$ is a $(\lambda,v)$-embedding, that $\widehat{v}$ is nondecreasing, and the left distributivity of $\leqslant$ over $\wedge$.
\end{proof}

\subsection{Rays}

\label{subsec:tracking-rays}
Let $Y$ be a proper geodesic hyperbolic space, and $\widetilde{\gamma} : \R \to Y$ a $O(u)$-geodesic ray. Inequality \eqref{eq:cauchy-estimates-Ou-geodesic} says in particular that $\left\{ \widetilde{\gamma}(t) \right\}_{t \in \Z_{\geqslant 0}}$ is a Cauchy-Gromov sequence. Since $Y$ is proper and geodesic, its Gromov boundary is equal to $\partial_\infty Y$ and there exists a geodesic ray $\gamma : (\R_{\geqslant 0}, 0) \to (Y,o)$ such that $\eta := [\gamma] = \partial_\infty \widetilde{\gamma}(+ \infty)$.
We will prove that $\gamma$ actually tracks $\widetilde{\gamma}$, in the sense that the growth of distance between them is in the $O(u)$-class. 

\begin{lemma}
\label{lem:on-asymptotic-rays}
Let $\delta \in \R_{\geqslant 0}$, and let $(Y,o)$ be a $\delta$-hyperbolic proper geodesic space. Let $\gamma : \R \to Y$ be a geodesic ray into $Y$, and let $\gamma'$ be a non-pointed geodesic ray asymptotic to $\gamma$, i.e. $[\gamma]=[\gamma'] \in \partial_\infty Y$. Then for all $s \in \R_{\geqslant 0}$ such that $s \geqslant \vert \gamma'(0) \vert + 16 \delta$,
\[ d(\gamma'(s), \mathrm{im}(\gamma)) \leqslant 8 \delta. \]
\end{lemma}

\begin{proof}
This is a classical result in hyperbolic metric geometry, use for instance the proof of $(ii) \implies (iii)$ in \cite[Proposition 7.1]{GhysHarpe} with appropriate changes of notations, and replace Ghys and Harpe's $D$ with $\sup \lbrace \vert \gamma'(0) \vert, 16 \delta \rbrace$.
\end{proof}

\begin{lemma}[Sublinear tracking for rays]
\label{lem:tracking-rays}
Let $v$ be an unbounded admissible function. Let $(Y,o)$ be a proper, geodesic, pointed metric space. Assume there exists $\delta \in \R_{\geqslant 0}$ such that $Y$ is $\delta$-hyperbolic.
Let $(\underline{\lambda}, \overline{\lambda})\in \R_{>0}^2$ be Lipschitz data, and let $\widetilde{\gamma} : \R_{\geqslant 0} \to Y$ be a $(\underline{\lambda}, \overline{\lambda},v)$-ray.
Let $\eta \in \partial_\infty Y$ be the endpoint of $\widetilde{\gamma}$, and let $\gamma$ be any geodesic ray such that $[\gamma] = \eta$.
Then there exist constants $H, \widetilde{H} \in \R_{>0}$, $t_{\backsimeq}, R_{\backsimeq}, \in \R_{\geqslant 0}$ such that for all positive real $t$ and $s$,
\begin{align}
t \geqslant t_{\backsimeq} \implies
d(\widetilde{\gamma}(t), \gamma) & \leqslant Hv(t)
\label{eq:Morse-estimate-Ou-ray} \\
s \geqslant R_{\backsimeq} \implies
d(\gamma(s), \mathrm{im}(\widetilde{\gamma})) & \leqslant \widetilde{H}v(s),
\label{eq:anti-morse-estimate-Ou-ray}
\end{align}
where $H$ and $\widetilde{H}$ depend on $\lambda$ and $v$ only, while $t_{\backsimeq}$ and
$R_{\backsimeq}$ can be decomposed into
\begin{align}
t_{\backsimeq} & = t_{\backsimeq}^0(\lambda, v, \delta) + 2 \lambda \vert \widetilde{\gamma}(0) \vert \\
R_{\backsimeq} & = R_{\backsimeq}^0(\lambda, v, \delta) + \vert \widetilde{\gamma}(0) \vert.
\end{align}
\end{lemma}

\begin{remark}
In view of Lemma \ref{lem:quantitative-Morse}, it does matter for us that $v$ be unbounded. If $v$ is bounded, though, $\widetilde{\gamma}$ is a quasi-geodesic ray and the same result classically holds, see for instance Ghys and Harpe \cite[5.25]{GhysHarpe}, with extra additive terms in the estimates \eqref{eq:Morse-estimate-Ou-ray} and \eqref{eq:anti-morse-estimate-Ou-ray}.
\end{remark}

\begin{remark}
It is important to make the dependence of the tracking radius $R_{\backsimeq}^0$ upon the function $v$ explicit, at least to some extent. However, in order not to overload the current proof, we reconstruct it separately (but along with other tracking radii) in subsection \ref{subsec:tracking-radii}, and only keep record of the steps needed for its definition here, with enough details to ensure that it depends on $\lambda$, $v$ and $\delta$ only.
\end{remark}

\begin{proof}[Sketch of proof for Lemma \ref{lem:tracking-rays}]
\renewcommand{\qedsymbol}{}
For every $t \in \R_{\geqslant 0}$, set a real positive $T$ large enough according to $t$ so that \eqref{eq:cauchy-estimates-Ou-geodesic} ensures the Gromov product $(\widetilde{\gamma}(T), \eta)_o$ is significantly greater than $\vert \widetilde{\gamma}(t) \vert$, and use the stability lemma \ref{lem:quantitative-Morse} to prove that $\widetilde{\gamma}(t)$ is not far from the geodesic segment $\gamma_T$ between $o$ and $\widetilde{\gamma}(T)$. 
Here keeping an efficient inequality requires that $T$ stay within linear control of $t$, which can be done consistently with the antagonist constraint of \eqref{eq:cauchy-estimates-Ou-geodesic}. 
Further, show that the projection of $\widetilde{\gamma}(t)$ on $\gamma_T$ is close to $\gamma$, using the slim triangle $o \widetilde{\gamma}(T) \eta$, see figure \ref{fig:tracking-Ou-rays}. 
Finally, \eqref{eq:anti-morse-estimate-Ou-ray} is deduced from \eqref{eq:Morse-estimate-Ou-ray} with a metric connectedness argument in the same way that \eqref{eq:antiMorse-quasigeodesic} was deduced from \eqref{eq:Morse-quasigeodesic} in the proof of Lemma \ref{lem:quantitative-Morse}.
\end{proof} 

\begin{proof}[Proof of lemma \ref{lem:tracking-rays}]
Setting $\alpha = \underline{\lambda}/2$ in the Gromov product estimate \eqref{eq:cauchy-estimates-Ou-geodesic} and letting $s \to + \infty$,
\begin{equation}
\forall T \in \Z_{\geqslant t_\alpha}, \,
([\gamma] \mid \widetilde{\gamma}(T))_o \geqslant \underline{\lambda} T /2  - M'(\underline{\lambda},\delta),
\label{eq:Gromov-product-estimate}
\end{equation}
where $M'(\underline{\lambda}, \delta) = \log_\mu \left( \frac{2}{1- \mu^{(3 \underline{\lambda})/4}} \right)$.
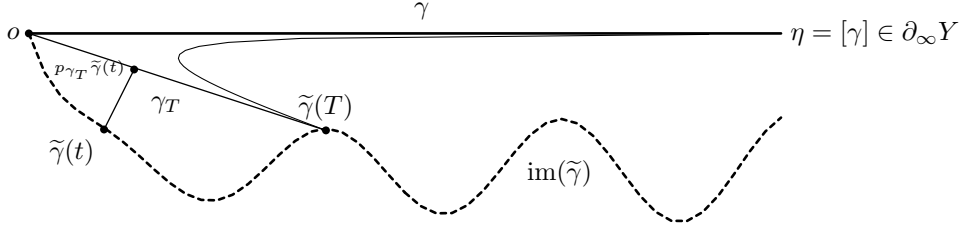
\begin{figure}[t]
\begin{center}
\begin{tikzpicture}[line cap=round,line join=round,>=triangle 45,x=1cm,y=0.6cm]
\clip(-0.5,-1.2) rectangle (15,5);
\fill [color=black] (0,3) circle (1.5pt) node[left]{$o$};
\draw [line width=1pt] (0,3)-- (10,3);
\draw [shift={(0,0)}, line width = 1pt, dash pattern = on 2pt off 2pt] plot[domain=0:5,variable=\t]({\t},{0.5*sin(2*\t r)*(ln(2+\t))+3*exp(-2*\t)});
\draw [shift={(0,0)}, line width = 1pt, dash pattern = on 2pt off 2pt] plot[domain=5:10,variable=\t]({\t},{0.5*sin(2*\t r)*(ln(2+\t)});
\fill [color=black] (3.95,0.85) circle (1.5pt) node[above]{$\widetilde{\gamma}(T)$};
\fill [color=black] (1,0.87) circle (1.5pt) node[anchor=north east]{$\widetilde{\gamma}(t)$};
\draw [line width=0.5pt] (0,3)-- (3.95,0.85);
\draw [line width=0.5pt] (1,0.87) --(1.4,2.2);
\fill [color=black] (1.4,2.2) circle (1.5pt) node[anchor=east]{\tiny{$p_{\gamma_T}\widetilde{\gamma}(t)$}};
\draw [shift={(0,0)}] plot[line width = 1pt, domain=0.01:1,variable=\t]({3.95*\t + 1/(sqrt(\t)) - 1},{3-2.15*\t});
\draw (1.5,1.8) node[anchor=north west] {$\gamma_{T}$};
;
\draw (10,3) node[anchor=west] {$ \eta = [\gamma] \in \partial_\infty Y$};
\draw (5,3.5) node[anchor=west] {$\gamma$};
\draw (6.5,0.4) node[anchor=north west] {$\mathrm{im}(\widetilde{\gamma})$};

\end{tikzpicture}
\caption{Sublinear tracking for $O(u)$-rays (case $o = \widetilde{\gamma}(0)$).}
\label{fig:tracking-Ou-rays}
\end{center}
\end{figure}
We will first prove the lemma in the case $\vert \widetilde{\gamma}(0) \vert = 0$, i.e. $o = \widetilde{\gamma}(0)$.
Let $(t,T) \in \R_{\geqslant 0}^2$ be such that $t \leqslant T$.
Since $v$ is nondecreasing and unbounded, there is $T_2 \in \R_{>0}$ such that for all $T \in \R_{\geqslant T_2}$, $v(T) \geqslant 6 \lambda^2 \delta$.
By inequality \eqref{eq:Morse-quasigeodesic} of Lemma \ref{lem:quantitative-Morse} applied to $\gamma_T = [o \widetilde{\gamma}(T)]$ and $\widetilde{\gamma}_{\mid [0, T]}$,
\begin{equation}
\label{eq:distance-Ou-to-gammaT}
\text{If} \; T \geqslant T_2,\; d (\widetilde{\gamma}(t), \mathrm{im}(\gamma_T)) \leqslant  h(\lambda) \left( \delta + v(T) \right).
\end{equation}
Similarly,
\begin{equation}
\forall S \in [0, \vert \widetilde{\gamma}(T) \vert],\; (\gamma_T(S), \mathrm{im}(\widetilde{\gamma})) \leqslant \widetilde{h}(\lambda) \left( \delta + v(T) \right).
\label{eq:distance-gammaST-gamma'}
\end{equation}
Now, fix $t \in \R_{\geqslant 0}$.
We look for $T \in \Z_{\geqslant t_\alpha}$ such that
$(\widetilde{\gamma}(T) \mid \eta)_o \geqslant 2 \vert \widetilde{\gamma}(t) \vert$.
Thanks to \eqref{eq:Gromov-product-estimate} this holds when $T = \left\lceil (4 / \underline{\lambda}) \left( \vert \widetilde{\gamma}(t) \vert + M (\underline{\lambda}, \delta) \right) \right\rceil$.
Let $\triangle_T$ be a (geodesic, semi-ideal) triangle with vertices $o$, $\widetilde{\gamma}(T)$ and $\eta$ (Recall that by convention, $\gamma_T$ is the side of $\triangle_T$ between $o$ and $\widetilde{\gamma}(T)$).
By \eqref{eq:distance-Ou-to-gammaT} and the triangle inequality,
\[ \left\vert p_{\gamma_T} (\widetilde{\gamma}(t)) \right\vert \leqslant h(\lambda) \left( \delta + v(T) \right) + \vert \widetilde{\gamma}(t) \vert . \]
Again by the triangle inequality,
\begin{align*}
d \left( p_{\gamma_T} \widetilde{\gamma}(t), [\widetilde{\gamma}(T) \eta) \right)
& \geqslant \left\vert p_{[\widetilde{\gamma}(T) \eta)} o \right\vert - \left\vert p_{\gamma_T} (\widetilde{\gamma}(t)) \right\vert \\
& \geqslant
 \left\vert p_{[\widetilde{\gamma}(T) \eta)} o \right\vert - h(\lambda) \left( \delta + v(T) \right) - \vert \widetilde{\gamma}(t) \vert .
\end{align*}
By the triangle inequality $\left\vert p_{[\widetilde{\gamma}(T) \eta)} o \right\vert \geqslant (\widetilde{\gamma}(T) \mid \eta)_o$, so that the previous inequality becomes
\begin{align}
d \left( p_{\gamma_T} \widetilde{\gamma}(t), [\widetilde{\gamma}(T) \eta) \right)
& \geqslant (\widetilde{\gamma}(T) \mid \eta)_o - h(\lambda) \left( \delta + v(T) \right) - \vert \widetilde{\gamma}(t) \vert \notag \\
& \geqslant \vert \widetilde{\gamma}(t) \vert - h(\lambda) \left( \delta + v(T) \right),
\label{eq:lower-bound-d-p-sigma}
\end{align}
where we have replaced the Gromov product according to the definition of $T$. Let us now bound $v(T)$. By definition,
\begin{align*}
T & \leqslant 4 \lambda \vert \widetilde{\gamma}(t) \vert + 4 \lambda M'(\underline{\lambda}, \delta) +1 \\
& \leqslant 4 \lambda \left( \lambda t + v(t) \right) + 4 \lambda M'(\underline{\lambda}, \delta) + 1,
\end{align*}
hence for $t \geqslant t_0 = \sup \lbrace r_\lambda (v), 4 \lambda M'(\underline{\lambda}, \delta) + 1 \rbrace$, $T \leqslant (1 + 8 \lambda^2)t$, and
\begin{equation}
v(T) \leqslant (v \uparrow {1+8 \lambda^2}) v(t).
\label{eq:bound-on-vT}
\end{equation}
Substituting this in inequality \eqref{eq:lower-bound-d-p-sigma}, for all $t$ such that $t \geqslant t_0$,
\[  d \left( p_{\gamma_T} \widetilde{\gamma}(t), [\widetilde{\gamma}(T) \eta) \right) \geqslant \vert \widetilde{\gamma}(t) \vert - h(\lambda)(\delta + (v \uparrow {1+8 \lambda^2}) v(t)). \]
Applying Lemma \ref{lem:techincal-def-embedding} to $\widetilde{\gamma}$, define
\begin{equation*}
t_1 := t_\ocircle \vee \sup \lbrace s : v(s) \leqslant \delta \rbrace \vee 3 \lambda r_{1/(2 h(\lambda) v \uparrow 1 + 8 \lambda^2)}(v) \vee 12 \lambda \delta \vee t_0.
\end{equation*}
Then by definition of $t_1$,
\begin{equation*}
\forall t \in \R_{> t_1}, \,  d \left( p_{\gamma_T} \widetilde{\gamma}(t), [\widetilde{\gamma}(T) \eta) \right) > 4 \delta.
\end{equation*}
But $\triangle_T$ is $4\delta$-slim, so
\begin{equation}
\forall t \in \R_{> t_1}, \, d \left( p_{\gamma_T} \widetilde{\gamma}(t), \gamma \right) \leqslant 4 \delta.
\label{eq:dist-projgammaT-to-gamma}
\end{equation}

Let $t_2 \in \R_{\geqslant t_1}$ be such that $T \geqslant T_2$ for all $t \in \R_{\geqslant t_2}$.
Recall that this was the necessary condition for \eqref{eq:distance-Ou-to-gammaT} to hold, and note that $t_2$ only depends on $\lambda, \delta, v$, by the explicit expression of $t_\ocircle$ and the fact that $\widetilde{\gamma}(0) = o$. By the triangle inequality,
\begin{align*}
\forall t \in \R_{> t_2}, \, d \left( \widetilde{\gamma}(t), \mathrm{im}(\gamma) \right)
& \leqslant \vert \widetilde{\gamma}(t) - p_{\gamma_T} \widetilde{\gamma}(t) \vert + d \left( p_{\gamma_T} \widetilde{\gamma}(t), \mathrm{im}(\gamma) \right) \\
& \underset{\eqref{eq:distance-Ou-to-gammaT},\, \eqref{eq:dist-projgammaT-to-gamma}}{\leqslant}
h(\lambda) \left( \delta + v(T) \right) + 4 \delta \\
 & \underset{\eqref{eq:bound-on-vT}}{\leqslant}
 h(\lambda) \left( \delta + (v \uparrow {1+8 \lambda^2}) v(t) \right).
\end{align*}
Define $t_3 = \sup \lbrace s : v(s) \leqslant h(\lambda) \delta \rbrace \vee t_2$. The last inequality implies
\begin{equation}
\forall t \in \R_{\geqslant t_3},\,
d (\widetilde{\gamma}(t), \mathrm{im}(\gamma))
\leqslant \left( 2 h(\lambda) (v \uparrow {1+8 \lambda^2}) + 1 \right) v(t).
\label{ineq:almost-finished-Morse-Ou-rays}
\end{equation}
We have proved \eqref{eq:Morse-estimate-Ou-ray} in the special case $\vert \widetilde{\gamma}(0) \vert = 0$ ; set $H_0(\lambda,v) = 2 h(\lambda) (v \uparrow {1+8 \lambda^2}) + 1$. In the general case, let $\gamma'$ be a non-pointed geodesic ray $[\widetilde{\gamma}(0)\eta)$. Apply \eqref{ineq:almost-finished-Morse-Ou-rays} to $\widetilde{\gamma}$ and $\gamma'$. This gives the existence, for all $t \in \R_{\geqslant t_3}$, of $s' \in \R_{\geqslant 0}$ such that $d(\widetilde{\gamma}(t), \gamma'(s')) \leq H_0(\lambda,v) v(t)$. Moreover $s' = d(\widetilde{\gamma}(0),s) \geqslant d(\widetilde{\gamma}(0), \widetilde{\gamma}(t)) - H_0(\lambda,v) v(t) \geqslant \underline{\lambda} t - \left( 1 + H_0 (\lambda,v) \right) v(t)$. Hence for all $t \in \R$ such that $t \geqslant t_4 := \sup \left\{ t_3,  r_{\underline{\lambda}/(2 + 2 H_0(\lambda,v))}(v) \right\}$,
\begin{equation}
s' \geqslant t/(2 \lambda).
\label{eq:lower-control-on-s'}
\end{equation}
Set $t_5 := t_4 \vee \sup \lbrace r : v(r) \leqslant 8 \delta \rbrace $, $t^0_{\backsimeq} := t_5 \vee 16 \delta$ and then  $t_{\backsimeq} = t^0_{\backsimeq} + 2 \lambda \vert \widetilde{\gamma}(0) \vert$.
By \eqref{eq:lower-control-on-s'}, $s' \geqslant \vert \widetilde{\gamma}(0) \vert + 16 \delta$. Moreover $\gamma'$ and $\gamma$ are asymptotic, so that by Lemma \ref{lem:on-asymptotic-rays} on asymptotic geodesic rays, $d(\gamma'(s'), \mathrm{im}(\gamma)) \leqslant 8 \delta$.
By the triangle inequality and the definition of $t_{\backsimeq}$ we conclude that
\[ d(\widetilde{\gamma}(t), \gamma) \leqslant \vert \widetilde{\gamma}(t) - \gamma'(s') \vert + d(\gamma'(s'), \mathrm{im}(\gamma)) \leqslant \left( 1 + H_0(\lambda,v) \right) v(t).  \]
By construction, $t^0_{\backsimeq}$ only depends on $\lambda, v, \delta$, so \eqref{eq:Morse-estimate-Ou-ray} is reached in the general case.
From now on we proceed to attain \eqref{eq:anti-morse-estimate-Ou-ray}.
As before start by assuming $\vert \widetilde{\gamma}(0) \vert = 0$. For all $t \in \R_{\geqslant 0}$, since $\widetilde{\gamma}_{\mid [0, t]}$ is $v(t)$-connected, $p_{\gamma}\widetilde{\gamma}_{\mid [0, t]})$ is $v(t) + 16 \delta$-connected by Lemma \ref{lem:metric-connectedness}, in particular it is $2v(t)$-connected as soon as $t \geqslant t_6 := \sup \lbrace r : v(r) \leqslant 16 \delta \rbrace$.
Since $\vert \widetilde{\gamma}(t) \vert \geqslant (\underline{\lambda}/3) t$ by Lemma \ref{lem:techincal-def-embedding} when $t \geqslant t_7 := \sup \lbrace t_6, t_\ocircle \rbrace$, under this last condition the convex hull of $p_{\gamma}\widetilde{\gamma}_{\mid [0, t]})$ contains $\gamma([0,(\underline{\lambda}/3) t - Hv(t)])$ where $H$ is the constant from \eqref{eq:Morse-estimate-Ou-ray} (note that $t_7$ only depends on $v, \lambda, \delta$ since we are assuming $\widetilde{\gamma}(0) = o$).
Hence for all $t \geqslant t_8 = \sup \lbrace t_7, r_{\underline{\lambda}/(6H)}(v) \rbrace$, every $s \in [0, (\underline{\lambda}/6)t]$ lies between two orthogonal projections of points of $\widetilde{\gamma}_{\mid [0, t]}$ on $\gamma$.
Define $R_8 := t_8 / (6 \lambda)$. For all $s \in \R$ such that $s \geqslant R_8$, there is $t_s \in [0, 6 \lambda s]$ such that
\begin{equation}
 \vert \gamma(s) - p_{\gamma}(\widetilde{\gamma}(t_s) \vert \leqslant 2 v ( 6 \lambda s) \leqslant 2 ( v \uparrow 6\lambda) v(s).
\end{equation}
By the triangle inequality,
\begin{align}
\left\vert \gamma(s) - \widetilde{\gamma}(t_s) \right\vert & \leqslant \left\vert \gamma(s) - p_{\gamma}(\widetilde{\gamma}(t_s)) \right\vert + \vert p_\gamma (\widetilde{\gamma}(t_s)) - \widetilde{\gamma}(t_s) \vert \notag \\
& \leqslant H_0 v(t_s) + 2 (v \uparrow 6 \lambda) v(s) \notag \\
& \leqslant 2 ( v \uparrow 6 \lambda) (H_0+1) v(s) \;\text{for} \; s \geqslant R_8,
\label{eq:antiMorse-techniqcal-rays}
\end{align}
where we used that $v(t_s) \leqslant (v \uparrow 6 \lambda)v(s)$ for the last inequality.
Set $\widetilde{H}_0(\lambda, v) :=   2 ( v \uparrow 6 \lambda) (H_0+1)$, and assume from now that $\widetilde{\gamma}(0)$ is arbitrary. Define $R_{\backsimeq}^0 = R_8 \vee \sup \lbrace r : \widetilde{H}_0 v(r) \leqslant 8 \lambda \rbrace  \vee 16 \delta$ and $\widetilde{H} = 2 \widetilde{H}_0$. Then by Lemma \ref{lem:on-asymptotic-rays} applied to $\gamma = o\eta$ and $\gamma' = \widetilde{\gamma}(0)\eta$, \eqref{eq:antiMorse-techniqcal-rays} and the triangle inequality, for all $s \geqslant R_{\backsimeq}^0 + \vert \widetilde{\gamma}(0) \vert$, $d(\gamma(s), \mathrm{im}(\widetilde{\gamma})) \leqslant \widetilde{H} v(s)$.
\end{proof}

\subsection{Geodesics}
\label{subsec:tracking-geodesics}

Our next aim consists in tracking $O(u)$-geodesics $\widetilde{\gamma}$. For this we need two steps:
\begin{enumerate}
\item{Control the Gromov product of ends $\partial_\infty \widetilde{\gamma}(- \infty)$ and $\partial_\infty \widetilde{\gamma}(+ \infty)$ with respect to $\vert \widetilde{\gamma}(0)\vert$. This is achieved by Lemma \ref{lem:no-round-trip}.}
\item{Track $\widetilde{\gamma}$ near both ends, starting at a distance linearly controlled by their Gromov product, and interpolate between using the classical version of the stability lemma. This strategy is set up in Lemma \ref{tracking-geodesics}.
\label{item:second-step}
}
\end{enumerate}

Beware that, in contrast to the situation with (quasi)geodesics, one cannot re-parametrize a $(\lambda, v)$-geodesic (e.g.\ to assume that $\widetilde{\gamma}(0)$ is the closest\footnote{Such a point $\widetilde{b}$ exists since $\widetilde{\gamma}$ is proper.} point $\widetilde{b}$ to $o$ in $\mathrm{im}(\widetilde{\gamma})$) without changing function $v$. For this reason, and in order to simplify bounds on the tracking distance in step \eqref{item:second-step}, we introduce an additional constant $L$ and, from Lemma \ref{tracking-geodesics} on, make an assumption that $\vert \widetilde{\gamma}(0) \vert \leqslant L \widetilde{b}$.

\begin{lemma}
\label{lem:no-round-trip}
Let $\delta \in \R_{\geqslant 0}$, $\lambda \in \R_{\geqslant 1}$ be constants, and let $(Y,o)$ be a pointed proper geodesic $\delta$-hyperbolic space. Let $v$ be an admissible function.
Let $\widetilde{\gamma}$ be a $(\lambda, v)$-geodesic into $Y$.
Denote $\eta_{\pm}$ in $\partial_\infty X$ its endpoints, precisely $\eta_{\pm} = \partial_\infty \widetilde{\gamma}(\pm \infty)$.
Then
\label{ineq:proj-projquasi}
there exist ${K} = {K}(\lambda, v, \delta)$ and $R_\sqcap = R_\sqcap(\lambda, v, \delta)$, both in $\R_{> 0}$ such that if $\vert \widetilde{\gamma}(0) \vert \geqslant R_{\sqcap}$,
\begin{equation}
\label{eq:no-round-trip}
\left( \eta_- \mid \eta_+ \right)_o \leqslant K \vert \widetilde{\gamma}(0) \vert.
\end{equation}
\end{lemma}

\begin{proof}
The proof uses that that $O(u)$-geodesics cannot make large round trips; see figure \ref{fig:no-round-trip}.
Assume by contradiction that $\left( \eta_- \mid \eta_+ \right)_o \geqslant 3 (R_{\backsimeq}^0 + \vert \widetilde{\gamma}(0) \vert) + 4 \delta = 3 R_{\backsimeq} + 4 \delta$ for $\widetilde{\gamma}(0)$ arbitrarily far.
Track the rays  $\widetilde{\gamma}_{-} : t \mapsto \widetilde{\gamma}(-t)$ and $\widetilde{\gamma}_+ : t \mapsto \widetilde{\gamma}(t)$ with geodesic rays $\gamma_-$ and $\gamma_+$.
Let $\gamma = (\eta_- \eta_+)$ be a geodesic line.
Define $p_\pm$ as the intersection point of $\gamma_\pm$ and $\partial B(o, 2 R_{\backsimeq})$, i.e. $p_\pm = \gamma_\pm (R_\backsimeq)$.
The twice-ideal triangle $o\eta_-\eta_+$ is $4\delta$-thin, and by the triangle inequality
\begin{align*}
d(p_\pm, \gamma) & \geqslant d(o, p_\gamma(o)) - d(o, p_\pm) \\
& \geqslant \left( \eta_- \mid \eta_+ \right)_o - 2 R_{\backsimeq} \geqslant R_{\backsimeq} + 4 \delta > 4 \delta,
\end{align*}
so $d(p_\pm, \gamma_\mp) < 4 \delta$ and $\vert p_- - p_+ \vert \leqslant 8 \delta$ (where we used that both points $p_+$ and $p_-$ lie on the same sphere centered at $o$).
By sublinear tracking lemma \ref{lem:tracking-rays}, there is $q_\pm$ on $\mathrm{im}(\widetilde{\gamma}_\pm)$ such that $\vert p_\pm - q_\pm \vert \leqslant \widetilde{H} v(2 R_\backsimeq)$, and thanks to the triangle inequality,
\begin{equation}
\vert q_+ - q_- \vert \leqslant \vert p_+ - p_- \vert + 2\widetilde{H}v(2 R_{\backsimeq}) \leqslant 8 \delta + 2\widetilde{H}v(2 R_{\backsimeq}).
\label{eq:qs-are-close}
\end{equation}
Let $t_+, t_-$ in $\R$ be such that $q_\pm = \widetilde{\gamma}(t_\pm)$, and write $T = \sup \lbrace \vert t_+ \vert, \vert t_- \vert \rbrace$.
The portion of $\widetilde{\gamma}$ between $t_-$ and $t_+$ is a $(\lambda, v(T))$ quasi-geodesic segment.
By length-distance estimate for quasi-geodesic, for $\alpha$ large enough,
\begin{align}
\ell_{\alpha} \left( \widetilde{\gamma}_{\mid [t_-, t_+]}  \right)  \leqslant 2 \lambda (t_+ - t_-)  & \leqslant 2 \lambda \left( \lambda \vert q_- - q_+ \vert + v(T) \right) \notag \\
& \underset{\eqref{eq:qs-are-close}}{\leqslant} 4 \lambda^2 Hv(2 R_{\backsimeq}) + 2 \lambda v(T) + 16 \lambda^2 \delta.
\label{eq:upper-bound-on-la-portion}
\end{align}
$T$ can be bounded above for $\vert \widetilde{\gamma}(0) \vert$ large enough:
\begin{align*}
\underline{\lambda} T - v(T) & \leqslant \sup \left\{ \vert \widetilde{\gamma}(0) - q_- \vert, \vert \widetilde{\gamma}(0) - q_+ \vert  \right\} \\
& \leqslant 2 R_{\backsimeq}^0 + \vert \widetilde{\gamma}(0) \vert + 2 \widetilde{H}v(T),
\end{align*}
so that since $v(T) \ll T$, there is a constant $T_0$ depending on $v, \lambda$ (explicitly $T_0 = r_{1/(8 \lambda \widetilde{H})}(v)$)  such that $T \leqslant \inf \lbrace T_0, \lambda (2 R_{\backsimeq}^0 + \vert \widetilde{\gamma}(0) \vert) \rbrace$.
On the other hand, $\ell_{\alpha} \left( \widetilde{\gamma}_{\mid [t_-, t_+]}  \right) $ is greater than $\vert q_+ - \widetilde{\gamma}(0) \vert + \vert q_- - \widetilde{\gamma}(0) \vert$, and
\begin{align*}
\vert q_+ - \widetilde{\gamma}(0) \vert + \vert q_- - \widetilde{\gamma}(0) \vert & \geqslant \vert p_+ - \widetilde{\gamma}(0) \vert + \vert p_- - \widetilde{\gamma}(0) \vert - 2 Hv(2 R_\backsimeq) \\
& \geqslant 2 R_\backsimeq^0 + \vert \widetilde{\gamma}(0) \vert - 2 Hv(2 R_\backsimeq).
\end{align*}
Substitute this in \eqref{eq:upper-bound-on-la-portion} and make all dependences over $\vert \widetilde{\gamma}(0) \vert$ explicit:
\begin{align*}
2 R_\backsimeq^0 + \vert \widetilde{\gamma}(0) \vert - 2 Hv(2 R_\backsimeq) \leqslant & 4 \lambda^2 Hv(2 R_{\backsimeq}) + 2 \lambda v(T) + 16 \lambda^2 \delta. \\
 \leqslant  & 4 \lambda^2 Hv(2 R_{\backsimeq}) + 2 \lambda v(T_0) \\
 & +  2 \lambda v ((\lambda/2) (2 R_{\backsimeq})) + 16 \lambda^2 \delta.
\end{align*}
The last inequality rewrites under the form
\begin{align}
\vert \widetilde{\gamma}(0) \vert & \leqslant \left[ 4 \lambda^2 H + 2 \lambda \left( v \uparrow \lambda \right) \right] \left[ v \uparrow 2 \right] v( R_{\backsimeq}) + 2 \lambda v(T_0) + 16 \lambda^2 \delta + 2 R_{\backsimeq}^0 \notag \\
& \leqslant H_3 v \left( R_{\backsimeq}^0 + \vert \widetilde{\gamma}(0) \vert \right) + \frac{\lambda}{4 \widetilde{H}} r_{1/(8 \lambda \widetilde{H})}(v) + 16 \lambda^2 \delta + 2 R_{\backsimeq}^0,
\label{eq:explicit-R-sqcup}
\end{align}
where $H_3 = \left[ 4 \lambda^2 H + 2 \lambda \left( v \uparrow \lambda \right) \right] \left[ v \uparrow 2 \right]$.
If $\vert \widetilde{\gamma}(0) \vert \geqslant 3 R_{\backsimeq}^0$ then \eqref{eq:explicit-R-sqcup} yields
\begin{equation*}
\vert \widetilde{\gamma}(0) \vert \leqslant 3(v \uparrow 2)H_3v(\vert \widetilde{\gamma}(0) \vert) + \frac{3 \lambda}{4 \widetilde{H}} r_{1/(8 \lambda \widetilde{H})}(v) + 48 \lambda^2 \delta.
\end{equation*}
This inequality would lead to a contradiction for $\vert \widetilde{\gamma}(0) \vert$ larger than
\begin{equation}
R_{\sqcap} := 3 R_{\backsimeq}^0 \vee \left( \frac{3 \lambda}{2 \widetilde{H}} r_{1/(8 \lambda \widetilde{H})}(v) + 96 \lambda^2 \delta \right) \vee r_{1/(6(v \uparrow 2)H_3(\lambda, \delta, v))}(v),
\label{eq:explicit-Rsqcap}
\end{equation}
precisely, if $\vert \widetilde{\gamma}(0) \vert \geqslant R_{\sqcap}$, then $(\eta_- \mid \eta_+)_o \leqslant 3(R_{\backsimeq}^0 + \vert \widetilde{\gamma}(0) \vert) + 4 \delta \leqslant 5 \vert \widetilde{\gamma}(0) \vert$ as $R_{\sqcap} \geqslant R_{\backsimeq}^0 \vee 4 \delta$. One may take $K = 5$. 
\end{proof}

\begin{figure}
\begin{center}
\begin{tikzpicture}[line cap=round,line join=round,>=triangle 45,x=0.7cm,y=0.7cm]
\clip(-5.5,-5.8) rectangle (5.5,0.5);
\draw[fill=black,fill opacity=0.05] (1.31,-3.36) -- (1.26,-3.19) -- (1.08,-3.24) -- (1.14,-3.42) -- cycle;
\draw(0,0) [dash pattern = on 2pt off 2pt] circle (3.5cm);
\draw(0,0) circle (2.1cm);
\draw [shift={(1.67,-5)}] plot[domain=0.64:pi,variable=\t]({1*1.67*cos(\t r)+0*1.67*sin(\t r)},{0*1.67*cos(\t r)+1*1.67*sin(\t r)});
\draw (0,0)-- (0,-5);
\draw (0,0)-- (3,-4);
\draw (0,0)-- (1.14,-3.42);
\draw [line width = 1pt, dash pattern = on 2pt off 2pt] plot[domain =-5:-3, variable = \u]({0.06*(\u-1)*(5+\u)*sin(6*\u r)},{\u});
\draw [line width = 1pt, dash pattern = on 2pt off 2pt] plot[domain =-3:-1.4, variable = \u]({0.06*(\u-1)*(5+\u)*sin(6*\u r)},{\u});
\draw [rotate around={0.643501109 r:(0,0)}] [line width = 1pt, dash pattern = on 2pt off 2pt] plot[domain =-5:-3, variable = \u]({0.05*(\u-1)*(5+\u)*sin(8*\u r)},{\u});
\draw [rotate around={0.643501109 r:(0,0)}] [line width = 1pt, dash pattern = on 2pt off 2pt] plot[domain =-3:-1.4, variable = \u]({0.05*(\u-1)*(5+\u)*sin(8*\u r)},{\u});
\fill [color=black] (0,0) circle (1.5pt) node[left]{$o$} ;
\fill [color=black] (0,-5) circle (1.5pt) node[below]{$\eta_-$} ;
\fill [color=black] (3,-4) circle (1.5pt) node[below]{$\eta_+$} ;
\fill [color=black] (-0.1,-3.15) circle (1.5pt) node[anchor=north east]{$q_-$} ;
\fill [color=black] (1.8,-2.55) circle (1.5pt)  node[below]{$q_+$} ;
\fill [color=black] (-0.5,-1.8) circle (1.5pt) node[left]{$\widetilde{\gamma}(0)$};
\draw (1.7,0) node{$ B(o, 2 R_\backsimeq)$};
\draw (-4,-4) node{$\partial_\infty X$};
\end{tikzpicture}
\caption{Main objects occurring in the proof of Lemma \ref{lem:no-round-trip}.}
\label{fig:no-round-trip}
\end{center}
\end{figure}
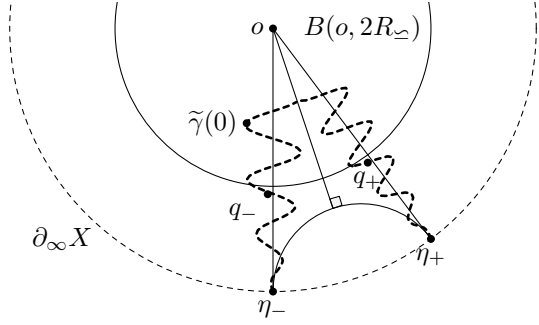

\begin{lemma}[Tracking for $O(u)$-geodesics]
\label{tracking-geodesics}
Let $\delta \in \R_{\geqslant 0}$, $\lambda \in \R_{\geqslant 1}$, let $u$ be an admissible function and let $v = O(u)$ be nondecreasing.
Let $(Y,o)$ be a proper geodesic pointed $\delta$-hyperbolic space, and  let $\widetilde{\gamma} : \R \to Y$ be a $(\lambda, v)$-geodesic.
Define $\widetilde{b}$ as a closest point to $o$ in $\mathrm{im}(\widetilde{\gamma}$). Let $L \in \R_{\geqslant 1}$ be a real constant and assume that the Gromov product $(\partial_\infty \widetilde{\gamma}(+\infty) \mid \partial_\infty \widetilde{\gamma}(-\infty))_o$ is larger than $60 \delta$.
There exists a constant $\widetilde{R} = \widetilde{R}(\lambda, \delta, v, L)$, $H_2$ and $\widetilde{H}_2$ in $\R_{> 0}$ (depending on $\lambda, v$ and $L$) such that if
\begin{equation}
\label{eq:assumption-in-geodesic-tracking}
\widetilde{R} \leqslant \vert \widetilde{\gamma}(0) \vert \leqslant L \vert \widetilde{b} \vert,
\end{equation}
then for any geodesic $\gamma : \R \to Y$ with $[\gamma]_{ \pm \infty} = \partial_\infty \widetilde{\gamma}(\pm \infty)$ and $\gamma(0) = p_\gamma o$,
\begin{align}
\label{eq:first-ineq-tracking-geod}
\forall t \in \R, d(\widetilde{\gamma}(t), {\gamma}) & \leqslant  H_2 v(\vert \widetilde{\gamma}(t) \vert ), \; \mathrm{and} \\
\forall s \in \R, d(\gamma(s), \mathrm{im}(\widetilde{\gamma})) & \leqslant \widetilde{H}_2 v(\vert \gamma(s) \vert ).
\label{eq:second-ineq-tracking-geod}
\end{align}
\end{lemma}

\begin{proof}
As before, write $\eta_\pm = \partial_\infty \widetilde{\gamma}(\pm \infty)$, cut $\widetilde{\gamma}$ in two $(\lambda, v)$-geodesic rays $\widetilde{\gamma}_\pm$ starting at $\widetilde{\gamma}(0)$, and track $\widetilde{\gamma}_\pm$ with geodesic rays $\gamma_\pm$.
Let $\gamma = (\eta_- \eta_+)$. Define $\widetilde{R}_0 = R_{\sqcap}$ and start assuming $\vert \widetilde{\gamma}(0) \vert \geqslant \widetilde{R}_0$.
Let $k$ be a real parameter whose value should be fixed later; only assume for now that $k \geqslant 2K+1$, where $K$ is the constant from Lemma \ref{lem:no-round-trip}.
Define
\begin{equation}
\label{eq:def-of-ppm}
p_\pm := \gamma_{\pm} \left( k \left[ \vert \widetilde{\gamma}(0) \vert + R_{\backsimeq}^0 \right] \vee 2 (2 \lambda +1) r_{1/3 \lambda}(v) \vee r_{1/(2 \widetilde{H})}(v) \right),
\end{equation}
where $R_{\backsimeq}^0$ is the constant from Lemma \ref{lem:tracking-rays} applied to $\widetilde{\gamma}_+$ or $\widetilde{\gamma}_-$ and let $q_\pm$ be a closest point to $p_\pm$ on $\mathrm{im}(\widetilde{\gamma}_\pm)$.
Since $k \geqslant 2K$ and $\vert \widetilde{\gamma}(0) \vert \geqslant R_{\sqcap}$, by Lemma \ref{lem:no-round-trip}, $k \vert \widetilde{\gamma}(0) \vert \geqslant 2 (\eta_- \mid  \eta_+)_o$, and $\vert p_+ \vert = \vert p_- \vert \geqslant 2(\eta_- \mid \eta_+)_o$ as well.
Let $\sigma$ be a geodesic segment from $o$ to $p_{\gamma} o$. By the triangle inequality, $d(p_\pm, \sigma) \geqslant 2(\eta_- \mid \eta_+)_o - \sup_{c \in \sigma} \vert c \vert \geqslant (\eta_- \mid \eta_+)_o - 56 \delta$.
As $(\eta_- \mid \eta_+)_o > 60 \delta$ by hypothesis,
\begin{equation}
d(p_\pm, \sigma) > 60 \delta - 56 \delta = 4 \delta, \, \text{hence}\; d(p_\pm, \gamma) \leqslant 4 \delta,
\label{eq:p-pm-notfar-gamma}
\end{equation}
where we used that the once-ideal triangles $o \eta_\pm (p_\gamma o)$ are $4\delta$-slim.
Further, because $k \geqslant 1$, inequality \eqref{eq:anti-morse-estimate-Ou-ray} of tracking lemma \ref{lem:tracking-rays} allows to bound $\vert q_- - p_- \vert$ and $\vert q_+ - p_+\vert$:
\begin{equation}
\vert p_\pm - q_\pm \vert \leqslant \widetilde{H}v(\vert p_\pm \vert),
\label{ineq:lowboundqplusmin}
\end{equation}
so that by the triangle inequality and the definition \eqref{eq:def-of-ppm} of $p_\pm$,
\begin{equation}
\vert q_\pm \vert \geqslant \vert p_\pm \vert - \widetilde{H}v( \vert p_\pm \vert ) \underset{\eqref{eq:def-of-ppm}}{\geqslant} \frac{1}{2} \vert p_\pm \vert.
\label{ineq:lowboundqplusmin-triangle-ineq}
\end{equation}
At this point, in order to control the quasi-geodesic additive error term of $\widetilde{\gamma}$ between $q_-$ and $q_+$ we need to select $k$ large enough so that $\vert p_\pm \vert \geqslant R_{\ocircle}$, where $R_{\ocircle}$ is associated to $\widetilde{\gamma}$ in lemma \ref{lem:techincal-def-embedding}.
Recall from the expression \eqref{eq:explicit-R-circle} of $R_{\ocircle}$ that 
\[ 
R_{\ocircle} = 4 \vert \widetilde{\gamma} (0) \vert \vee 2 (2 {\lambda} + 1) (3 \lambda \vert \widetilde{\gamma}(0) \vert \vee r_{1/(3 \lambda)} v ).
\]
Thus from now on we fix
$ k = (2K+1) \vee 8 \vee 12 \lambda ( 2 \lambda + 1)  $.
By inequality \eqref{ineq:lowboundqplusmin}, this is sufficient to ensure $\vert q_\pm \vert \geqslant R_{\ocircle}$, and then using the estimates and notations of lemma \ref{lem:techincal-def-embedding}, the portion of $\widetilde{\gamma}$ situated between $q_+$ and $q_-$ is a $(\lambda, c)$-quasigeodesic segment, with $c = \widehat{v}(\vert q_+ \vert \vee \vert q_- \vert)$.
Let $\overline{\gamma}$ be a geodesic segment between $q_+$ and $q_-$.
By Lemma \ref{lem:quantitative-Morse} $\operatorname{dist_H}(\overline{\gamma}, \widetilde{\gamma}_{\mid [t_-, t_+]}) \leqslant (h(\lambda) \vee \widetilde{h}(\lambda))(\delta + c)$, and by hyperbolic geometry, letting $s_\pm$ be such that $\gamma(s_\pm) = p_\gamma(q_\pm)$, $\operatorname{dist_H}(\overline{\gamma}, \gamma_{\mid[s_-, s_+]})$ cannot be much greater than the distance pairwise between the endpoints of these geodesic segments:
\begin{align}
\operatorname{dist_H}(\overline{\gamma}, \gamma_{\mid[s_-, s_+]}) & \leqslant  \vert q_\pm - \gamma(s_\pm) \vert + 8 \delta \notag \\ & \leqslant 4 \delta + \widetilde{H} v(\vert p_\pm \vert) + 8 \delta,
\label{ineq:interpolation-hausdorff}
\end{align}
where we combined \eqref{eq:p-pm-notfar-gamma} and \eqref{ineq:lowboundqplusmin} by means of the triangle inequality. Hence
\begin{align*}
\forall t \in [t_-, t_+],\, d(\widetilde{\gamma}(t), \gamma) & \leqslant (h(\lambda) \vee \widetilde{h}(\lambda))(\delta + (v \uparrow 3 \lambda) v(\vert  q_+ \vert \vee \vert q_- \vert)) \\
& \quad + 8 \delta + 4 \delta + \widetilde{H} v \left(\vert  q_+ \vert \vee \vert q_- \vert\right) \\
& \leqslant \left( 12 +  \widetilde{h}(\lambda) \right)  \delta + (\widetilde{h}(\lambda)(v \uparrow 3 \lambda) + \widetilde{H})  v \left( \vert p_\pm \vert + \widetilde{H} v(\vert p_\pm \vert \right).
\end{align*}
(here $\widetilde{h}(\lambda)$ is used alone as it is equal to $\widetilde{h}(\lambda) \vee h(\lambda)$) If $\vert \widetilde{\gamma}(0) \vert \geqslant R_\sqcap$, then
\begin{align}
\vert p_\pm \vert & = k \left[ \vert \widetilde{\gamma}(0) \vert + R^0_{\backsimeq} \right] \vee 2 (2 \lambda +1) r_{1/3 \lambda}(v) \vee r_{1/(2 \widetilde{H})}(v) \notag \\ &
\leqslant \left( k + \frac{R^0_{\backsimeq}}{R_\sqcap} \right) \vert \widetilde{\gamma}(0) \vert \vee 2 (2 \lambda +1) r_{1/3 \lambda}(v) \vee r_{1/(2 \widetilde{H})}(v) \notag \\
& \underset{\eqref{eq:explicit-Rsqcap}}{\leqslant} (2k \vert \widetilde{\gamma}(0) \vert) \vee 2 (2 \lambda +1) r_{1/3 \lambda}(v) \vee r_{1/(2 \widetilde{H})}(v) 
\label{ineq:bound-on-ppm}
\end{align}
where we used $k \geqslant 1$ so that $k + 1/3 \leqslant 2k$ in the last inequality.
Define $\widetilde{R}_1 = \widetilde{R}_0 \vee 2 (2 \lambda +1) r_{1/3 \lambda}(v) \vee r_{1/(2 \widetilde{H})}(v)$. 
Then if $\vert \widetilde{\gamma}(0) \vert \geqslant \widetilde{R}_1$, by \eqref{ineq:bound-on-ppm},
\[
v( \vert p_\pm \vert + \widetilde{H} v(\vert p_\pm \vert) \leqslant v \left( \frac{3}{2} \vert p_\pm \vert \right) \leqslant 3k \vert \widetilde{\gamma}(0) \vert
\]
Let $t_\pm \in \R$ be such that $\widetilde{\gamma}(t_\pm) = \widetilde{\gamma}_{\pm}(\pm t_\pm) = q_\pm$.
Using the right-hand side of assumption \eqref{eq:assumption-in-geodesic-tracking} that $\vert \widetilde{\gamma}(0) \vert \leqslant L \vert \widetilde{b} \vert = L \inf \lbrace \vert \widetilde{\gamma}(t) \vert : t \in \R \rbrace $ and plugging \eqref{ineq:bound-on-ppm} in the previous inequality, one obtains that for all $t$ in $[t_-, t_+]$,
\begin{align*}
d(\widetilde{\gamma}(t), \gamma) & \leqslant (12 + \widetilde{h}(\lambda)) \delta + (\widetilde{h}(\lambda) + (v \uparrow 3 \lambda) \widetilde{H}) v( 3k \vert \widetilde{\gamma}(0) \vert)  \\
& \leqslant (12 + \widetilde{h}(\lambda)) \delta + h_2^0(v, \lambda) v( 3Lk \vert \widetilde{\gamma}(t) \vert),
\end{align*} 
where $h_2^0(v, \lambda) = (\widetilde{h}(\lambda) + (v \uparrow 3 \lambda) \widetilde{H})$.
Define $\widetilde{R}_3 = \widetilde{R}_2 \vee L^{-1} \sup \lbrace r : v(r) \leqslant (12 + \widetilde{h}(\lambda)) \delta \rbrace$.
If $\vert \widetilde{\gamma}(0) \vert \geqslant \widetilde{R}_3$,
the right-hand side of assumption \eqref{eq:assumption-in-geodesic-tracking} ensures that $v( \vert \widetilde{\gamma}(t) \vert) \geqslant (12 + \widetilde{h}(\lambda)) \delta $ for all $t$, so we have proved
\begin{equation}
\forall t \in [t_-, t_+],\, d(\widetilde{\gamma}(t), \gamma) \leqslant ( 1 + (v \uparrow 3Lk) h_2^0(v,\lambda))v(\vert \widetilde{\gamma}(t) \vert).
\label{ineq:track-geodesic-near}
\end{equation}
On the other hand, in view of the tracking lemma \ref{lem:tracking-rays}, for all $t \in (- \infty, t_-)$, $d(\widetilde{\gamma}(t), \gamma_-) \leqslant (v \uparrow 3 \lambda) H v(\vert \widetilde{\gamma}(t) \vert)$ and similarly for all $t \in (t_+, +\infty), d(\widetilde{\gamma}(t), \gamma_+) \leqslant (v \uparrow 3 \lambda)v(\vert \widetilde{\gamma}(t) \vert)$. Since the twice-ideal triangle $o\eta_-\eta_+$ is $4 \delta$-slim, using the triangle inequality and the fact that $v(\vert \widetilde{\gamma}(t) \vert) \geqslant 12 \delta$ for all $t$ provided $\vert \widetilde{\gamma}(0) \vert \geqslant \widetilde{R}_3$ by definition of $\widetilde{R}_3$,
\begin{equation}
\forall t \in \R \setminus [t_-, t_+],\,  d(\widetilde{\gamma}(t), \gamma) \leqslant ((v \uparrow 3 \lambda) H + 4\delta/(12 \delta))v(\vert \widetilde{\gamma}(t) \vert).
\label{ineq:track-geodesic-far}
\end{equation}
Putting \eqref{ineq:track-geodesic-near} and \eqref{ineq:track-geodesic-far} together yields the expected tracking inequality \eqref{eq:first-ineq-tracking-geod} for the provisional $\widetilde{R}_3$.
Precisely ${H}_2$ may then be taken as 
\begin{equation}
H_2  = 2 (v \uparrow 3 Lk)(\widetilde{h}(\lambda) + ( v \uparrow 3 \lambda) \widetilde{H}) \vee 2 ( v \uparrow 3 \lambda)H.
\label{eq:explicitH2}
\end{equation}
From here, one could deduce \eqref{eq:second-ineq-tracking-geod} using \eqref{eq:first-ineq-tracking-geod} for $\widetilde{R}$ large enough by a metric connectedness argument as in Lemma \ref{lem:quantitative-Morse} or Lemma \ref{lem:tracking-rays}, but let us rather use the former estimates from the current proof. Define $\widetilde{R}_4 = \widetilde{R}_3 \vee r_{1/(2 L H_2)}(v)$ ; then if $\vert \widetilde{\gamma}(0) \vert \geqslant \widetilde{R}_4$, it follows from the tracking inequality just obtained for $\widetilde{\gamma}$ that $\vert p_\gamma o \vert = \vert {\gamma} (0) \vert \geqslant (1/2) \vert \widetilde{\gamma}(0) \vert$. Then for all $s \in [s_-, s_+]$, by \eqref{ineq:interpolation-hausdorff},
\begin{equation}
d(\gamma(s), \overline{\gamma}) \leqslant \widetilde{H} v (\vert p_\pm \vert) + 12 \delta \leqslant 2 \widetilde{H} v(2k \vert \gamma (s) \vert) \vee 24 \delta.
\label{ineq:bound-on-dist-gamma-gammabar}
\end{equation}
On the other hand, recall that by Lemma \ref{lem:quantitative-Morse}, for all $c \in \overline{\gamma}$, $d(c, \widetilde{\gamma} \leqslant \widetilde{h}(\lambda)( \delta + (v \uparrow 3 \lambda)v ( \vert q_+ \vert \vee \vert q_- \vert)$. Combining this with \eqref{ineq:bound-on-dist-gamma-gammabar} by means of the triangle inequality while remembering the bound on $\vert q_\pm \vert$ implied by \eqref{ineq:lowboundqplusmin}, one obtains
\begin{equation}
\label{ineq:control-dist-gamma-close}
 \forall s \in [s_-, s_+], \, d(\gamma(s), \widetilde{\gamma}) \leqslant (2 \widetilde{H} +  \widetilde{h}(\lambda)( \delta + (v \uparrow 3 \lambda)) v ( 3k \vert \gamma (s) \vert).
\end{equation}
Finally, if $s \in \R$ is such that $s \leqslant s_-$ or $s \geqslant s_+$, since $o$, $p_\gamma o$ and $\gamma(s)$ are $28 \delta$-almost lined up, $\vert \gamma(s) \vert \geqslant \vert s \vert - \vert p_\gamma o \vert \geqslant \vert s \vert /2$. $\gamma(s)$ is at most $4\delta$ away from its orthogonal projection on $\gamma_{\epsilon(s)}$, where $\epsilon(s)$ is the sign of $s$. Given the definition of $p_\pm$, $p_{\gamma_{\epsilon(s)}} \gamma(s)$ is at a distance at least $R_{\backsimeq}$ from the origin, and inequality \eqref{eq:anti-morse-estimate-Ou-ray} from Lemma \ref{lem:tracking-rays} bounds its distance to $\widetilde{\gamma}$ so that
\begin{align*}
 d(\gamma(s), \widetilde{\gamma}) & \leqslant \vert \gamma(s) -  p_{\gamma_{\epsilon(s)}} \gamma(s) \vert +  d(  p_{\gamma_{\epsilon(s)}} \gamma(s), \widetilde{\gamma}) \\
 & \leqslant 4 \delta + \widetilde{H}v( \vert p_{\gamma_{\epsilon(s)}} \gamma(s) \vert) \leqslant 2 \widetilde{H} v (\vert \gamma (s) \vert).
\end{align*}
Together with \eqref{ineq:control-dist-gamma-close}, this proves \eqref{eq:second-ineq-tracking-geod} with $\widetilde{R} = \widetilde{R}_4$ and
\begin{equation}
\label{eq:def-of-Htilde-2}
\widetilde{H}_2 = \left( 2 \widetilde{H} + \widetilde{h}(\lambda) \right) \left( \delta + ( v \uparrow 3 \lambda) \right) ( v \uparrow 3 k ). \qedhere
\end{equation}
\end{proof}

\subsection{Distance between $O(u)$-geodesics}
\label{subsec:distance-btw-geodesics}

\begin{lemma}
\label{lem:on-linear-divergence}
Let $\delta \in \R_{\geqslant 0}$ be a constant.
Let $\gamma_1$ and $\gamma_2$ be geodesic lines into a $\delta$-hyperbolic space, with four pairwise distinct endpoints. Define $\Delta = d(\mathrm{im}(\gamma_1), \mathrm{im}(\gamma_2))$. Then for all $s_1, s_2 \in \R$,
\begin{equation}
\vert \gamma_1 (s_1) - \gamma_2(s_2) \vert \geqslant \Delta + d(\gamma_1(s_1), p_{\gamma_1} \mathrm{im}(\gamma_2)) \vee d(\gamma_1(s_1), p_{\gamma_1} \mathrm{im}(\gamma_2)) - 56 \delta.
\label{ineq:lower-bound-linear-div}
\end{equation}
\end{lemma}

\begin{proof}
The distance on the left is symmetric relatively to $\gamma_i(s_i)$, so it suffices to prove $\vert \gamma_1 (s_1) - \gamma_2(s_2) \vert \geqslant \Delta + d(\gamma_1(s_1), p_{\gamma_1} \mathrm{im}(\gamma_2)) - 56 \delta$.
The points $\gamma_1(s_1), p_{\gamma_2}(\gamma_1(s_1))$ and $\gamma_2(s_2)$ are the vertices of a right-angled hyperbolic triangle so that by Lemma \ref{lem:triangle-rectangle-hyperbolique}, they are $28 \delta$-almost lined up. By the triangle inequality,
\begin{align*}
d(\gamma_1(s_1), \gamma_2(s_2)) + 2 \cdot 28 \delta & \geqslant d(\gamma_1(s_1), p_{\gamma_2}(\gamma_1(s_1))) + d(p_{\gamma_2}(\gamma_1(s_1)), \gamma_2(s_2) ) \\
& \geqslant \Delta + d(\gamma_1(s_1), p_{\gamma_1} \mathrm{im}(\gamma_2)). \qedhere
\end{align*}
\end{proof}

\begin{figure}
\begin{center}
\begin{tikzpicture}[line cap=round,line join=round,>=triangle 45,x=3.0cm,y=3.0cm]
\clip(-2,-1) rectangle (2,1);
\draw (-1.2,0.8)--(-0.4,0.0)--(-0.4,-0.0)--(-1.2,-0.8);
\draw (1.2,0.8)--(0.4,0.0)--(0.5,-0.1)--(1.2,-0.8);
\draw [line width = 0.5pt, dash pattern = on 2pt off 2pt] plot[domain=-1:1, variable=\y] ({-1.2+0.3*(\y-0.8)*(\y+0.8)}, {\y});
\draw [line width = 0.5pt, dash pattern = on 2pt off 2pt] plot[domain=-1:1, variable=\y] ({1.2-0.3*(\y-0.8)*(\y+0.8)}, {\y});
\draw [line width = 1pt, dash pattern = on 2pt off 0pt] plot[domain=-0.8:0, variable=\y] ({-1.2-\y*\y+0.64+0.2*(\y - 0.8)*(\y+0.8)*sin(1000* \y)}, {\y});
\draw [line width = 1pt, dash pattern = on 2pt off 0pt]  plot[domain=0:0.8, variable=\y] ({-1.2-\y*\y+0.64+0.2*(\y - 0.8)*(\y+0.8)*sin(1000* \y)}, {\y});
\draw [line width = 1pt, dash pattern = on 2pt off 0pt]  plot[domain=-0.8:0, variable=\y] ({1.2+\y*\y-0.64-0.2*(\y - 0.8)*(\y+0.8)*cos(1000* \y)}, {\y});
\draw [line width = 1pt, dash pattern = on 2pt off 0pt]  plot[domain=0:0.8, variable=\y] ({1.2+\y*\y-0.64-0.2*(\y - 0.8)*(\y+0.8)*cos(1000* \y)}, {\y});
\fill (-1.2,0.8) circle(1.5pt) node[anchor=east]{$\eta_1^+$};
\fill (-1.2,-0.8) circle(1.5pt) node[anchor=east]{$\eta_1^-$};
\fill (1.2,0.8) circle(1.5pt) node[anchor=west]{$\eta_2^+$};
\fill (1.2,-0.8) circle(1.5pt) node[anchor=west]{$\eta_2^-$};
\fill (-0.4,0.0) circle(1.5pt) node[anchor=south west]{$\gamma_1(s_1)$};
\fill (0.4,0.0) circle(1.5pt) node[anchor=south east]{$\gamma_2(s_2)$};
\draw (-0.4,0.0)--(0.4,-0.0);
\fill (-0.45,-0.1) circle(1.5pt) node[anchor=north west]{$\widetilde{\gamma}_1(t_1)$};
\fill (0.47,-0.17) circle(1.5pt) node[anchor=north east]{$\widetilde{\gamma}_2(t_2)$};
\draw (-0.45,-0.1)--(0.47,-0.17);
\fill (0,0.6) circle(1.5pt) node[anchor=south]{$o$};
\draw (-0.5,0.1)--(0,0.6);
\draw[fill=black,fill opacity=0.05] (-0.5,0.1) -- (-0.53, 0.13) -- (-0.5,0.16) -- (-0.47,0.13) -- cycle;
\draw[fill=black,fill opacity=0.05] (0.5,0.1) -- (0.47, 0.13) -- (0.44,0.1) -- (0.47,0.07) -- cycle;
\draw (0.5,0.1)--(0,0.6);
\end{tikzpicture}
\caption{Main points occuring in the proof of Lemma \ref{lemma:estimate-distance-quasigeodesics}. Straight, resp.\ wavy lines depict geodesic, resp. $O(u)$-geodesic lines; boundary is dashed.}
\label{fig:dist-btw-Ou-geodesics}
\end{center}
\end{figure}
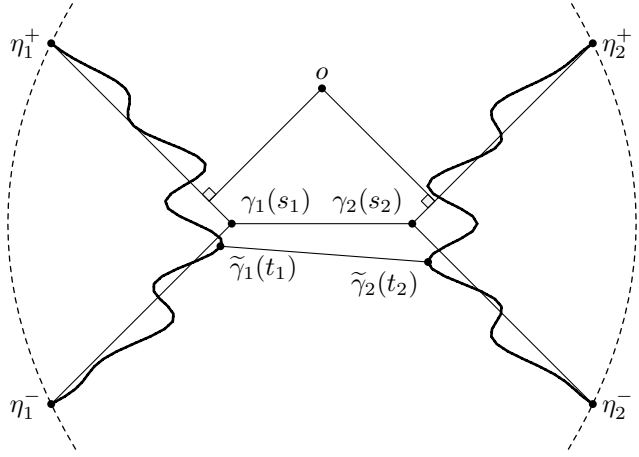

\begin{lemma}
\label{lemma:estimate-distance-quasigeodesics}
Let $v^1$ and $v^2$ be admissible functions, and define $v = v^1 \vee v^2$. Let $L \in \R_{>1}$ be a constant.
Let $\delta$ be a hyperbolicity constant and let $\lambda=(\underline{\lambda}, \overline{\lambda})\in \R_{>0}^2$ be expansion and Lipschitz constants. 
There exist $J=J( \lambda, v, L)$, $R = R(\delta, \lambda,v, L)$ and, for $i \in \lbrace 1, 2 \rbrace$, $\widetilde{R}^i = \widetilde{R}^i(\delta, \lambda, v^i, L)$ in $\R_{> 0}$ such that for any $\delta$-hyperbolic, proper geodesic, pointed hyperbolic space $(Y,o)$, if $(\gamma_1, \widetilde{\gamma}_1)$ and $(\gamma_2, \widetilde{\gamma}_2)$ are such that
\renewcommand{\theenumi}{\roman{enumi}}
\begin{enumerate}
\item{
\label{item:gamma-i-are-geodesics}
$\gamma_1$, $\gamma_2$ are geodesics $\R \to Y$ with four distinct endpoints $\eta_i^\pm = \gamma_i(\pm \infty)$,}
\item{
\label{item:gammatilde-i-are-Ougeodesics}
for $i \in \lbrace 1, 2 \rbrace$, $\widetilde{\gamma}_i$, is a $(\lambda, v^i)$-geodesics $\R \to Y$,}
\item{
\label{item:asymptotic-assumption}
for $i \in \lbrace 1,2 \rbrace$, $\partial_\infty \widetilde{\gamma}_i(\pm \infty) = [\gamma_i]_\pm$,
}
\item{
\label{item:conditions-on-gromprods}
$\underline{\boxtimes} \left\{ {\eta_1^\pm},{\eta_2^\pm} \right\} \geqslant 60 \delta$, and $\overline{\boxtimes} \left\{ {\eta_1^\pm},{\eta_2^\pm} \right\} \geqslant R$,}
\item{
\label{item:proximal-assumption}
for all $i \in \lbrace 1,2 \rbrace$, $\widetilde{R}^i \leqslant \vert \widetilde{\gamma}_i(0) \vert \leqslant L \inf_{t \in \R} \vert \widetilde{\gamma}_i(t) \vert$,}
\end{enumerate}
then
\begin{equation}
\vert d(\gamma_1, \gamma_2) - d (\widetilde{\gamma}_1, \widetilde{\gamma}_2) \vert \leqslant J v \left( \overline{\boxtimes} \left\{ {\eta_1^\pm},{\eta_2^\pm} \right\} \right).
\label{eq:distance-Ou-geodesics}
\end{equation}
\end{lemma}

\begin{proof}[Sketch of proof for Lemma \ref{lemma:estimate-distance-quasigeodesics}]
\renewcommand{\qedsymbol}{}
See figure \ref{fig:dist-btw-Ou-geodesics}.
The main tool is the geodesic tracking lemma; however the tracking between $\widetilde{\gamma}_i$ and $\gamma_i$ becomes inefficient far from the origin. Thus we need prove that shortest geodesic segments between $\gamma_1$ and $\gamma_2$ on the one hand, and between $\widetilde{\gamma}_1$ and $\widetilde{\gamma}_2$ on the other hand, are close to the origin (at most not significantly farther than the largest Gromov product). The part concerning $\gamma_1$ and $\gamma_2$ was already expressed by Lemma \ref{lem:controle-projetes-geodesiques}; as for the other part we show (inequality \eqref{eq:perp-btw-quasigeodeiscs-is-high}) that letting $t_1$, $t_2$ be such that $d(\widetilde{\gamma}_1, \widetilde{\gamma}_2) = \vert \widetilde{\gamma}_1(t_1) - \widetilde{\gamma}_2(t_2) \vert$,
$\vert \widetilde{\gamma}_1(t_1) \vert \vee \vert \widetilde{\gamma}_2(t_2)\vert$ is linearly controlled by $\overline{\boxtimes} \lbrace \eta_1^\pm, \eta_2^\pm \rbrace$ on the large-scale.
This uses the well-known behavior described by Lemma \ref{lem:on-linear-divergence}: geodesic rays spread apart linearly from each other after the Gromov products are reached; since they track $O(u)$-geodesics at a distance growing sublinearly, $\widetilde{\gamma}_1$ and $\widetilde{\gamma}_2$ also spread away from each other, which prevents $\widetilde{\gamma}_i(t_i)$ from being much farther than all the Gromov products.
\end{proof}

\begin{proof}
For $i \in \lbrace 1, 2 \rbrace$, let $s_i \in \R$ be such that $\vert \gamma_1(s_1) - \gamma_2(s_2) \vert = d(\gamma_1, \gamma_2)$.
As $\gamma_1(s_1) \in p_{\gamma_1}(\gamma_{2})$, and similarly $\gamma_2(s_2) \in p_{\gamma_2}(\gamma_{1})$, by the projection lemma \ref{lem:controle-projetes-geodesiques}, $\sup_i \vert \gamma_i(s_i) \vert \leqslant \overline{\boxtimes} \left\{ {\eta_1^\pm},{\eta_2^\pm} \right\} + 284 \delta$.
Further, set $\widetilde{R}^i = \widetilde{R}^i(\lambda, \delta, v^i, L)$ according to the tracking lemma \ref{tracking-geodesics}. Note that by the assumptions \eqref{item:gamma-i-are-geodesics} to \eqref{item:asymptotic-assumption}, the first inequality in assumption \eqref{item:conditions-on-gromprods} and the right-hand side inequality in assumption \eqref{item:proximal-assumption}, applied to the pairs $(\gamma_i, \widetilde{\gamma}_i)$, by Lemma \ref{tracking-geodesics},
\begin{align*}
\forall i \in \lbrace 1, 2 \rbrace, \,
d(\gamma_i (s_i), \widetilde{\gamma}_i) & \leqslant \widetilde{H}_2 v(\vert \gamma_i (s_i) \vert) \\
& \leqslant \widetilde{H}_2 v( \overline{\boxtimes} \left\{ {\eta_1^\pm},{\eta_2^\pm} \right\} + 284 \delta).
\end{align*}
By the triangle inequality, setting $J^+ = 2\widetilde{H}_2(v \uparrow 2)$ and $R_0 = \sup \lbrace r : v(r) \leqslant 284 \delta \rbrace$, as soon as $\overline{\boxtimes} \left\{ {\eta_1^\pm},{\eta_2^\pm} \right\} \geqslant R_0$,
\begin{equation}
d(\widetilde{\gamma}_1, \widetilde{\gamma}_2) - d(\gamma_1, \gamma_2) \leqslant d(\gamma_1 (s_1), \widetilde{\gamma}_1) + d(\gamma_2 (s_2), \widetilde{\gamma}_2) \leqslant J^+ v \left( \overline{\boxtimes} \left\{ {\eta_1^\pm},{\eta_2^\pm} \right\} \right).
\label{eq:first-half-of-ineq-dist-between-Ou-geodesics}
\end{equation}
This is one half of inequality \eqref{eq:distance-Ou-geodesics}.

For $i \in \lbrace 1, 2 \rbrace$ let $t_i \in \R$ be such that $d(\widetilde{\gamma}_1, \widetilde{\gamma}_2) = \vert \widetilde{\gamma}_1 (t_1) - \widetilde{\gamma_2}(t_2) \vert$. Let $\widetilde{s}_i$ be such that $\gamma_i(\widetilde{s}_i) = p_{\gamma_i} \widetilde{\gamma}_i(t_i)$. By the triangle inequality and tracking lemma \ref{tracking-geodesics},
\begin{equation}
 \vert \gamma_1(\widetilde{s}_1) - \gamma_2(\widetilde{s}_2) \vert \leqslant d(\widetilde{\gamma}_1, \widetilde{\gamma}_2) + 2 H_2 v( \vert \widetilde{\gamma}_1(t_1) \vert \vee \vert \widetilde{\gamma}_2(t_2) \vert ).
\label{eq:opposite-ineq-start}
\end{equation}
Inequality \eqref{ineq:lower-bound-linear-div} of Lemma \ref{lem:on-linear-divergence} gives a lower bound on $\vert \gamma_1(\widetilde{s}_1) - \gamma_2(\widetilde{s}_2) \vert$, which can be plugged into \eqref{eq:opposite-ineq-start} yielding
\begin{align}
d(\gamma_1, \gamma_2) + & d \left( \gamma_1(\widetilde{s}_1), p_{\gamma_1} \mathrm{im}(\gamma_2) \right) \vee d \left( \gamma_2(\widetilde{s}_2), p_{\gamma_2} \mathrm{im}(\gamma_1) \right) \notag \\
& \leqslant d(\widetilde{\gamma}_1, \widetilde{\gamma}_2) + 2H_2 v( \vert \widetilde{\gamma}_1(t_1) \vert \vee \vert \widetilde{\gamma}_2(t_2) \vert ) + 56\delta.
\label{eq:dist-gamma_i-stilde-i}
\end{align}
On the other hand, using twice the triangle inequality and Lemma \ref{lem:controle-projetes-geodesiques},
\begin{align*}
d \left( \gamma_1(\widetilde{s}_1), p_{\gamma_1} \mathrm{im}(\gamma_2) \right) \vee d \left( \gamma_2(\widetilde{s}_2), p_{\gamma_2} \mathrm{im}(\gamma_1) \right)
& \geqslant \vert \gamma_1(\widetilde{s}_1) \vert \vee \vert \gamma_1(\widetilde{s}_1) \vert \\ & \quad - \overline{\boxtimes} \left\{ {\eta_1^\pm},{\eta_2^\pm} \right\} - 284 \delta \\
& \geqslant  \vert \widetilde{\gamma}_1(t_1) \vert \vee \vert \widetilde{\gamma}_2(t_2) \vert - \overline{\boxtimes} \left\{ {\eta_1^\pm},{\eta_2^\pm} \right\} \\ & \quad - H_2 v( \vert \widetilde{\gamma}_1(t_1) \vert \vee \vert \widetilde{\gamma}_2(t_2) \vert ) - 284 \delta.
\end{align*}
Reorganizing \eqref{eq:dist-gamma_i-stilde-i},
\begin{align*}
\vert \widetilde{\gamma}_1(t_1) \vert \vee \vert \widetilde{\gamma}_2(t_2) \vert
& \leqslant
\overline{\boxtimes} \lbrace \eta_1^\pm, \eta_2^\pm \rbrace
+ 340 \delta + 3H_2 v( \vert \widetilde{\gamma}_1(t_1) \vert \vee \vert \widetilde{\gamma}_2(t_2) \vert ) \\
& \quad + d(\widetilde{\gamma}_1, \widetilde{\gamma}_2) - d(\gamma_1, \gamma_2) \\
& \underset{\eqref{eq:first-half-of-ineq-dist-between-Ou-geodesics}}{\leqslant} \overline{\boxtimes} \lbrace \eta_1^\pm, \eta_2^\pm \rbrace
+ 340 \delta + 3H_2 v( \vert \widetilde{\gamma}_1(t_1) \vert \vee \vert \widetilde{\gamma}_2(t_2) \vert ) \\
& \quad + J^+ v \left( \overline{\boxtimes} \lbrace \eta_1^\pm, \eta_2^\pm \rbrace  \right)
\end{align*}
when $\overline{\boxtimes} \left\{ {\eta_1^\pm},{\eta_2^\pm} \right\} \geqslant F_0$. Hence,
\begin{align*}
\vert \widetilde{\gamma}_1(t_1) \vert \vee \vert \widetilde{\gamma}_2(t_2) \vert
& \leqslant \inf \left\{ r_{1/(6H_2)}(v), \right. \\
& \left. \quad 2 \left[ \overline{\boxtimes} \lbrace \eta_1^\pm, \eta_2^\pm \rbrace
+ 340 \delta + J^+ v \left( \overline{\boxtimes} \lbrace \eta_1^\pm, \eta_2^\pm \rbrace  \right) \right] \right\}.
\end{align*}
Set $R_1 = \sup \lbrace r : v(r) \geqslant 584 \delta / J^+ \rbrace$ and $R_2 = \sup \lbrace R_0, R_1, r_{1/(2J^+)}(v) \rbrace$. Then if $\overline{\boxtimes} \left\{ {\eta_1^\pm},{\eta_2^\pm} \right\} \geqslant R_2$,
\begin{equation}
\vert \widetilde{\gamma}_1(t_1) \vert \vee \vert \widetilde{\gamma}_2(t_2) \vert
 \leqslant \inf \left\{ r_{1/(6H_2)}(v), 4 \overline{\boxtimes} \lbrace \eta_1^\pm, \eta_2^\pm \rbrace \right\}.
\label{eq:perp-btw-quasigeodeiscs-is-high}
\end{equation}
Thus if $R_{3} =  r_{1/(4H_2)}(v)$, and if $t_1, t_2 \in \R$ are such that $\vert \widetilde{\gamma}_1(t_1) \vert \vee \vert \widetilde{\gamma}_2(t_2) \vert \geqslant R_{3}$, then
\begin{equation*}
H_2v(\widetilde{\gamma}_1(t_1) \vert \vee \vert \widetilde{\gamma}_2(t_2) \vert) \leqslant H_2(v \uparrow 4) v \left( \overline{\boxtimes} \lbrace \eta_1^\pm, \eta_2^\pm \rbrace \right).
\end{equation*}
Finally by the triangle inequality, writing $J^- = 2 H_2(v \uparrow 4)$,
\begin{align}
d (\gamma_1, \gamma_2) - d (\widetilde{\gamma}_1, \widetilde{\gamma}_2) & \leqslant  d(\gamma_1(\widetilde{s}_1),\gamma_2(\widetilde s_2)) - d(\widetilde{\gamma}_1(t_1), \widetilde{\gamma}_2(t_2)) \notag \\ & \leqslant J^-  v \left( \overline{\boxtimes} \lbrace \eta_1^\pm, \eta_2^\pm \rbrace \right).
\label{eq:second-half}
\end{align}
To reach the conclusion of Lemma \ref{lemma:estimate-distance-quasigeodesics}, define $J = J^- \vee J^+$ and then combine \eqref{eq:first-half-of-ineq-dist-between-Ou-geodesics} with \eqref{eq:second-half}.
\end{proof}

\subsection{Tracking radii}
\label{subsec:tracking-radii}
While there are four relevant parameters ($\lambda$, $v$, $\delta$, $L$) to express $R_{\backsimeq}^0$, $R_{\sqcap}$, $\widetilde{R}$ and $R$, only the dependence on $v$ is of interest for what follows. Consequently, a constant depending on the remaining parameters $\lambda, \delta, L$ can be written as, e.g., $C(\lambda, \delta)$ or $C(\lambda, \delta, L)$. 

\begin{lemma}
\label{lemma:tracking-radii-grow-sublinearly}
Let $v$ be an admissible function. Let $\lambda \in \R_{\geqslant 1}$ be a biLipschitz constant. Let $\delta$ be a hyperbolicity constant. There exist a positive integer $n$ and constants $C(\lambda)$, $C(\lambda, \delta)$, $C(\lambda, \delta, L)$ such that in Lemma \ref{lem:tracking-rays}, Lemma \ref{lem:no-round-trip}, Lemma \ref{tracking-geodesics} and Lemma \ref{lemma:estimate-distance-quasigeodesics}, tracking radii may be taken as
\begin{align}
R_{\backsimeq}^0 & = r_{C(\lambda) (v \uparrow 1+ \lambda)^{-n}}(v) \vee C(\lambda, \delta) \left( 1 + \sup \left\{ r: v(r) \leqslant C(\lambda, \delta) \right\} \right)
\label{eq:explicitR_backsimeq}
\\
\widetilde{R} & = C(\lambda, \delta)  r_{C(\lambda) (v \uparrow L)^{-1} (v \uparrow 1+ \lambda)^{-n}}(v) \vee\left( 1 + \sup \left\{ r: v(r) \leqslant C(\lambda, \delta, L) \right\} \right)
\label{eq:explicitR_tilde}
\\
R & = r_{C(\lambda, \delta) (v \uparrow 1+ \lambda)^{-n}}(v) \vee C(\lambda, \delta) \left( 1 + \sup \left\{ r: v(r) \leqslant C(\lambda, \delta, L) \right\} \right)
\label{eq:explicitR}
\end{align}
\end{lemma}

\begin{proof}It will be used without further notice that
$r_{\alpha}(v) \vee r_{\beta}(v) = r_{\alpha \wedge \beta}(v)$, for all $\alpha, \beta \in \R_{>0}$, and that $\lambda \geqslant 1$, especially $1/\lambda \leqslant \lambda \leqslant \lambda^2$. The bounds we obtain need not be excessively precise, and we allow losing multiplicative factors frequently.
Start with {\eqref{eq:explicitR_backsimeq}}, and notation as in the proof of Lemma \ref{lem:tracking-rays}.
\begin{align}
t_1 & = r_{\lambda}(v) \vee r_{1/(3\lambda)}(v) \vee 3 \lambda r_{1/(2 h(\lambda) v \uparrow 1 + 8 \lambda^2)}(v) \vee \sup \lbrace s : v(s) \leqslant \delta \rbrace \notag \\
& \quad \vee 12 \lambda \delta \vee (4 \lambda M'(\underline{\lambda}, \delta) +1) \notag \\
& \leqslant 3 \lambda r_{1/(2 h(\lambda)v \uparrow 1 + 8 \lambda^2)}(v) \vee C(\lambda, \delta).
\label{eq:bound-on-t1}
\end{align}
Next, $t_1 \geqslant t_0$ by definition, so that $t_2$ can be defined as
\begin{equation*}
t_2 = t_1 \vee (1+8 \lambda^2)^{-1} T_2 = t_1 \vee (1+8 \lambda^2)^{-1} \sup \lbrace r : v(r) \leqslant 6 \lambda^2 \delta \rbrace.
\end{equation*}
Then, $t_3 = t_2 \vee \sup \lbrace r : v(r) \leqslant h(\lambda) \delta \rbrace = t_1 \vee \sup \lbrace r : v(r) \leqslant h(\lambda) \delta \rbrace$ since $h(\lambda) \delta \geqslant 6 \lambda^2 \delta$. After that,
$t \geqslant t_4 := \sup \left\{ t_3,  r_{\underline{\lambda}/(2 + 2 H_0(\lambda,v))}(v) \right\}$, where $H_0(\lambda, v) = 2 h(\lambda) (v \uparrow 1 + 8 \lambda^2) +1$. From this and \eqref{eq:bound-on-t1} we deduce
\begin{equation}
t_4 \leqslant r_{1/(3 \lambda H_0)} (v) \vee \sup \lbrace r : v(r) \leqslant h(\lambda) \delta \rbrace \vee C(\lambda, \delta)
\end{equation}
and then
\begin{align*}
t_5 & \leqslant t_4 \vee \sup \lbrace r : v(r) \leqslant 8 \delta \rbrace \\ & \leqslant
r_{1/(3 \lambda H_0)} (v) \vee \sup \lbrace r : v(r) \leqslant 8 h(\lambda) \delta \rbrace \vee C(\lambda, \delta);
\end{align*}
\begin{align*}
t_{\backsimeq^0} & \leqslant t_5 \vee 16 \delta = t_5 \vee C(\lambda, \delta); \\
t_6 & = t_{\backsimeq^0} \vee \sup \lbrace r : v(r) \leqslant 8 \delta \rbrace \\ & \leqslant r_{1/(3 \lambda H_0)} (v) \vee \sup \lbrace r : v(r) \leqslant 16 h(\lambda) \delta \rbrace \vee C(\lambda, \delta).
\end{align*}
As $t_7 = t_6 \vee t_{\ocircle}$, $t_{\ocircle} = r_{1/(3\lambda)}(v)$ and $H_0 \geqslant 1$, the same bound applies to $t_7$. Next,
\begin{equation}
t_8 = t_7 \vee r_{\underline{\lambda}/(6H)}(v) \leqslant r_{\underline{\lambda}/(6H)}(v) \vee \sup \lbrace r : v(r) \leqslant 16 h(\lambda) \delta \rbrace \vee C(\lambda, \delta)
\label{bound-on-t8}
\end{equation}
(remember that $H = 1 + H_0$ by definition). Thus
\begin{align}
R_{\backsimeq}^0 & = t_8/(6 \lambda) \vee \sup \left\{ r : 2(v \uparrow 6 \lambda) (H_0 +1) v(r) \leqslant 8 \lambda \right\} \notag \\
& \leqslant t_8 \vee \sup \left\{ r : 2 (H_0 +1) v(6 \lambda r) \leqslant 8 \lambda \right\} \notag \\
& \leqslant t_8 \vee \sup \left\{ r : 2 \left(2h(\lambda) (v \uparrow 1 + 8 \lambda^2) +1 \right) v(6 \lambda r) \leqslant 8 \lambda \right\} \notag \\
& \leqslant t_8 \vee \sup \left\{ r : 6 h(\lambda)  v(6 \lambda (1+ 8 \lambda^2) r) \leqslant 8 \lambda \right\} \notag \\
& \underset{\eqref{bound-on-t8}}{\leqslant} r_{1/(6 \lambda H)}(v) \vee C(\lambda, \delta) \left( 1 + \sup \left\{ r: v(r) \leqslant C(\lambda, \delta) \right\} \right). 
\label{eq:precise-bound-R_backsimeq}
\end{align}
This inequality implies \eqref{eq:explicitR_backsimeq} (one may take $n = 4$ there), since $H = 2h(\lambda)(v \uparrow 1+ 8 \lambda^2) + 2 \leqslant C(\lambda) (v \uparrow 1 + \lambda)^4$.
Let us turn to \eqref{eq:explicitR_tilde}. Start establishing a similar bound for $R_{\sqcap}$, with notation as in the proof of Lemma \ref{lem:no-round-trip}. By \eqref{eq:explicit-Rsqcap},
\[ R_{\sqcap} = 3 R_{\backsimeq}^0 \vee \frac{3\lambda}{2 \widetilde{H}}r_{1/(8 \lambda \widetilde{H})}(v) \vee 192 \lambda^2 \delta \vee r_{1/(6 (v \uparrow 2) H_3(\lambda, v, \delta))}(v), \]
where $H_3 = \left( 4 \lambda^2 H  + 2 \lambda (v \uparrow \lambda) \right) ( v \uparrow 2) \leqslant 6 \lambda^2 H  ( v \uparrow 2) \leqslant 6 \lambda^2 \widetilde{H}$, hence by \eqref{eq:precise-bound-R_backsimeq},
\begin{equation}
R_{\sqcap} \leqslant 3 \lambda r_{1/(8 \lambda^2 \widetilde{H} (v \uparrow 2))}(v) \vee C(\lambda, \delta) \left( 1 + \sup \left\{ r: v(r) \leqslant C(\lambda, \delta) \right\} \right).
\label{eq:explicit-intermediate-Rsqcap}
\end{equation}
In the proof of Lemma \ref{tracking-geodesics}, $\widetilde R$ was defined as a supremum of four terms:
\begin{align*}
\widetilde{R} & = R_{\sqcap} \vee 2(2 \lambda + 1) r_{1/(3 \lambda)}(v) \vee r_{1/(2 \widetilde{H})}(v) \vee r_{1/(2 L H_2)}(v) \\
& \quad \vee L^{-1} \sup \lbrace r : v(r) \leqslant (12 + \widetilde{h}(\lambda))\delta \rbrace \\
& \leqslant R_{\sqcap} \vee 5 \lambda^2 r_{1/(3 \lambda) \wedge 1/(2 \widetilde{H}) \wedge 1/(2LH_2)}(v) \vee C (\lambda, \delta).
\end{align*}
We need to bound the tracking constants $\widetilde{H}$ and  $H_2$. It can be observed from \eqref{ineq:track-geodesic-far} that $\widetilde{H} = (v \uparrow 3\lambda) H + 1/3 \leqslant 2 ( v \uparrow 3 \lambda) H \leqslant C(\lambda) (v \uparrow 1 + \lambda)^7$, and by \eqref{eq:explicitH2} with $k = 11 \vee 12 \lambda ( 2 \lambda + 1) \leqslant 36 \lambda^2$,
\begin{align*}
H_2 & \leqslant 2 \left[  ( v \uparrow 3 Lk) C(\lambda) (v \uparrow 3 \lambda) \widetilde{H} \right] \vee (v \uparrow 3 \lambda) H \\
& \leqslant C(\lambda) (v \uparrow L) (v \uparrow 1 + \lambda)^{n_0},  
\end{align*}
where $n_0$ is large enough. By \eqref{eq:explicit-intermediate-Rsqcap} and the previous bounds, $\widetilde{R}$ may be taken as 
\[ \widetilde{R} = 5 \lambda^2 r_{1/(C(\lambda) (v \uparrow L) (v \uparrow 1 + \lambda)^{n_0})} \vee C(\lambda, \delta) ( 1 + \sup \lbrace r : v(r) \leqslant C (\lambda, \delta) \rbrace). \]
This is a precise form of \eqref{eq:explicitR_tilde}.
Finally, we must prove \eqref{eq:explicitR}. 
With notation as in the proof of Lemma \ref{lemma:estimate-distance-quasigeodesics},
\begin{align*}
R = r_{1/(2J^+)}(v) \vee \sup \lbrace r : v(r) \leqslant 284 \delta \vee 584 \delta /J^+ \rbrace,
\end{align*}
where $J^+ = 2 \widetilde{H}_2 (v \uparrow 2)$ so that 
\begin{equation}
\label{ineq:bound-on-R}
R \leqslant r_{1/(4 \widetilde{H}_2 (v \uparrow 2))}(v) \vee 
\sup \lbrace r : v(2r) \leqslant 584 \delta /(2 \widetilde{H}_2) \rbrace,
\end{equation}
and we need bound $\widetilde{H}_2$. With notation as in the proof of \ref{tracking-geodesics}, recall from
 \eqref{eq:explicitH2} that 
$\widetilde{H}_2$ can be bounded by
\begin{align*}
\widetilde{H}_2 = ( 2 \widetilde{H} + \widetilde{h}(\lambda))( \delta + ( v \uparrow 3 \lambda)) & \leqslant C (\lambda) \widetilde{H} C (\lambda, \delta) ( v \uparrow 3 \lambda) (v \uparrow 3k) \\
& \leqslant C(\lambda, \delta) (v \uparrow 1 + \lambda)^7 (v \uparrow 3 \lambda)(v \uparrow 3 \cdot 36 \lambda^2) (v \uparrow 3 \lambda).
\end{align*}
Plugging this inequality in \eqref{ineq:bound-on-R} yields the expected \eqref{eq:explicitR}.
\end{proof}

\begin{lemma}[Sublinear growth of tracking radii]
\label{lem:sublin-growth-tracking-radii}
Let $w$ be an admissible function. For all $p \in \R_{\geqslant 0}$, define $w_p(r) = w(p+r)$, and then denote by $R_p$, resp.\ $\widetilde{R}_p$ the constants $R(\lambda, \delta, w_p)$ and $\widetilde{R}(\lambda, \delta, w_p)$ of Lemma \ref{lemma:estimate-distance-quasigeodesics}.
There exist $\widetilde{K} = \widetilde{K}(\lambda, \delta, w, L)$ and $K= K(\lambda, \delta, w, L)$ in $\R_{>0}$ such that
\begin{align}
\label{eq:explicitRtilde-sublin}
\widetilde{R}_p 
& \leqslant \widetilde{K} w(p), \text{and} \\
R_p & \leqslant K w(p).
\label{eq:explicitR-sublin}
\end{align}
\end{lemma}
\begin{proof}
By Lemma \ref{lemma:tracking-radii-grow-sublinearly}, there exists a positive integer $n$ such that $\widetilde{R}_p$ and $R_p$ may be taken as
\begin{align}
\widetilde{R}_p & = C(\lambda, \delta) r_{C(\lambda)(w_p \uparrow L)^{-1}(w_p \uparrow 1+ \lambda)^{-n}} (w_p) \notag \\ & \quad \vee C(\lambda, \delta) \left( 1 + \sup \left\{ r: w_p(r) \leqslant C(\lambda, \delta, L) \right\} \right)
\label{eq:explicitRtilde-in-proof-of-main-thm}
\\
R_p & = r_{C(\lambda)(w_p \uparrow 1+ \lambda)^{-n}} (w_p) \vee C(\lambda, \delta) \left( 1 + \sup \left\{ r: w_p(r) \leqslant C(\lambda, \delta, L) \right\} \right).
\label{eq:explicitR-in-proof-of-main-thm}
\end{align}
The rightmost terms $C(\lambda, \delta) \left( 1 + \sup \left\{ r: w_p(r) \leqslant C(\lambda, \delta, L) \right\} \right)$ are nonincreasing functions of $p$, since $\lbrace w_p \rbrace$ is a nondecreasing sequence of functions, so that their dependence over $p$ can be removed. Further, $w_p \uparrow 1 + \lambda$ is a nonincreasing function of $p$ by Lemma \ref{lem:advancin-fun} \eqref{item:vpL-smaller-than-vL}, hence $(w_p \uparrow 1+\lambda)^{-n}$ is a nondecreasing function of $p$. Thus \eqref{eq:explicitRtilde-in-proof-of-main-thm} and \eqref{eq:explicitR-in-proof-of-main-thm} may be simplified as
\begin{align*}
\widetilde{R}_p & = C(\lambda, \delta) r_{C(\lambda)(w_p \uparrow L)^{-1}(w \uparrow 1+ \lambda)^{-n}} (w_p) \vee C(\lambda, \delta, L, w)
\\
R_p & =  r_{C(\lambda)(w \uparrow 1+ \lambda)^{-n}} (w_p) \vee C(\lambda, \delta, L, w).
\end{align*}
Then by Lemma \ref{lem:advancin-fun} \eqref{item:r_epsilon(vp)}, 
$
\widetilde{R}_p \leqslant C(\lambda, \delta, L, w) \vee C(\lambda) {(w \uparrow 2)}{(w \uparrow 1+ \lambda)^{n}} w(p)
$. This proves \eqref{eq:explicitRtilde-sublin} for a constant $\widetilde{K} = \widetilde{K}(\lambda, \delta, L, w)$,
and similarly there exists $K= K(\lambda, \delta, L, w)$ such that $R_p \leqslant K w(p)$, which is \eqref{eq:explicitR-sublin}.
\end{proof}

\section{On the sphere at infinity}
\label{sec:on-the-sphere-at-infinity}

\subsection{Sublinearly quasiM{\"o}bius homeomorphisms}
With geodesic boundaries of hyperbolic spaces in mind, we define sublinearly quasiM{\"o}bius homeomorphisms abstractly between compact metric spaces:

\begin{definition}
\label{def:subqM}
Let $u$ be an admissible function.
Let $(\underline{\alpha}, \overline{\alpha}) \in \R_{>0}^2$ be a couple of constants.
Let $(\Xi, {\varrho})$ and $(\Psi, {\vartheta})$ be metric spaces and let $\varphi : \Xi \to \Psi$ be a homeomorphism.
$\varphi$ is a $(\underline{\alpha}, \overline{\alpha}, O(u))$-sublinearly quasiM{\"o}bius homeomorphism if there exist $v = O(u)$, $\nu \in \R_{>1}$ and $\mathcal{E} \in \R_{>0}$ such that for all $(\xi_1, \ldots \xi_4) \in \Xi^4$ with $0< \inf_{i \neq j} \varrho(\xi_i, \xi_j) \leqslant \sup_{i \neq j} \varrho(\xi_i, \xi_j)   < \mathcal{E}$,
\begin{align*}
\underline{\alpha} \log^+_\nu [\xi_i] - v \left(\sup_{i \neq j} \left[ - \log_\nu {\varrho}(\xi_i, \xi_j) \right]   \right)
& \leqslant \log_\nu^+ \left[\varphi(\xi_i) \right] \\
 \overline{\alpha} \log^+_\nu [\xi_i] + v \left(\sup_{i \neq j} \left[ - \log_\nu {\varrho}(\xi_i, \xi_j) \right]   \right)
 & \geqslant \log_\nu^+ \left[\varphi(\xi_i) \right].
\end{align*}
Note that one would only need a change of function $v$ within the $O(u)$-class to compensate a different choice of $\nu$.
We call $\underline{\alpha}$, \ $\overline{\alpha}$ and $\alpha = \sup \left\{ \overline{\alpha}, 1/ \underline{\alpha} \right\}$ the Lipschitz-M{\"o}bius constants of $\varphi$.
\end{definition}

Although this is not a direct consequence of Definition \ref{def:subqM}, sublinearly quasi-M{\"o}bius homeomorphisms between uniformly perfect spaces are stable under composition; we postpone the proof to subsection \ref{subsec:properties-sqM}.
Also note that in the definition one could replace the source and target distance with any equivalent real-valued kernels $\widehat{\varrho}$ and $\widehat{\vartheta}$, or even with kernels such that $\widehat{\varrho}^{\gamma_1}$ and $\widehat{\vartheta}^{\gamma_2}$ are equivalent to $\varrho$ and $\vartheta$ for a pair of exponents $\gamma_1, \gamma_2 \in \R_{>0}$ if no special attention is required on precise Lipschitz-M{\"o}bius constants.
This occurs on geodesic boundaries when $\widehat{\varrho}$ and $\widehat{\vartheta}$ are visual quasimetrics while $\varrho$ and $\vartheta$ are visual distances.

Recall that, by Proposition \ref{prop:Gromov-boundary-map}, any large-scale sublinearly Lipschitz embedding $f$ between proper geodesic Gromov-hyperbolic spaces induces a boundary map, which only depends on the $O(u)$-closeness class of $f$ so that it can be denoted $\partial_\infty[f]_{O(u)}$.

\begin{theorem}
\label{th:boundary-maps-are-sqm}
Let $u$ be an admissible function. Let $(\underline{\lambda}, \overline{\lambda})\in \R_{>0}^2$ be expansion and Lipschitz constants.
Let $f : X \to Y$ be a $(\underline{\lambda}, \overline{\lambda}, O(u))$-sublinearly biLipschitz equivalence betwen proper, geodesic hyperbolic spaces. Then $\partial_\infty[f]_{O(u)}$ is a $(\underline{\lambda}, \overline{\lambda}, O(u))$-sublinearly quasiM{\"o}bius homeomorphism.
\end{theorem}

\begin{proof}[Sketch of proof for Theorem \ref{th:boundary-maps-are-sqm}]
\renewcommand{\qedsymbol}{}
Our argument is inspired from the lecture notes by Bourdon \cite[Theorem 2.2]{Bourdon} on Mostow rigidity and Tukia's theorem; the main ingredient is Lemma \ref{lemma:estimate-distance-quasigeodesics}, which ensures that the geometric interpretation of the cross-difference (see Proposition \ref{prop:geometric-interpretation-of-X-ratio} and Figure \ref{fig:geom-x-ratio}) subsists with a sublinear error when applying a sublinearly biLipschitz equivalence and measuring distances between $O(u)$-geodesics in the target space.
Lemma \ref{lemma:estimate-distance-quasigeodesics} must be applied with care, though, since the control functions and tracking radii deteriorate as the Gromov products of endpoints grow. This is where Lemma \ref{lem:sublin-growth-tracking-radii} intervenes and certifies that the growth of tracking radii is sublinear with respect to Gromov products, so that the tracking estimates and their consequences are ultimately valid.
\end{proof}

\begin{proof}
Fix basepoints $o$ in $X$ and $Y$, and let $w = O(u)$ be an admissible function such that $f$ is a $(\underline{\lambda}, \overline{\lambda}, w)$-sublinearly biLipschitz equivalence from $(X,o)$ to $(Y,o)$.
For any quadruple $(\xi_1, \ldots \xi_4) \in \partial_\infty^4 X$, write for short $\eta_i = \partial f(\xi_i)$ for all $i$ in $\lbrace 1, \ldots 4 \rbrace$, and for all $\varepsilon, \mathcal{E} \in \R_{>0}$ such that $\varepsilon < \mathcal{E}$, let $F(\varepsilon, \mathcal{E})$ be the subspace of $\partial^4 X$ defined by
\[ \begin{cases}
 \overline{\boxtimes} \lbrace \xi_i \rbrace > - \log_\mu \varepsilon,  \\
 \underline{\boxtimes}\lbrace \xi_i \rbrace > - \log_\mu \mathcal{E}.
\end{cases} \]
Note that, since $\partial_\infty X$ is compact the space defined by the first inequality is a neighborhood of the ends in $\partial_\infty^4 X$, hence it suffices to prove the inequality
\[ \underline{\lambda}[\xi_i] - v \left( \overline{\boxtimes} \lbrace \xi_i \rbrace \right) \leqslant [\eta_i] \leqslant \overline{\lambda}[\xi_i] + v \left( \overline{\boxtimes} \lbrace \xi_i \rbrace \right) \]
for all $(\xi_i) \in F({\varepsilon}, \mathcal{E})$, for some small $\varepsilon$ and $\mathcal{E}$ and $v = O(u)$.
For any pair $\lbrace i, j \rbrace \in \left\{ \lbrace 1, 4 \rbrace, \lbrace 2, 3 \rbrace \right\}$ let $\chi_{ij}$ be a geodesic in $X$ with endpoints $\xi_i$ and $\xi_j$, resp.\ $\gamma_{ij}$ a geodesic in $Y$ with endpoints $\eta_i$ and $\eta_j$ such that $\chi_{ij}(0) = p_{\chi_{ij}}(o)$ and $\gamma_{ij}(0) = p_{\gamma_{ij}}(o)$ for all pairs $i \neq j$. Finally, write  $\widetilde{\gamma}_{ij}(t) = f \circ \chi_{ij}(t)$, and observe that $\widetilde{\gamma}_{ij}$ is a $(\lambda, w')$-geodesic, where $w' (r) := w ( \vert \chi_{14}(r) \vert \vee \vert \chi_{23}(r) \vert) \leqslant w \left( (\xi_1 \mid \xi_4)_o \vee (\xi_2 \mid \xi_3)_o + r \right)$.
Especially, $\widetilde{\gamma}_{ij}$ is a $(\lambda, w_p)$ geodesic, where 
\[ w_p(r):= w(p +r), \]
and $p = (\xi_i \mid \xi_j)_o$.
We shall apply Lemma \ref{lemma:estimate-distance-quasigeodesics} with $v^1 = w_{(\xi_1 \mid \xi_4)}$, $v^2=w_{(\xi_2 \mid \xi_3)}$ and $v = w_{\overline{\boxtimes} \lbrace \xi_i \rbrace }$. Assumptions \eqref{item:gamma-i-are-geodesics}, \eqref{item:gammatilde-i-are-Ougeodesics} and \eqref{item:asymptotic-assumption} follow from the definitions of $\gamma_{ij}$ and $\widetilde{\gamma}_{ij}$.
Then, recall from inequality \eqref{eq:explicit-t-circle} in Lemma \ref{lem:techincal-def-embedding} that if
$\vert \chi_{ij}(0) \vert \geqslant t_{\ocircle}(\vert f(o) \vert, w)$, then for all $t \in \R$, $\frac{1}{3 \lambda} \vert \chi_{ij}(t) \vert \leqslant \vert \widetilde{\gamma}_{ij}(t) \vert \leqslant 3 \lambda \vert \chi_{ij}(t) \vert$, and then
\[ \forall t \in \R , \, \vert \widetilde{\gamma_{ij}}(0) \vert \leqslant 3 \lambda \vert \chi_{ij}(0) \vert \leqslant 3 \lambda \vert \chi_{ij}(t) \vert \leqslant 9 \lambda^2 \vert \widetilde{\gamma_{ij}}(t) \vert.  \]
This is the right-hand side inequality of \eqref{item:proximal-assumption} with $L = 9 \lambda^2$, that we fix for the rest of the proof. Observe that the lower bound needed on the radii $\vert \chi_{ij}(0) \vert$ is guaranteed as soon as $\underline{\boxtimes} \lbrace \xi_i \rbrace \geqslant t_{\ocircle}(\vert f(o) \vert, w) = r_{1/(3 \lambda)}(w) \vee 3 \lambda \vert f(o) \vert$. 
On the other hand, by Cornulier's theorem \ref{thm:Cornulier-Holder} $\partial_\infty[f]$ is uniformly continuous on $\partial_\infty X$, so there exists $R_{\square} \in \R_{\geqslant 0}$ such that $\underline{\boxtimes} \lbrace \xi_i \rbrace \geqslant R_{\square} \implies \underline{\boxtimes} \lbrace \eta_i \rbrace \geqslant 60 \delta$.
Let $\widetilde{K} = \widetilde{K}(\lambda, w, \delta, L)$ be the constant from Lemma \ref{lem:sublin-growth-tracking-radii}, and define 
\[\mathcal{E} = \mu^{- (R_\square \vee t_{\ocircle}(\vert f(o) \vert, w) \vee r_{1/ (3 \lambda \widetilde{K})}(w))}. \] 
Then as soon as $\underline{\boxtimes} \lbrace \xi_i \rbrace > - \log_{\mu} \mathcal{E}$,
\[ \begin{cases}
\underline{\boxtimes} \lbrace \eta_i \rbrace \geqslant 60 \delta.
& \text{as} \; \underline{\boxtimes} \lbrace \xi_i \rbrace \geqslant R_\square \\
\vert \widetilde{\gamma}_{ij}(0) \vert \leqslant L \inf_{t \in \R} \vert \widetilde{\gamma_{ij}}(t) \vert 
& \text{as} \; \underline{\boxtimes} \lbrace \xi_i \rbrace \geqslant t_{\ocircle}(\vert f(o) \vert, w)
\\
{\widetilde{R}_{(\xi_i \mid \xi_j)_o}} \leqslant \frac{(\xi_i \mid \xi_j)_o}{3 \lambda} \leqslant \vert \widetilde{\gamma}_i(0) \vert 
& \text{as} \; (\xi_i \mid \xi_j)_o \geqslant \underline{\boxtimes} \lbrace \xi_i \rbrace \geqslant t_{\ocircle}(\vert f(o) \vert, w)
\vee r_{1/ (3 \lambda \widetilde{K})}(w).
\end{cases} 
\]
The first line is the first condition in \eqref{item:conditions-on-gromprods}, the second and third one are the assumption \eqref{item:proximal-assumption}; we used \eqref{eq:explicitRtilde-sublin} from Lemma \ref{lem:sublin-growth-tracking-radii} in the third line.
By the conclusion of Cornulier's theorem \ref{thm:Cornulier-Holder} applied to both $\partial_\infty [f]$ and to $\partial_\infty [f]^{-1}$, there exists $\varepsilon_0 \in \R_{>0}$ such that  
\[ 
\overline{\boxtimes} \lbrace \xi_i \rbrace > - \log_{\mu} \varepsilon_0 \implies 2 \lambda \overline{\boxtimes} \lbrace \xi_i  \rbrace \geqslant \overline{\boxtimes} \lbrace \eta_i \rbrace \geqslant \frac{1}{2 \lambda} \overline{\boxtimes} \lbrace \xi_i \rbrace.
\]
Let $K$ be the constant from Lemma \ref{lem:sublin-growth-tracking-radii}. Define $\varepsilon = \varepsilon_0 \wedge \mathcal{E} \wedge \mu^{-2 \lambda r_{(1/ 3 \lambda K)}(w)}$. Then by \eqref{eq:explicitR-sublin} of Lemma \ref{lem:sublin-growth-tracking-radii},
$
\overline{\boxtimes} \lbrace \xi_i \rbrace >  - \log_\mu \varepsilon \implies \overline{\boxtimes}\lbrace \eta_i \rbrace \geqslant R_{\overline{\boxtimes} \lbrace \xi_i \rbrace}
$.
Thus if $(\xi_i) \in F(\mathcal{E}, \varepsilon)$ then Lemma \ref{lemma:estimate-distance-quasigeodesics} applies to $(\gamma_{ij}, \widetilde{\gamma}_{ij})$, and 
\begin{align}
\left\vert d_Y(\gamma_{23}, \gamma_{14}) - d_Y(\widetilde{\gamma}_{23}, \widetilde{\gamma}_{14}) \right\vert & \leqslant Jw_{\overline{\boxtimes}\lbrace \xi_i \rbrace} (\overline{\boxtimes} \lbrace \eta_i \rbrace) \notag \\
& \leqslant J(w \uparrow 2 \lambda) w(\overline{\boxtimes} \lbrace \eta_i \rbrace ).
\label{eq:application-of-lemma-est-dist}
\end{align}
Thanks to Proposition \ref{prop:geometric-interpretation-of-X-ratio}, there exists $C = C(\delta)$ in $\R_{\geqslant 0}$ such that
\[
\begin{cases}
d_X(\chi_{14}, \chi_{23}) - C(\delta) & \leqslant \log^+[\xi_i] \leqslant d_X(\chi_{14}, \chi_{23}) + C(\delta). \\
d_Y(\gamma_{14}, \gamma_{23}) - C(\delta) & \leqslant \log^+[\eta_i] \leqslant d_Y(\gamma_{14}, \gamma_{23}) + C(\delta).
\end{cases}
\]
In view of \eqref{eq:application-of-lemma-est-dist} and the previous set of inequalities, it suffices to prove
\begin{equation}
\underline{\lambda} d_X (\chi_{14}, \chi_{23}) - v \left( \overline{\boxtimes} \lbrace \xi_i \rbrace \right) \leqslant d_Y(\gamma_{14}, \gamma_{23}) \leqslant \overline{\lambda} d_X (\chi_{14}, \chi_{23}) + v \left( \overline{\boxtimes} \lbrace \xi_i \rbrace \right)
\label{eq:to-prove-for-main-thm}
\end{equation}
for some function $v = O(u)$.
Start with the left-hand side inequality. Letting $\widetilde{s}_1,\widetilde{s}_2 \in \R$ be such that $\vert f \circ \chi_{14}(\widetilde s_1) - f \circ \chi_{23}(\widetilde s_2) \vert = d (\widetilde{\gamma}_{14}, \widetilde{\gamma}_{23})$,
\begin{align*}
\underline{\lambda} d(\chi_{14}, \chi_{23}) - w \left( \overline{\boxtimes} \lbrace \xi_i \rbrace \right) & \leqslant
 \underline{\lambda} \vert \chi_{14}(\widetilde s_1) - \chi_{23}(\widetilde s_2) \vert  - w \left( \vert \chi_{14}(\widetilde s_1) \vert \vee \vert \chi_{23}(\widetilde s_2) \vert \right) \\
& \leqslant
\left\vert f \circ \chi_{14}(\widetilde s_1) - f \circ \chi_{23}(\widetilde s_2) \right\vert \\
& =  d (\widetilde{\gamma}_{14}, \widetilde{\gamma}_{23}) \\
& \underset{\eqref{eq:application-of-lemma-est-dist}}{ \leqslant} d (\gamma_{14}, \gamma_{23}) +  J (w \uparrow 2 \lambda) \left( \overline{\boxtimes} \lbrace \eta_i \rbrace \right),
\end{align*}
hence
\begin{equation}
\underline{\lambda} d(\chi_{14}, \chi_{23}) \leqslant d (\gamma_{14}, \gamma_{23}) + \left( 1 + J(w \uparrow 2 \lambda) \right) v \left( \overline{\boxtimes} \lbrace \eta_i \rbrace + \overline{\boxtimes} \lbrace \xi_i \rbrace \right).
\end{equation}
Let us proceed in the same way for the right-hand side of \eqref{eq:to-prove-for-main-thm}. By Lemma \ref{lemma:estimate-distance-quasigeodesics}, letting $s_1,s_2 \in \R$ be such that $\vert \chi_{14}(s_1) - \chi_{23}(s_2) \vert = d (\chi_{14}, \chi_{23})$,
\begin{align}
d(\gamma_{14}, \gamma_{23})
& \leqslant d (\widetilde{\gamma}_{14}, \widetilde{\gamma}_{23}) + J (w \uparrow 2 \lambda) \left( \overline{\boxtimes} \lbrace \eta_i \rbrace \right) \notag \\
& \leqslant
\vert \widetilde{\gamma}_{14}(s_1) - \widetilde{\gamma}_{23}(s_2) \vert + J (w \uparrow 2 \lambda) \left( \overline{\boxtimes} \lbrace \eta_i \rbrace \right) \notag  \\
& \leqslant \overline{\lambda} d(\chi_{14}, \chi_{23}) + \left( 1 + (w \uparrow 2 \lambda )^2 \right) w \left( \overline{\boxtimes} \lbrace \xi_i \rbrace  \right).
\end{align}
Setting $v = \left( 1 + (w \uparrow 2 \lambda )^2 \right) w$ this proves \eqref{eq:to-prove-for-main-thm} and the theorem. 
\end{proof}

\subsection{Properties of sublinearly quasiM{\"o}bius homeomorphisms}
\label{subsec:properties-sqM}
After simplifying the cross-ratio estimates when two, resp.\ one points are far away, one obtains that sublinearly quasiM{\"o}bius homeomorphisms between appropriate spaces are H{\"o}lder, resp.\ almost quasisymmetric, see figure \ref{fig:SQMtoHolder}. Precisely we work under the following assumption (Buyalo and Schroeder {\cite[7.2]{BuyaloSchroeder}} or Mackay and Tyson \cite[1.3.2]{MackayTyson}); see however Remark \ref{rem:on-uniform-perfectness}.

\begin{definition}
\label{def:unif-perfect}
Let $\Xi$ be a metric space. Then $X$ is uniformly perfect if there exists $\tau \in (0,1)$ such that for every ball $B \subsetneq \Xi$, the annulus $B \setminus \tau B$ is non-empty.
\end{definition}

Note that in the definition, for any positive integer $k$, up to replacing $\tau$ with $\tau^k$ one can assume for free that $B \setminus \tau B$ has $k$ points. Uniform perfectness is granted for boundaries of non-elementary hyperbolic groups, or for connected spaces.
\begin{center}
\begin{figure}
\begin{tikzpicture}[line cap=round,line join=round,>=angle 45,x=0.3cm,y=0.5cm]
\clip(-1,-3.1) rectangle (14,3.5);
\draw [shift={(6.83,3)},dash pattern=on 2pt off 2pt]  plot[domain=3.43:4.14,variable=\t]({1*7.12*cos(\t r)+0*7.12*sin(\t r)},{0*7.12*cos(\t r)+1*7.12*sin(\t r)});
\draw [shift={(16.83,3)},dash pattern=on 2pt off 2pt]  plot[domain=3.43:4.14,variable=\t]({1*7.12*cos(\t r)+0*7.12*sin(\t r)},{0*7.12*cos(\t r)+1*7.12*sin(\t r)});
\draw [dash pattern=on 2pt off 2pt] (0,1)-- (10,1);
\draw [dash pattern=on 2pt off 2pt] (3,-3)-- (13,-3);
\draw [domain=-1:13] plot(\x,{(--14.5-0*\x)/4.83});
\draw [shift={(5.75,-2)}] plot[domain=0:2.21,variable=\t]({1*1.25*cos(\t r)+0*1.25*sin(\t r)},{0*1.25*cos(\t r)+1*1.25*sin(\t r)});
\draw [shift={(-1.45,3)}] plot[domain=-0.46:0,variable=\t]({1*8.68*cos(\t r)+0*8.68*sin(\t r)},{0*8.68*cos(\t r)+1*8.68*sin(\t r)});
\fill  (5,-1) circle (1.5pt) node[anchor=east]{$\xi_2$} ;
\fill  (7,-2) circle (1.5pt) node[anchor=west]{$\xi_3$}  ;
\draw [->] (0,2.5) -- (-1,2.5);
\draw [->] (12,2.5) -- (13,2.5);
\draw (-0.5,2) node [anchor=center]{$\xi_1$} ;
\draw (12.5,2) node [anchor=center]{$\xi_4$} ;
\end{tikzpicture}
\begin{tikzpicture}[line cap=round,line join=round,>=angle 45,x=0.4cm,y=0.32cm]
\clip(-3.63,-2.5) rectangle (14,7);
\draw [dash pattern=on 2pt off 2pt] (2,4)-- (-1,-2);
\draw [dash pattern=on 2pt off 2pt] (-1,-2)-- (9,-2);
\draw [dash pattern=on 2pt off 2pt] (9,-2)-- (12,4);
\draw [dash pattern=on 2pt off 2pt] (2,4)-- (12,4);
\draw (8,2)-- (8,8);
\draw [shift={(4,0)}] plot[domain=0:3.14,variable=\t]({1*1*cos(\t r)+0*1*sin(\t r)},{0*1*cos(\t r)+1*1*sin(\t r)});
\draw [shift={(8,1)}] plot[domain=1.57:3.14,variable=\t]({1*4*cos(\t r)+0*4*sin(\t r)},{0*4*cos(\t r)+1*4*sin(\t r)});
\fill  (3,0) circle (1.5pt) node[anchor=north]{$\xi_1$} ;
\fill  (5,0) circle (1.5pt) node[anchor=north]{$\xi_2$} ;
\fill  (8,2) circle (1.5pt) node[anchor=north]{$\xi_3$} ;
\draw [->] (9,6) -- (9,7);
\draw (10,6.5) node[anchor=center]{$\xi_4$};
\end{tikzpicture}
\caption{Hyperbolic ideal tetrahedra. On the left two points are far away from the remaining pair; on the right one point is far from the remaining triple.}
\label{fig:SQMtoHolder}
\end{figure}
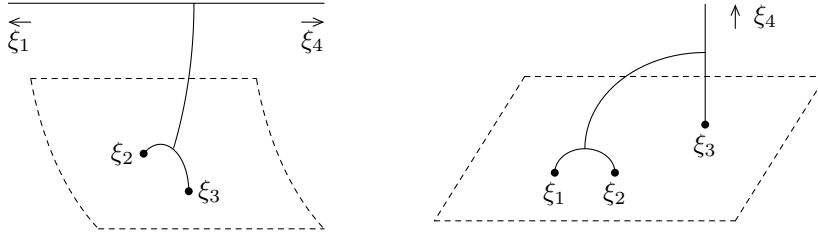
\end{center}
\begin{proposition}[``almost" H{\"o}lder continuity]
\label{prop:almost-Holder}
Let $(\Xi; \varrho)$ and $(\Psi, \vartheta)$ be compact uniformly perfect metric spaces and let $\varphi : \Xi \to \Psi$ be a $(\underline{\lambda}, \overline{\lambda}, O(u))$-sublinearly quasiM{\"o}bius homeomorphism. Then $\varphi$ admits a modulus of continuity
\begin{equation}
\omega(t) =  \exp \left( \underline{\lambda} \log t  +  v(- \log t) \right),
\label{eq:almost-Holder}
\end{equation}
with $v = O(u)$.
\end{proposition}

\begin{remark}
As a consequence, under the same assumptions and for all $\alpha \in (0,\underline{\lambda})$, $\varphi$ is $\alpha$-H{\"o}lder continuous. Under the restriction on spaces, Cornulier's theorem (used in the proof) can be recovered as a corollary of Theorem \ref{th:boundary-maps-are-sqm}.
\end{remark}

\begin{proof}
Let $\mathcal{E}$ be the constant from Definition \ref{def:subqM} associated to $\varphi$, and let $\tau$ be such that $\Xi$ is $\tau$-uniformly perfect.
Define
\begin{align*}
\mathcal{D}_1 & := \frac{\tau^4}{4} \left( \mathcal{E} \wedge \frac{\operatorname{diam} \Xi}{3} \right) \;\text{and} \\
\mathcal{D}'_1 & = \inf \left\{ \vartheta (\varphi(\xi_1), \varphi(\xi_2)) : \xi_1, \xi_2 \in \Xi, \varrho(\varphi(\xi_1), \varphi(\xi_2)) \geqslant \left( \tau^{-1} - 1 \right) \mathcal{D}_1 \right\}.
\end{align*}
Let $\xi_1$ and $\xi_2$ in $\Xi$ be such that $\varrho(\xi_1, \xi_2) <  \mathcal{D}_1$. The ball $B:= \tau^{-4}B(\xi_1, \mathcal{D}_1)$ is not equal to $\Xi$ (this would indeed contradict the definition of $\mathcal{D}_1$), so there exists $\alpha \in B \setminus \tau B$ and $\beta \in \tau^2 B \setminus \tau^3 B$. By the triangle inequality
\begin{equation*}
\varrho(\alpha, \beta) \geqslant \left( \tau^{-3} - \tau^{-2} \right) \mathcal{D}_1
\; \text{and} \;
\varrho(\beta, \xi_2) \geqslant \left( \tau^{-1} - 1 \right) \mathcal{D}_1,
\end{equation*}
for short
\begin{equation}
\inf_i \varrho(\alpha, \xi_i) \wedge \inf_i \varrho(\beta, \xi_i)   \geqslant \left( \tau^{-1} - 1 \right) \mathcal{D}_1.
\end{equation}
Further, by definition of $\mathcal{D}'_1$, a similar inequality holds in the target space:
\begin{equation}
\inf_i \vartheta(\varphi(\alpha), \varphi(\xi_i)) \wedge \inf_i \vartheta(\varphi(\beta), \varphi(\xi_i))   \geqslant \mathcal{D}'_1.
\end{equation}
By definition of the metric cross ratios,
\[ \frac{ \left( \tau^{-1} - 1 \right)^2 \mathcal{D}_1^2}{\operatorname{diam} (\Xi)} \frac{1}{\varrho(\xi_1, \xi_2)} \leqslant [\alpha,\xi_1,\xi_2, \beta] \leqslant \frac{\operatorname{diam} (\Xi)^2}{\mathcal{D}_1} \frac{1}{\varrho(\xi_1, \xi_2)} . \]
\[ \frac{{\mathcal{D}'_1}^2}{\operatorname{diam} (\Psi)} \frac{1}{ \vartheta(\varphi(\xi_1), \varphi(\xi_2))} \leqslant [\alpha',\varphi(\xi_1),\varphi(\xi_2); \beta'] \leqslant \frac{\operatorname{diam} (\Psi)^2}{\mathcal{D}'_1} \frac{1}{\vartheta(\varphi(\xi_1), \varphi(\xi_2))}. \]
thus $\log^+ [\alpha,\xi_1,\xi_2, \beta] - \log^+ \frac{1}{\varrho(\xi_1, \xi_2)}$ and $[\alpha',\varphi(\xi_1),\varphi(\xi_2); \beta'] - \log^+ \frac{1}{\vartheta(\varphi(\xi_1), \varphi(\xi_2))}$ are bounded by
\[ \mathcal{L}= 2 \left( \left\vert \log \frac{1- \tau}{\tau} \right\vert + \vert \log \mathcal{D}_1 \vert + \vert \log \operatorname{diam}(\Xi) \vert \right) \; \text{and} \; \mathcal{L}' = \left( \vert \log \mathcal{D'} \vert + \vert \log \operatorname{diam}(\Psi) \vert \right)   \]
respectively. Now by hypothesis $\varphi$ is $(\underline{\lambda}, \overline{\lambda}, v_0)$-sublinearly quasiM{\"o}bius for some $v_0 = O(u)$. By definition, setting $v = v_0 + \mathcal{L}$, for all $\xi_1, \xi_2$ such that $\varrho(\xi_1, \xi_2) < \mathcal{D}_1 \wedge 1$ (note  that $\lbrace \xi_1, \xi_2 \rbrace$ is the closest pair among $\xi_1, \xi_2, \alpha, \beta$),
\begin{align}
- \log \vartheta(\varphi(\xi_1), \varphi(\xi_2)) & \leqslant \overline{\lambda}(- \log \varrho(\xi_1, \xi_2)) + v(- \log \varrho(\xi_1, \xi_2)), \label{eq:anti-Holder} \\
- \log \vartheta(\varphi(\xi_1), \varphi(\xi_2)) & \geqslant \underline{\lambda}(- \log \varrho(\xi_1, \xi_2)) - v(- \log \varrho(\xi_1, \xi_2)). \label{eq:Holder}
\end{align}
In particular the conclusion \eqref{eq:almost-Holder} is equivalent to the second inequality.
\end{proof}

The H{\"o}lder continuity \eqref{eq:Holder} intervenes in the following analog of Lemma \ref{lem:techincal-def-embedding}, a technical refinement of definition \ref{def:subqM}.

\begin{lemma}
\label{lem:technical-subqM}
Let $u$ be an admissible function. Let $(\underline{\alpha}, \overline{\alpha})$ be Lipschitz-M{\"o}bius data. Let $\varphi$ be a $(\underline{\alpha}, \overline{\alpha}, O(u))$ sublinearly quasiM{\"o}bius homeomorphism between compact uniformly perfect spaces $(\Xi, \widehat{\varrho})$ and $(\Psi, \widehat{\vartheta})$. There exist $\widehat{v} = O(u)$ and $\mathcal{E}_2 \in \R_{>0}$ such that for all $(\xi_1, \ldots \xi_4) \in \Xi^4$ with $0< \inf_{i \neq j} \varrho(\xi_i, \xi_j) \leqslant \sup_{i \neq j} \varrho(\xi_i, \xi_j)   < \mathcal{E}_2$,
\begin{align*}
\underline{\alpha} \log^+_\nu [\xi_i] - \widehat{v} \left(\sup_{i \neq j} \left[ - \log_\nu \widehat{\varrho}(\xi_i, \xi_j) \right]  \wedge 
\sup_{i \neq j} \left[ - \log_\nu \widehat{\vartheta}(\varphi(\xi_i), \varphi(\xi_j)) \right]  \right) & \leqslant \log_\nu^+ \left[\varphi(\xi_i) \right] \\
\overline{\alpha} \log^+_\nu [\xi_i] + \widehat{v} \left(\sup_{i \neq j} \left[ - \log_\nu \widehat{\varrho}(\xi_i, \xi_j) \right]  \wedge 
\sup_{i \neq j} \left[ - \log_\nu \widehat{\vartheta}(\varphi(\xi_i), \varphi(\xi_j)) \right]  \right) & \geqslant \log_\nu^+ \left[\varphi(\xi_i) \right],
\end{align*}
where $\widehat{v} = O(u)$.
\end{lemma}

\begin{proof}
Let $v = O(u)$ be such that $\varphi$ is a $(\underline{\alpha}, \overline{\alpha}, v)$-sublinearly quasiM{\"o}bius homeomorphism. Then by Proposition \ref{prop:almost-Holder} and the fact that $v$ is sublinear, there is $\mathcal{E}_H \in \R_{>0}$ such that for all $(\xi_1, \ldots \xi_4) \in \Xi$ distinct and such that $\inf \widehat{\varrho}(\xi_1, \xi_2) \leqslant \mathcal{E}_H \wedge e^{-(\log \nu) r_{\underline{\lambda}/2}(v)}$,
\[
\sup_{i \neq j} \left[ - \log_\nu \widehat{\vartheta}(\varphi(\xi_i), \varphi(\xi_j)) \right] \geqslant (\underline{\lambda} /2) \sup_{i \neq j} \left[ - \log_\nu \widehat{\varrho}(\xi_i, \xi_j) \right]. \]
The conclusion follows, with $\widehat{v} = \left( v \uparrow \frac{2}{\lambda} \right) v$.
\end{proof}

\begin{proposition}
\label{prop:groupoid-sqM}
Let $u$ be an admissible function. $O(u)$-sublinearly quasiM{\"o}bius homeomorphisms form a groupoid $\mathcal{M}_{O(u)}$ with uniformly perfect compact metric spaces as objects. Composition in $\mathcal{M}_{O(u)}$ has a multiplicative effect on Lipschitz-M{\"o}bius and reverse Lipschitz-M{\"o}bius constants.
\end{proposition}

\begin{proof}
Let $\Omega, \Xi, \Psi$ be compact metric spaces, and let $\varphi : \Xi \to \Psi$ and $\psi : \Omega \to \Xi$ be $O(u)$-quasiM{\"o}bius homeomorphisms, with respective parameters $(\underline{\alpha}_\varphi, \overline{\alpha}_\varphi, v_\varphi)$ and $(\underline{\alpha}_\psi, \overline{\alpha}_\psi, v_\psi)$. Let $(\omega_1, \ldots, \omega_4)$ be a $4$-tuple of distinct points in $\Omega$ ; for $i \in \lbrace 1, \ldots, 4 \rbrace$ set $\xi_i = \psi(\omega_i)$ and $\eta_i = \varphi(\xi_i)$. Set
\begin{equation*}
w = \overline{\alpha}_\psi v_{\varphi} + \left( v \uparrow \frac{2}{\alpha_\varphi} \right) v_{\psi}.
\end{equation*}
Then by the previous lemma, $\varphi \circ \psi$ is a $(\underline{\alpha}_\varphi \underline{\alpha}_\psi, \overline{\alpha}_\varphi \overline{\alpha}_\psi, w)$-sublinearly quasiM{\"o}bius homeomorphism.
\end{proof}

\begin{remark}
\label{rem:on-uniform-perfectness}
The assumption of uniform perfectness (Definition \ref{def:unif-perfect}) could be dropped in Proposition \ref{prop:groupoid-sqM} if one adopts the heavier form of Definition \ref{def:subqM} given by the inequalities of Lemma \ref{lem:technical-subqM}. It follows from the proof of Theorem \ref{th:boundary-maps-are-sqm} that this more restrictive definition is still valid for boundary maps of sublinearly biLipschitz equivalences.
\end{remark}

We now turn to the scale-sensitive moduli distortion property of sublinearly quasiM{\"o}bius homeomorphisms.

\begin{proposition}
\label{prop:sublin-quasisym}
Let $\varphi$ be a $(\underline{\lambda}, \overline{\lambda}, O(u))$-sublinearly quasiM{\"o}bius homeomorphism between spaces $(\Xi, \varrho)$ and $(\Psi,  \vartheta)$. Assume that $\Xi$ is uniformly perfect. There exist $\mathcal{D}_1 \in \R_{>0}$ and $w = O(u)$ such that the following holds: let $A \subset \Xi$ be an annulus of inner radius $r \in \R_{>0}$ and outer radius $R \leqslant \mathcal{D}_1$. Then $\varphi(A)$ is contained in an annulus with modulus $2\lambda \mathfrak{M} + w(- \log r))$.
\end{proposition}

\begin{proof}
Define $\mathcal{D}_1$ and $\mathcal{D}'_1$ as in the proof of Proposition \ref{prop:almost-Holder}.
For any triple $(\xi_1, \xi_2, \xi_3) \in \Xi^3$ such that $\lbrace \xi_1, \xi_2, \xi_3 \rbrace$ has diameter less than $\mathcal{D}_1$ one can find $\omega \in \tau^{-4} B(\xi_1, \mathcal{D}_1) \setminus \tau^{-3} B(\xi_1, \mathcal{D}_1)$. Define $\omega' = \varphi( \omega)$. By the triangle inequality and the definition of $\mathcal{D}'_1$
\[ \begin{cases}
\inf_i \varrho(\omega, \xi_i) & \geqslant (\tau^{-3} - 1) \mathcal{D}_1 \geqslant \frac{1 - \tau}{\tau} \mathcal{D}_1,\, \text{and} \\
\inf_i \vartheta( \omega', \eta_i) & \geqslant \mathcal{D}'_1,
\end{cases} \]
where $\eta_i = \varphi(\xi_i)$ for $i \in \lbrace 1,2,3 \rbrace$. Define $\mathcal{D}_2 = \frac{1 - \tau}{\tau} \mathcal{D}_1$. Applying the definition of the metric cross-ratio we deduce from the previous inequalities
\begin{equation}
\left\vert \log [\omega, \xi_1, \xi_2, \xi_3] - \log \frac{\varrho(\xi_1, \xi_3)}{\varrho(\xi_1, \xi_2)} \right\vert \leqslant 2 \vert \log \operatorname{diam}(\Xi) \vert \vee \vert \log \mathcal{D}_2 \vert
\label{eq:first-qh-qs}
\end{equation}
\begin{equation}
\left\vert \log [\omega', \eta_1, \eta_2, \eta_3] - \log \frac{\vartheta(\eta_1, \eta_3)}{\vartheta(\eta_1, \eta_2)} \right\vert \leqslant 2 \vert \log \operatorname{diam}(\Xi) \vert \vee \vert \log \mathcal{D}'_1 \vert.
\label{eq:second-qh-qs}
\end{equation}
Denote by $\mathcal{L}$, resp.\ $\mathcal{L'}$ the right-hand side bounds of \eqref{eq:first-qh-qs}, resp.\ \eqref{eq:second-qh-qs}.
Let $r \in \R_{>0}$ and $\mathfrak{M} \in \R_{\geqslant 0}$ be such that $R = r \exp(\mathfrak{M}) \leqslant \mathcal{D}_1$. Fix $\xi_1$ and write $B = B(\xi_1, r)$. Fix $\xi_2$ in $\tau B \setminus \tau^2 B$. For any $\xi_3 \in A = B(R) \setminus B(r)$ the triangle inequality gives
\[ \varrho(\xi_1, \xi_2) \wedge \varrho(\xi_1, \xi_3)  \wedge \varrho(\xi_2, \xi_3) \geqslant \left( (1 - \tau) \wedge \tau^2 \right) r. \]
Let $v_0$ be such that $\varphi$ is $(\underline{\lambda}, \overline{\lambda}, v_0)$-quasiM{\"o}bius. Define $v_1 = v_0 + \mathcal{L} \vee \mathcal{L'}$ and then $v_2 = (v_1 \uparrow (1-\tau) \wedge \tau^2) v_1$.
Applying Definition \ref{def:subqM} to $\varphi$ for $(\omega, \xi_1, \xi_2, \xi_3)$ together with \eqref{eq:first-qh-qs} and \eqref{eq:second-qh-qs}, one obtains the set of inequalities
\begin{align}
\log \frac{\vartheta(\eta_1, \eta_3)}{\vartheta(\eta_1, \eta_2)} \leqslant \log^+ \frac{\vartheta(\eta_1, \eta_3)}{\vartheta(\eta_1, \eta_2)} & \leqslant \overline{\lambda} \log^+ \frac{\varrho(\xi_1, \xi_3)}{\varrho(\xi_1, \xi_2)} + v_1 \left( - \log \left( (1 - \tau) \wedge \tau^2 \right) r \right) \notag \\
& \leqslant \overline{\lambda} \mathfrak{M} + v_2(- \log r) - 2 \overline{\lambda} \log \tau, \notag \\
- \log \frac{\vartheta(\eta_1, \eta_3)}{\vartheta(\eta_1, \eta_2)} \leqslant \log^+ \frac{\vartheta(\eta_1, \eta_2)}{\vartheta(\eta_1, \eta_3)} & \leqslant \overline{\lambda} \log \frac{\varrho(\xi_1, \xi_2)}{\varrho(\xi_1, \xi_3)} + v_2( -\log r) \notag = v_2(- \log r).
\end{align}
Hence for any $\xi_3, \xi_3' \in A$, by the triangle inequality in $\R$, using $\vartheta(\eta_1, \eta_2)$ as an intermediate point,
\begin{equation}
\left\vert \log \frac{\vartheta(\eta_1, \eta_3)}{\vartheta(\eta_1, \eta'_3)} \right\vert \leqslant 2 \overline{\lambda}\mathfrak{M} + 2v_2( - \log r) - 4 \overline{\lambda} \log \tau
\leqslant 2 \lambda \mathfrak{M} + w( - \log r),
\end{equation}
where $w= O(u)$. The proposition follows from the last statement. The expansion constant $\underline{\lambda}$ would intervene in lower bounds on $\inf \frac{\vartheta(\varphi(\xi_1), \varphi(\xi_3))}{\vartheta(\varphi(\xi_1), \varphi(\xi_2))}$ for $\xi_2$ in the internal ball, and $\xi_3$ outside the external ball, centered at $\xi_1$.
\end{proof}

This last property of sublinearly M{\"o}bius maps will be of use in section \ref{sec:Riemannian} where we implement some measure theory on the boundary. There is still a need to reformulate it slightly, however, since we will be then working with balls rather than annuli, and quasimetrics rather than true distances. In that purpose, we introduce the following terminology:
for any $s \in \R_{>0}$, if $B$ is a quasiball $B = B^{\widehat{\varrho}}(\Xi,r)$ where $\widehat{\varrho}$ is a kernel equivalent to the distance in $X$, then $sB$ is $B^{\widehat{\varrho}}(\Xi,sr)$.
If $\Xi$ is $\tau$-uniformly perfect for every $\tau \in (0,1)$ with respect to $\widehat{\varrho}$ (for instance, if it is connected) this is a continuous operation of $\R_{>0}$ on the space of quasiballs.

\begin{proposition}
\label{prop:sqM-for-balls}
Assume that $(\Xi, \varrho)$ and $(\Psi, \vartheta)$ are compact connected topological manifolds, and that $\varphi : \Xi \to \Psi$ is a $(\underline{\lambda}, \overline{\lambda}, O(u))$-sublinearly quasiM{\"o}bius homeomorphism.  Let $Q \in \R_{\geqslant 1}$ be a constant.
Let $\widehat{\varrho}$, resp.\ $\widehat{\vartheta}$ be an equivalent kernel on $\Xi$, resp.\ on $\Psi$.
Then for any $\alpha \in (0, \underline{\lambda})$ and $\beta \in (\overline{\lambda}, + \infty)$ there exists $w = O(u)$ (depending on $Q$) such that for any $\widehat{\varrho}$-quasiball $B \subset \Xi$ with center $\xi$ and small enough radius $r$ there exists a $\widehat{\vartheta}$-quasiball $B'$ in $\Psi$, and
\begin{equation}
\begin{cases}
r^\beta \leqslant \mathrm{radius}(B') \leqslant r^\alpha \\
B' \subseteq \varphi \left( Q^{-1}B \right) \subset \varphi(B) \subseteq Q^{\lambda} e^{w(-\log r)} B'.
\end{cases}
\end{equation}
\end{proposition}

\begin{remark}
Though this would be valid, we do not include in the conclusion that $B$ have center $\varphi(\xi)$, since it will not be required in section \ref{sec:Riemannian}.
\end{remark}

\begin{proof}
The statement for any equivalent kernel follows from the particular case when $\widehat{\varrho} = \varrho$ and $\widehat{\vartheta}= \vartheta$. 
Let $B'' \setminus B'$ be an annulus containing $\varphi( B '' \setminus B')$.
Since $\varphi$ is a homeomorphism, images of balls, resp.\ spheres by $\varphi$ are topological balls, resp. spheres. By the Jordan-Brouwer separation theorem, $\varphi(Q^{-1}B)$ is one of the two connected components of $\Psi \setminus \varphi(\partial(Q^{-1}B)$, and by Proposition \ref{prop:almost-Holder}, if $r$ is small enough its diameter is bounded by $r^\alpha$. Since $\varphi(\partial(Q^{-1}B)) \subset B'' \setminus B'$, $B' \subseteq \varphi(Q^{-1}B)$ and $\operatorname{radius}(B') \leqslant r^\alpha$.
By Proposition \ref{prop:sublin-quasisym} $B''$ can be written $Q^{\lambda}e^w(- \log r) B'$.
Finally, by Proposition \ref{prop:almost-Holder}, for all $\beta' \in (\overline{\lambda}, \beta)$, $\operatorname{diam}(B'') \geqslant \operatorname{diam} \varphi(B) \geqslant r^{\beta'}$ if $r$ is small enough. This implies the lower bound on $\operatorname{radius}(B')$ for $r$ small enough.
\end{proof}

\section{Riemannian negatively curved homogeneous spaces}
\label{sec:Riemannian}

\subsection{Setting}

Simple Lie groups of real rank one with left invariant metrics are mentioned early in Gromov's essay as important examples of $\delta$-hyperbolic spaces \cite[1.5(2)]{Gromov_HG} and it is natural to ask to which extent they -- or their quasiisometrically related hyperbolic symmetric spaces -- differ on the large scale.
Beyond hyperbolic symmetric spaces, it was proved in 1974 by Heintze \cite[§ 2]{Heintze} that any homogeneous, simply connected negatively curved Riemannian manifold is the principal space of a solvable Lie group $S = N \rtimes_\alpha \R$, where $N$ is nilpotent with Lie algebra $\mathfrak{n}$ and $\alpha \in \operatorname{Der}(\mathfrak{n})$ is such that for any compact neighborhood $K$ of $1$ in $N$, $\cup_{t\geqslant 0} \exp(-t\alpha) K = N$. Such an $S$ is called a Heintze group.

For a principal space $X$ of the Heintze group $S$, denote by $\omega$ the endpoint on $\partial_\infty X$ (in positive time) of the orbits of the $\R$ factor, and by $\partial_\infty^\ast X$ the punctured boundary $\partial_\infty X \setminus \lbrace \omega \rbrace$. 
Any choice of a basepoint $o \in X$ will determine
a chart $\Phi : \partial_\infty^\ast X \to N$ by letting $(\omega \xi)$ be the $\Phi(\xi)$-left translate of the $\R$ factor in $(X,o) \simeq (S,1)$, and
a horofunction $-t : X \to \R$ from $\omega$ and such that $t(o) = 0$.

\subsection{Quasimetrics and measure on the punctured boundary}
\label{subsec:mm-boundary}

From now on, we make an assumption that $S$ is purely real, i.e.\ $\alpha$ has only positive real eigenvalues.
This is not restrictive as far as large-scale properties are concerned, due to the following fact (see for instance Cornulier, {\cite[Corollary 5.16]{Cornulier_focal}} for a generalized form):

\begin{proposition}
Any Heintze group $S = N \rtimes_\alpha \R$ is quasiisometric to a purely real Heintze group.
\end{proposition}

For any $s \in \R_{>0}$ there is a homomorphism $N \rtimes_{s\alpha} \R \to N \rtimes_\alpha \R$, $(n,t) \mapsto (n,ts)$. Up to rescaling the operation of $\R$, we will work under a normalization assumption:

\begin{definition}
\label{def:normalization-Heintze}
A purely real Heintze group $N \rtimes_{\alpha} \R$ is normalized if the smallest eigenvalue of $\alpha$ is equal to $1$. In this case, the eigenvalues are ordered in increasing order, $ 1 = \lambda_1, \ldots \lambda_r$ and one defines $p = \operatorname{tr} \alpha$.
\end{definition}

\begin{lemma}
\label{lem:dist-homogen}
Choose a horofunction $\beta$ from $\omega$ in $X$, and let $\widehat{\varrho}$ be the visual quasimetric on $\partial_\infty^\ast X$ with parameter $e$ with respect to $\beta$.
$\widehat{\varrho}$ is a $N$-invariant, $S$-equivariant adapted kernel on $\partial_\infty^\ast X$; precisely
\begin{equation}
\label{eq:S-equivariance}
\forall \xi_1, \xi_2 \in \partial_\infty^\ast X,\, \widehat{\varrho}(s\xi_1, s\xi_2) = e^t \widehat{\varrho}(\xi_1, \xi_2),
\end{equation}
if $s=(n,t)$ in the semidirect product decomposition $N \rtimes \R$.
\end{lemma}

\begin{proof}
Applying $s$ is equivalent to adding $t$ to the horofunction $\beta$.
\end{proof}

We refer to $\widehat{\varrho}$ as the homogeneous quasimetric on the punctured boundary; it is indeed a quasimetric (see e.g.\ Buyalo and Schroeder \cite[3.3]{BuyaloSchroeder}). Different, equally natural choices for $\widehat{\varrho}$ are possible; under the constraint of satisfying \eqref{eq:S-equivariance} and a quasiultrametric inequality they would lead to equivalent kernels.
We shall give farther (Lemma \ref{lem:Carnot-type}) a sufficient condition for $\widehat{\varrho}$ to be equivalent to a true distance.
For the moment however we only draw measure-theoretic conclusions.

By definition, $N$ operates on $\partial_\infty^\ast X$, and then on the space of measures on $\partial_\infty^\ast X$; the invariant subspace is an affine line $\mathcal{L}$, by uniqueness of the Haar measure of $N$ up to scaling. This operation extends to $S \curvearrowright \mathcal{L}$ via its modular function: for any $\mu \in \mathcal{L}$, for any $\widehat{\varrho}$-quasiball $B$,
\begin{equation}
\forall s \in S,\, \mu (sB) = \Delta(s) \mu(B),
\label{eq:scale-hom-volume}
\end{equation}
where $\Delta(s) = \exp(t \cdot \operatorname{tr} \alpha) = e^{pt}$ if $s = (n,t)$, and we recall that $p$ is the trace of $\alpha$.

\subsection{Horizontal lines and horizontal curves}
In the tangent space of $\partial_\infty^\ast X$, the distribution $\Phi^\ast \mathfrak{n}_1(\alpha)$ does not depend on the chart $\Phi : \partial_\infty^\ast X \to N$. We refer to it as the horizontal distribution, and denote it by $\tau$ (not forgetting the left action of $N$).
For any $N$-invariant section $L$ of $\mathbb{P}(\tau)$, denote by $\Gamma_L$ the family of horizontal $L$-lines in $\partial_\infty^\ast X$, that is, smooth horizontal curves $\gamma$ tangent to $L$.
After picking arbitrarily a bi-invariant volume form $\omega$ of $N$ and a nonzero $v \in \mathfrak{n}_1(\alpha)$ such that $\Phi^\ast v$ represents $L$ one defines a measure $\rho$ on $\Gamma_L$ by pushing down the interior product $v \rfloor \omega$ to $N / \lbrace e^{tv} \rbrace \simeq \Gamma_L$.

\begin{lemma}
\label{lem:meas-ball-by-counting-lines}
Let $L$ be as above, and let $\mu$ be a $N$-invariant measure on $\partial_\infty^\ast X$.
Then for any $Q \in \R_{>0}$ there exists $c \in \R_{>0}$ (depending on $\mu$, $Q$ and $L$) such that for any $\widehat{\varrho}$-quasiball $B$,
\begin{equation}
\label{eq:condtion-technique}
\rho \left\{ \gamma \in \Gamma_L : \gamma \cap B \neq \emptyset \right\} = c \mu \left( Q^{-1} B \right) ^{(p-1)/p}.
\end{equation}
\end{lemma}

\begin{proof}
$S$ operates simply transitively on the space of $\widehat{\varrho}$-quasiballs, while $N$ operates simply transitively on their centers preserving radii, and $\theta : B \mapsto \left\{ \gamma : \gamma \cap B \neq \emptyset \right\}$ defines a $S$-equivariant map.
Hence it suffices to show \eqref{eq:condtion-technique} for a one-parameter family of balls $\lbrace e^tB \rbrace_{t \in \R}$.
Let $v \in \mathfrak{n}_1(\alpha)$ be a nonzero vector such that $\left[ \Phi^\ast v \right] = L$. Since $v \in \mathfrak{n}_1(\alpha)$, the linear map $\alpha$ operates on $\mathfrak{n}/(\R v)$ with trace $p-1$, and $\rho \left\{ \gamma \in \Gamma_L : \gamma \cap e^tB \neq \emptyset \right\}$ is proportional to $e^{t(p-1)}$. On the other hand, by Lemma \ref{lem:dist-homogen} and $\eqref{eq:scale-hom-volume}$, $\mu(Q^{-1}tB)$ is proportional to $e^{tp}$, hence $\mu(Q^{-1}tB)^{(p-1)/p}$ is proportional to $e^{t(p-1)}$ as well.
\end{proof}

\begin{lemma}
\label{lem:Carnot-type}
Assume that $S$ is normalized (Definition \ref{def:normalization-Heintze}) and that the operation of the derivation $\alpha$ on $\mathfrak{n}^{\mathrm{ab}}$ is scalar, hence the identity. Then
\begin{enumerate}
\item $(N,\alpha)$ is a Carnot graded group, i.e.
\begin{enumerate}
\item $\mathfrak{n}$ admits a grading $(\mathfrak{n}_i)$ by $\Z_{>0}$ such that $\mathfrak{n}_i = \ker (\alpha - i)$
\item $\mathfrak{n}$ is generated by $\mathfrak{n}_1$.
\label{iitem:accessibility}
\end{enumerate}
\label{item:N-is-Carnot}
\item Let $\Vert \cdot \Vert$ be a norm on $\mathfrak{n}_1$, and let $\Phi : \partial_\infty^\ast X \to N$ be a chart. Then $\widehat{\varrho}$ is equivalent to a subRiemannian Carnot-Caratheodory metric
\[ d_{\mathrm{CC}}(\xi_1, \xi_2) = \inf \left\{ \ell(\gamma) : \gamma \in \Gamma(n,n') \right\}, \]
where $\Gamma(\xi_1, \xi_2)$ denotes the space of absolutely continuous curves $[0,1] \to \partial_\infty^\ast X$ between $\xi_1$ and $\xi_2$ with derivative almost everywhere in the horizontal distribution $\tau$, and $\ell(\gamma) = \int_{[0,1]} \Vert \Phi_\ast \gamma' \Vert$ is the length of $\gamma$.
\label{item:equiv-to-CC}
\end{enumerate}
In this case, $X$ is said to be of Carnot type (following Cornulier's terminology).
\end{lemma}

If $X$ has Carnot type, condition \eqref{iitem:accessibility} ensures that $\Gamma(\xi_1, \xi_2)$ is never empty and $d_{\mathrm{CC}}$ takes finite values.

\begin{proof}
See the survey of Cornulier \cite[2.G.1 and 2.G.2]{Cornulier_qihlc} for \eqref{item:N-is-Carnot}. Further, $s = (n,t) \in S$ acts on the space of horizontal curves sending $\Gamma(\xi_1, \xi_2)$ on $\Gamma(s\xi_1, s\xi_2)$ and multiplying lengths by $e^t$, hence
\begin{equation}
d_{\mathrm{CC}}(s \xi_1, s \xi_2) = e^t d_{\mathrm{CC}}(\xi_1, \xi_2)
\end{equation}
for all $\xi_1, \xi_2 \in \partial_\infty^\ast X$.
Select $\xi \in \partial_\infty^\ast X$. Since $d_{\mathrm{CC}}$ and $\widehat{\varrho}$ are both quasimetrics, they are continuous, hence bounded, over unit quasiballs of each other centered at $\xi$. Finally, $S$ operates transitively on the spaces of quasiballs of $\widehat{\varrho}$ and $d_{\mathrm{CC}}$. Hence $\widehat{\varrho}$ and $d_{\mathrm{CC}}$ are equivalent (the control constants depend on $\Vert \cdot \Vert$).
\end{proof}

\subsection{Volumes of quasiballs and intersecting horizontal lines}

\begin{lemma}
\label{lem:Federer-Besicovitch}
Let $q \in \R_{\geqslant 1}$ be a constant and let $\mathcal{X}$ be a proper metric space. Let $\widehat{\varrho}$ be an equivalent kernel on $\mathcal{X}$ with quasi-ultrametric constant $q$. There exists a constant $Q$ depending on $q$, such that
for any countable covering $\mathcal{B}$ of $\mathcal{X}$ by $\widehat{\varrho}$-quasi-balls,
there exists an extraction $\mathcal{B}'$ of $\mathcal{B}$ whose elements are disjoint and such that $\lbrace Q B \rbrace_{B \in {\mathcal{B'}}}$ is a covering of $X$.
\end{lemma}

\begin{proof}
See A.P. Morse, \cite[Theorem 3.4]{MorseAP}, or Federer \cite[2.8.4-2.8.6]{FedererGMT}.
\end{proof}

In the following, whenever $q \in \R_{\geqslant 1}$ is a constant, $Q$ is another constant depending on $q$ defined by the previous lemma.

\begin{lemma}[adapted from Pansu, {\cite[Lemme 6.3]{Pansu_CCqi}}]
\label{lem:maj-Phi1gamma}
Let $\mathcal{X}$ be a proper metric space, and let $\Gamma$ a measured space of curves on $\mathcal{X}$ (denote its measure by $\rho$).
Let $p \in \R_{> 1}$ and $q \in \R_{\geqslant 1}$ be constants. Let $\widehat{\varrho}$ be a kernel on $\mathcal{X}$, equivalent to the original distance and with a $q$-quasiultrametric inequality. Let $U$ be an open, bounded subset of $\mathcal{X}$, endowed with Borel measures $\mu$ and $\nu$, such that
for any $\widehat{\varrho}$-quasiball $B$ contained in $U$,
\begin{equation}
\label{eq:hypothese}
\rho \left\{ \gamma \in \Gamma : \gamma \cap B \neq \emptyset \right\} \leqslant \mu \left( Q^{-1} B \right) ^{(p-1)/p}.
\tag{H}
\end{equation}
For all $\gamma \in \Gamma$ and for all $r>0$, set
\[
\Phi^1_r (\gamma) = \inf_{\mathcal{F}} \sum_{ B\in \mathcal{F}} \phi(B),
 \]
where $\phi(B) := \nu \left( Q^{-1} B \right)^{1/p}$, the infimum taken over countable coverings $\mathcal{F}$ of $\gamma \cap U$ with balls of radius $r$ exactly, contained in $U$.
Then
\begin{equation}
\int_\Gamma \Phi_r^1(\gamma) d \rho \leqslant \nu(U)^{1/p} \mu(U)^{(p-1)/p}.
\end{equation}
\end{lemma}

For Lemma \ref{lem:maj-Phi1gamma}, Pansu's proof can be reproduced almost verbatim \cite[Lemme 6.3]{Pansu_CCqi}, with the only differences of using Lemma \ref{lem:Federer-Besicovitch} instead of the covering lemma used by Pansu, having $r$ fixed and not going to the limit in the end. The argument is based on the H{\"o}lder inequality; in a more general setting it is aimed at bounding a discretized version of the conformal modulus, and then to obtain lower bounds for the conformal dimension, \cite[§2 and 3]{PansuConf}.

\begin{lemma}[compare {\cite[Proposition 6.5]{Pansu_CCqi}}]
\label{lem:inv-dim}
Let $(N, \alpha)$ and $(N', \alpha')$ be Carnot groups with grading derivations $\alpha$, $\alpha'$, normalized, with positive eigenvalues, of traces $p$ and $p'$. Let $X$ and $X'$ be principal spaces of $N \rtimes_\alpha \R$ and $N'\rtimes_{\alpha'} \R$ respectively, and assume there exists a homeomorphism $\partial_\infty^\ast X \to \partial_\infty^\ast X'$ which is sublinearly quasiM{\"o}bius over every compact subset. Then $p \leqslant p'$.
\end{lemma}

\paragraph{Proof sketch}
Define $\tau$ as $p'/p$ and let $\Gamma_L$ be a family of horizontal lines in the boundary of $X$.
We follow the lines of Pansu \cite[Proposition 6.5]{Pansu_CCqi}, despite loosing strength in the conclusion. Precisely this amounts to comparing two facts:
\begin{enumerate}
\item Without any assumption on $N$ and $\alpha$, for any $\sigma \in (\tau, +\infty)$, the image of almost every horizontal curve $\gamma \in \Gamma_L$ has locally finite $\sigma$-dimensional $\widehat{\varrho}$-Hausdorff measure. Hence almost every curve has $\widehat{\varrho}$-Hausdorff dimension less than $\tau$.
\item
Since $X$ has Carnot type, $\widehat{\varrho}$ is equivalent to the subRiemannian distance $d_{\mathrm{CC}}$ by Lemma \ref{lem:Carnot-type}, hence any nonconstant curve should have $\widehat{\varrho}$-Hausdorff-dimension greater than $1$.
\end{enumerate}
This proves that $\tau \geqslant 1$, i.e. $p \leqslant p'$.

\begin{proof}
Let $U$ be a open, relatively compact subset of $\partial_\infty^\ast X$. Define $U'=\varphi(U)$.
Let $\Gamma_L^U$ be the (non-empty) set  $\left\{ \gamma \cap U : \gamma \in \Gamma_L\right\}$ measured with
\[ (\cap_U)_\star \left( \rho \lfloor \left\{ \gamma \in \Gamma_L : \gamma \cap U \neq \emptyset \right\} \right), \] where $\rho$ has been defined in \ref{subsec:mm-boundary}, and $\cap_U(\gamma) = U \cap \gamma$.  We still denote this measure $\rho$.
Let $\mu$, resp.\ $\mu'$ be a $N$-invariant measure on $\partial_\infty^\ast X$, resp.\ on $\partial_\infty^\ast X'$, restricted to $U$, resp.\ to $U'$.
Define a measure $\nu$ on $U$ as
\[ \nu(B) = \mu'(\varphi(B)) \]
for any Borel subset $B \subset U$.
Let $\widehat{\varrho}$ be the homogeneous quasimetric on $\partial_\infty^\ast X$, let $q$ be its ultrametric constant and define $Q$ accordingly (see Lemma \ref{lem:Federer-Besicovitch}).
Let $r \in \R_{>0}$ be a radius that will be repeatedly assumed as small as needed.
Choose $\gamma \in \Gamma_L^U$, and let $\mathcal{F}$ be any covering of $\gamma$ with quasiballs of the same $\widehat{\varrho}$-radius $r$ (we emphasize that all quasiballs must have radius $r$).
By assumption, the quasiballs $\left\{ \varphi(B), B \in \mathcal{F} \right\}$ cover $\varphi(\gamma)$.
By Theorem \ref{th:boundary-maps-are-sqm} and Proposition \ref{prop:sqM-for-balls}, there exists $v = O(u)$, and if $r$ is small enough, a collection $\mathcal{F'}$ of quasiballs and $\mathcal{F}\to \mathcal{F}',\, B \mapsto B'$ such that
\begin{equation}
\label{eq:inclusion-u-gross-quasisym}
\forall B \in \mathcal{F}, \, B' \subset \varphi \left(Q^{-1} B \right) \subset \varphi (B) \subset Q^{2\lambda} e^{v (-\log r)} B' =: B''.
\end{equation}
Define $\mathcal{F''} = \left\{ B'' \right\}$ together with a map $\mathcal{F} \to \mathcal{F''}, B \mapsto B''$. This is a quasiball covering of $\varphi(\gamma)$.

Next, define a gauge function $\phi(B) := \nu \left( Q^{-1} B \right)^{1/p} = \mu' \left( \varphi \left( Q^{-1} B \right) \right)^{1/p}$. By \eqref{eq:scale-hom-volume} there exists a constant $c_0 \in \R_{>0}$, not depending on $r$ and such that
\begin{align}
\phi(B) & \geqslant c_0^{\tau} \mathrm{diam}(B')^{\tau} \notag \\
& \overset{\eqref{eq:inclusion-u-gross-quasisym}}{\geqslant} \left( \frac{c_0}{Q^{2\lambda} e^{v(-\log r)}} \right)^{\tau} \mathrm{diam}(B'')^{\tau}. \label{eq:minoration-de-phiB}
\end{align}
Define $r'' = r^{1/(2 \lambda)}$. Using Cornulier's theorem \ref{thm:Cornulier-Holder}, if $r$ is small enough, then
\begin{align*}
\forall B \in \mathcal{F}, \, \operatorname{diam} B''  \leqslant e^{v(- \log r)} Q^{2 \lambda} \operatorname{diam} B' & \leqslant  e^{v(- \log r)} Q^{2 \lambda} \operatorname{diam} \varphi(B) \\
& \leqslant  e^{v(- \log r)} Q^{2 \lambda} r^{2/(3\lambda)} \\
& \leqslant r'',
\end{align*}
where we used $v(s) \ll s$ and took $r$ small enough in the last line.
On the other hand, using \eqref{eq:anti-Holder} from the proof of Proposition \ref{prop:almost-Holder}, one obtains a reverse inequality:
\begin{equation}
\forall B'' \in \mathcal{F}, \, \log \operatorname{diam} B'' \geqslant 2 \lambda \log r = 4 \lambda^2 \log r''.
\label{eq:reverse-Holder-in-quasiball-covering}
\end{equation}
One can rewrite $Q^{2\lambda} e^{v(- \log r)}$ as $e^{w(- \log r'')}$ with $w = O(u)$. Taking logarithms in \eqref{eq:minoration-de-phiB},
\begin{align*}
\log \phi(B) & \geqslant \tau \log c_0 - w(-\log r'') + \tau \log \operatorname{diam} B'' \\
& \underset{\eqref{eq:reverse-Holder-in-quasiball-covering}}{\geqslant} \tau \log c_0 - w\left( - \frac{1}{4 \lambda^2} \log \operatorname{diam} B'' \right) + \tau \log \operatorname{diam} B''.
\end{align*}
The function $w$ is strictly sublinear, so for any $\sigma \in (\tau, + \infty)$, there is $r_\sigma \in \R_{>0}$ such that
\begin{equation}
\forall r \in (0, r_\sigma), \forall B \in \mathcal{F}, \, \phi(B) \geqslant (r'')^\sigma \geqslant \left( \operatorname{diam} B'' \right)^\sigma.
\end{equation}
Recall that for all $\mathcal{F}$ the quasiballs $B''\in \mathcal{F}''$ cover $\varphi(\gamma)$. By definition of the $\widehat{\varrho}$-Hausdorff premeasure at scale $r''$,
\begin{equation}
\label{eq:haus-lower-than-confmod}
\Phi_r^1(\gamma) = \inf_{\mathcal{F}} \sum_{B \in \mathcal{F}} \phi(B) \geqslant \sum_{B''} \operatorname{diam}(B'')^\gamma \geqslant \mathscr{H}^\sigma_{r''} \varphi(\gamma).
\end{equation}
By Lemma \ref{lem:meas-ball-by-counting-lines}, the hypothesis \eqref{eq:hypothese} of Lemma \ref{lem:maj-Phi1gamma} is fullfilled. Hence, for all $r \in (0, r_\sigma)$,
\[ \int_{\Gamma_L^U} \Phi_r^1(\gamma) d \rho \leqslant \nu(U)^{1/p} \mu(U)^{(p-1)/p}. \]
By monotone convergence, for $\rho$-almost every $\gamma$, $\sup_{r} \Phi^1_r(\gamma)$ is finite, and then by \eqref{eq:haus-lower-than-confmod}, $\mathscr{H}^\sigma \varphi(\gamma)$ is finite. Considering this fact for all terms of a decreasing sequence $\lbrace \sigma_j \rbrace$ converging to $\tau$, one deduces that, still for $\rho$-almost every $\gamma$,
\begin{equation}
\label{eq:tau-grand}
\dim_{\mathrm{H}} \varphi(\gamma) \leqslant \inf_j \sigma_j = \tau.
\end{equation}
Finally, $X$ has been assumed of Carnot type, hence $\widehat{\varrho}$ is equivalent to the Carnot-Caratheodory metric $d_{\mathrm{CC}}$ by Lemma \ref{lem:Carnot-type}. By the triangle inequality, the $1$-dimensional $d_{\mathrm{CC}}$-Hausdorff measure of any nonconstant curve is nonzero, in particular its $d_\mathrm{CC}$-Hausdorff dimension must be greater than $1$.
This dimension does not change when replacing $d_{\mathrm{CC}}$ with the equivalent quasimetric $\widehat{\varrho}$.
By \eqref{eq:tau-grand} there exists $\gamma \in \Gamma_L^U$ such that $1 \leqslant \dim_{\mathrm{H}} \varphi(\gamma) \leqslant \tau$. Hence $1 \leqslant \tau$.
\end{proof}

Lemma \ref{lem:inv-dim} is applied to show that $p$ is a SBE invariant between spaces of Carnot type. In fact this can be made slightly more general:

\begin{proposition}
\label{prop:almost-Carnot}
Let $X_1$ and $X_2$ be principal spaces of purely real, normalized Heintze groups $N_1 \rtimes_{\alpha_1} \R$ and $N_2 \rtimes_{\alpha_2} \R$. Assume that for all $i \in \lbrace 1, 2 \rbrace$ the operation defined by $\alpha_i$ on $\mathfrak{n}_i^{\mathrm{ab}}$ is unipotent.
If there exists a large scale sublinearly biLipschitz equivalence between $X_1$ and $X_2$, then
Then if $\operatorname{tr}(\alpha_1) = \operatorname{tr}(\alpha_2)$.
\end{proposition}

\begin{proof}
For every $i \in \lbrace 1, 2 \rbrace$, decompose $\alpha_i$ into $\alpha_i^\sigma + \alpha_i^\nu$, where $\alpha_i^\sigma$ is semi-simple and $\alpha_i^\nu$ is nilpotent. By hypothesis, $\mathfrak{n}_i^\sigma$ operates as the identity on $\mathfrak{n}_i^{\mathrm{ab}}$, hence $N_i$ are Carnot gradable groups, and $\alpha_i$ are grading derivations of their Lie algebra.
A particular instance of a theorem by Cornulier implies that there exists a $O(\log)$-sublinearly biLipschitz equivalence $\psi : N \rtimes_{\alpha^\sigma_i} \R \to N \rtimes_{\alpha_i} \R$ (see \cite[Theorem 4.4]{Cornulier_Coneq}: in our very special case the exponential radical is $N$, and the Cartan subgroup is $\R$). Then, $N \rtimes_{\alpha^\sigma_i} \R$ are of Carnot type, so by Theorem \ref{th:boundary-maps-are-sqm} and Lemma \ref{lem:inv-dim}, a necessary condition for existence of a sublinearly biLipschitz equivalence between them is that $\operatorname{tr} \alpha_1^\sigma = \operatorname{tr} \alpha_2^\sigma$. Finally, this condition can be rewritten without reference to the $\alpha_i^\sigma$ since $\operatorname{tr} \alpha_i = \operatorname{tr} \alpha_i^\sigma$ for all $i \in \lbrace 1, 2 \rbrace$.
\end{proof}

Note that if sublinearly biLipschitz equivalences are replaced by quasiisometries in the last statement, known invariants are much finer than the trace. In this direction, M.\ Carrasco Piaggio and E.\ Sequeira obtained that for normalized purely real Heintze groups, resp.\ for normalized purely real Heintze groups with a fixed Heisenberg group as exponential radical $N$, the characteristic polynomial, resp.\ the full Jordan form of $\alpha$, are quasiisometric invariants \cite[Theorem 1.1, resp.\ Theorem 1.3]{CarrascoSequeira}.

\subsection{Proof of Theorem \ref{th:SBE-symm-sapce}}
Notation is as before. When $X$ is a hyperbolic symmetric space, several restrictions appear (see Heintze, \cite[Proposition 4 and Corollary]{Heintze}):
\begin{enumerate}
\item $X$ is of Carnot type.
\label{item:sym-is-of-Carnot-type}
\item The Lie algebra $\mathfrak{n}$ is two-step, $\mathfrak{n} = \mathfrak{n}_1 \oplus \mathfrak{n}_2$ where $\mathfrak{n}_2$ is possibly zero.
\label{item:sym-is-two-steps}
\item  \label{item:division-algebra}
Save for one case, namely the Cayley hyperbolic plane, there exists a division algebra structure on $\R \oplus \mathfrak{n}_2$, and $\mathfrak{n}_1$ is a module over this division algebra. The structure of $\mathfrak{n}$ is completely determined by these data.
\end{enumerate}
The Frobenius classification of division algebras over $\R$ reduces considerably the list of candidates thanks to \eqref{item:division-algebra}: the two relevant parameters are the division algebra $\mathbf{K} \in \lbrace \R, \mathbf{C}, \mathbf{H} \rbrace$ and a positive integer, the rank of $\mathfrak{n}_1$ over $\mathbf{K}$. The Cayley hyperbolic plane fits in this list, setting $\mathfrak{n}_1 = \mathbf{O}$.
The homogeneous dimension is computed as
\begin{align*}
\operatorname{tr}(\alpha) = \dim \mathfrak{n}_1 + 2 \dim \mathfrak{n}_2 & =
\dim \mathfrak{n} + \dim \mathfrak{n}_2 \\ & =
 \dim X - 1 + \dim \mathfrak{Im}(\mathbf{K}),
\end{align*}
and $\mathbf{K}$ is completely determined by $\dim \mathfrak{Im} (\mathbf{K}) \in \left\{ 0,1,3,7 \right\}$.
By Theorem \ref{thm:Cornulier-Holder} and Proposition \ref{prop:almost-Carnot}, $\mathbf{K}$ is a SBE invariant, as
\[ \dim \mathfrak{Im} (\mathbf{K}) = \dim_\mathrm{H} (\partial^\ast_\infty X, \widehat{\varrho}) - \dim \partial_\infty X. \]
The rank $n$ of $\mathfrak{n}_1$ over $\mathbf{K}$ is a SBE invariant as well, since it can be computed by the formula
\[ (1 + n) \dim_\R \mathbf{K} = 1 + \dim \partial_\infty X. \]

\bibliographystyle{amsplain}

\end{document}